\documentclass[10pt]{amsart}
\usepackage{amssymb}
\usepackage{amscd}

\numberwithin{equation}{section}

\newcounter{TmpEnumi}

\def\today{\number\day\space\ifcase\month\or   January\or February\or
   March\or April\or May\or June\or   July\or August\or September\or
   October\or November\or December\fi\   \number\year}

\theoremstyle{definition}
\newtheorem{thm}{Theorem}[section]
\newtheorem{lem}[thm]{Lemma}
\newtheorem{prp}[thm]{Proposition}
\newtheorem{dfn}[thm]{Definition}
\newtheorem{cor}[thm]{Corollary}

\newtheorem{rmk}[thm]{Remark}
\newtheorem{ntn}[thm]{Notation}
\newtheorem{exa}[thm]{Example}

\newcommand{\beq}{\begin{equation}}
\newcommand{\eeq}{\end{equation}}
\newcommand{\beqr}{\begin{eqnarray*}}
\newcommand{\eeqr}{\end{eqnarray*}}
\newcommand{\bal}{\begin{align*}}
\newcommand{\eal}{\end{align*}}
\newcommand{\bei}{\begin{itemize}}
\newcommand{\eei}{\end{itemize}}
\newcommand{\limi}[1]{\lim_{{#1} \to \infty}}

\newcommand{\af}{\alpha}
\newcommand{\bt}{\beta}
\newcommand{\gm}{\gamma}
\newcommand{\dt}{\delta}
\newcommand{\ep}{\varepsilon}
\newcommand{\zt}{\zeta}
\newcommand{\et}{\eta}
\newcommand{\ch}{\chi}
\newcommand{\io}{\iota}

\newcommand{\ld}{\lambda}
\newcommand{\sm}{\sigma}

\newcommand{\ph}{\varphi}
\newcommand{\ps}{\psi}
\newcommand{\rh}{\rho}
\newcommand{\om}{\omega}
\newcommand{\ta}{\tau}

\newcommand{\Dt}{\Delta}

\newcommand{\Q}{{\mathbb{Q}}}
\newcommand{\Z}{{\mathbb{Z}}}
\newcommand{\R}{{\mathbb{R}}}
\newcommand{\C}{{\mathbb{C}}}
\newcommand{\N}{{\mathbb{Z}}_{> 0}}
\newcommand{\Nz}{{\mathbb{Z}}_{\geq 0}}

\newcommand{\id}{{\mathrm{id}}}

\newcommand{\diag}{{\mathrm{diag}}}
\newcommand{\supp}{{\mathrm{supp}}}

\newcommand{\spn}{{\mathrm{span}}}

\newcommand{\OA}[1]{{\mathcal{O}}_{#1}}

\newcommand{\andeqn}{\,\,\,\,\,\, {\mbox{and}} \,\,\,\,\,\,}

\newcommand{\ts}[1]{{\textstyle{#1}}}
\newcommand{\ds}[1]{{\displaystyle{#1}}}
\newcommand{\ssum}[1]{{\ts{ {\ds{\sum}}_{#1} }}}
\newcommand{\sssum}[2]{{\ts{ {\ds{\sum}}_{#1}^{#2} }}}

\newcommand{\wolog}{without loss of generality}
\newcommand{\Wolog}{Without loss of generality}

\newcommand{\tfae}{the following are equivalent}
\newcommand{\ifo}{if and only if}
\newcommand{\wrt}{with respect to}

\newcommand{\ca}{C*-algebra}

\newcommand{\hm}{homomorphism}

\newcommand{\ct}{continuous}
\newcommand{\cfn}{continuous function}

\newcommand{\rpn}{representation}
\newcommand{\cog}{contractive on generators}
\newcommand{\fis}{forward isometric}
\newcommand{\sfi}{strongly forward isometric}
\newcommand{\sga}{$\sm$-algebra}
\newcommand{\mb}{measurable}
\newcommand{\msp}{measure space}
\newcommand{\sft}{$\sm$-finite}
\newcommand{\sfm}{$\sm$-finite measure space}
\newcommand{\shm}{$\sm$-\hm}
\newcommand{\mst}{measurable set transformation}
\newcommand{\XBM}{(X, {\mathcal{B}}, \mu)}
\newcommand{\YCN}{(Y, {\mathcal{C}}, \nu)}
\newcommand{\cB}{{\mathcal{B}}}
\newcommand{\cC}{{\mathcal{C}}}
\newcommand{\LLp}{L (L^p (X, \mu))}

\newcommand{\OP}[2]{{\mathcal{O}}_{#1}^{#2}}
\newcommand{\MP}[2]{M_{#1}^{#2}}

\renewcommand{\S}{\subset}
\newcommand{\ov}{\overline}
\newcommand{\SM}{\setminus}
\newcommand{\I}{\infty}
\newcommand{\E}{\varnothing}

\title[Cuntz algebras on $L^p$~spaces]{Analogs
    of Cuntz algebras on $L^p$~spaces}

\author{N.~Christopher Phillips}

\date{19~January 2012}

\address{Department of Mathematics, University  of Oregon,
       Eugene OR 97403-1222, USA,
    and
       Research Institute for Mathematical Sciences,
  Kyoto University,
  Kitashirakawa-Oiwakecho, Sakyo-ku, Kyoto
  606-8502, Japan.}

\email[]{ncp@darkwing.uoregon.edu}

\thanks{This material is based upon work supported by the
  US National Science Foundation under Grants
  DMS-0701076 and DMS-1101742.
  It was also partially supported by the Centre de Recerca
  Matem\`{a}tica (Barcelona) through a research visit conducted
  during 2011.}

\begin{document}

\begin{abstract}
For $d = 2, 3, \ldots$ and $p \in [1, \infty),$
we define a class of representations $\rho$
of the Leavitt algebra $L_d$ on spaces of the form $L^p (X, \mu),$
which we call the spatial representations.
We prove that for fixed $d$ and~$p,$
the Banach algebra ${{\mathcal{O}}_{d}^{p}} = {\overline{\rho (L_d)}}$
is the same for all spatial representations~$\rho.$
When $p = 2,$
we recover the usual Cuntz algebra~${\mathcal{O}}_{d}.$
We give a number of equivalent conditions for a
representation to be spatial.
We show that for distinct $p_1, p_2 \in [1, \I)$
and arbitrary $d_1, d_2 \in \{ 2, 3, \ldots \},$
there is no nonzero continuous homomorphism
from ${\mathcal{O}}_{d_1}^{p_1}$ to ${\mathcal{O}}_{d_2}^{p_2}.$
\end{abstract}

\maketitle

\indent
The algebras that we call the Leavitt algebras $L_d$
(see Definition~\ref{D:Leavitt}) are a special case of algebras
introduced in characteristic~$2$ by Leavitt~\cite{Lv},
and generalized to arbitrary ground fields in~\cite{Lv2}.
The Cuntz algebra $\OA{d},$
introduced in~\cite{Cu1},
can be defined as the norm closure of the range of
a *-representation of $L_d$ on a Hilbert space.
(Cuntz
did not define $\OA{d}$ this way.)
The Cuntz algebras have turned out to be one of the most
fundamental families of examples of \ca s.
Leavitt algebras were long obscure,
but they have recently attracted renewed attention.

In this paper we study the analogs of Cuntz algebras
on $L^p$~spaces.
That is, we consider the norm closure of the range of
a representation of $L_d$ on a space of the form $L^p (X, \mu).$
It turns out that there is a rich theory of such algebras,
of which we exhibit the beginning in this paper.

Our main results
are as follows.
There are many possible $L^p$~analogs of Cuntz algebras
(although we mostly do not know for sure that they really are
essentially different),
but there is a natural class of such algebras,
namely those that come from what we call spatial \rpn s
(Definition~\ref{D:SpatialRep}(\ref{D:SpatialRep-Sp})).
Spatial \rpn s are those for which the standard generators
form an $L^p$~analog of what is called a row isometry
in multivariable operator theory on Hilbert space.
(The survey article~\cite{Dv}
emphasizes the more general row contractions on Hilbert spaces,
but row isometries also play a significant role.
See especially Section~6.2 of~\cite{Dv}.)
We give a number of
rather different equivalent conditions for a \rpn\  to be
spatial---further evidence that this is a natural class.
For fixed $p \in [1, \I) \SM \{ 2 \},$
we prove a uniqueness theorem:
the Banach algebras coming from any two spatial \rpn s
are isometrically isomorphic via an isomorphism which
sends the standard generators to the standard generators.
We call the Banach algebra obtained this way~$\OP{d}{p}.$
The usual Cuntz algebra is~$\OP{d}{2}.$
We further obtain a strong dependence on~$p.$
Specifically, for $p_1 \neq p_2$
and any $d_1, d_2 \in \{ 2, 3, \ldots \},$
there is no nonzero \ct\  \hm\  %
from $\OP{d_1}{p_1}$ to~$\OP{d_2}{p_2}.$

Some of our results are valid,
with the same proofs,
for the infinite Leavitt algebra~$L_{\infty}$
(which gives $L^p$~analogs of~$\OA{\infty}$)
and for the Cohn algebras
(which give $L^p$~analogs of Cuntz's algebras~$E_d$).
In such cases, we include the corresponding results.
However, in many cases, the results,
or at least the proofs,
must be modified.
We do not go in that direction in this paper.

The methods here have little in common with \ca\  methods.
Indeed, the results on spatial \rpn s have no analog for \ca s,
and the result on nonexistence of nonzero \ct\  \hm s
does not make sense if one only considers \ca s.
Uniqueness of~$\OP{d}{p}$ is of course true when $p = 2,$
but, as far as we can tell, our proof for $p \neq 2$
does not work when $p = 2.$

This paper is organized as follows.
In Section~\ref{Sec:LCAlgs},
we define Leavitt algebras
and Cohn algebras,
and give some basic facts about them which will be needed in the rest
of the paper.
In Section~\ref{Sec:Reps},
we discuss representations on Banach spaces.
We define several natural conditions on \rpn s
(weaker than being spatial).
We describe ways to get new \rpn s from old ones,
some of which work in general
and some of which are special to Leavitt and Cohn algebras.
Section~\ref{Sec:ExsOrReps} contains a large collection
of examples of \rpn s on $L^p$~spaces.

In Section~\ref{Sec:SPI} we develop the theory
of (semi)spatial partial isometries on $L^p$~spaces
associated to \sfm s.
The results here are the basic technical tools
needed to prove our main results.
Roughly speaking,
a semispatial partial isometry from $L^p (X, \mu)$
to $L^p (Y, \nu)$ comes from a map from a subset of~$Y$
to a subset of~$X.$
Lamperti's Theorem~\cite{Lp},
which plays a key role,
asserts that for $p \in (0, \I) \SM \{ 2 \},$
every isometry from $L^p (X, \mu)$
to $L^p (Y, \nu)$ is semispatial.

It is unfortunately not really true that
semispatial partial isometries from $L^p (X, \mu)$
to $L^p (Y, \nu)$ come from point maps.
Instead, they come from suitable \hm s
of \sga s.
For our theory of spatial partial isometries,
we need a much more extensive theory of these
than we have been able to find in the literature.
In Section~\ref{Sec:Lp1} we recall some standard facts
about Boolean $\sm$-algebras,
and in Section~\ref{Sec:MST} we discuss the maps on
functions and measures induced by a suitable \hm\  of
the Boolean $\sm$-algebras of measurable sets mod null sets.

Sections \ref{Sec:SpatialReps}, \ref{Sec:SpatialIsSame},
and~\ref{Sec:NonIso}
contain our main results.
In Section~\ref{Sec:SpatialReps},
we give equivalent conditions for \rpn s of Leavitt algebras
on $L^p$~spaces
to be spatial.
Along the way, we define spatial \rpn s
of the algebra~$M_d$ of $d \times d$ matrices,
and we give a number of equivalent conditions
for a \rpn\  of~$M_d$ to be spatial.
In Section~\ref{Sec:SpatialIsSame},
we prove that and two spatial \rpn s of the
Leavitt algebra~$L_d$ on $L^p$~spaces
give the same norm on~$L_d,$
and thus lead to isometrically isomorphic Banach algebras.
In Section~\ref{Sec:NonIso},
we prove the nonexistence of nonzero \ct\  \hm s
between the resulting algebras for different values of~$p.$

In~\cite{Ph6}, we will show that~$\OP{d}{p}$
is an amenable purely infinite simple Banach algebra,
and in~\cite{Ph7},
we will show that its topological K-theory is the same
as for the ordinary Cuntz algebra~$\OA{d}.$
The methods in these papers are much closer to \ca\  methods.

Scalars will always be~$\C.$
Much of what we do also works for real scalars.
We use complex scalars for our proof of the equivalence
of several of the conditions in Theorem~\ref{T:SpatialRepsMd}
for a \rpn\  of~$M_d$ to be spatial.
We do not know whether complex scalars are really necessary.

We have tried to make this paper accessible to operator algebraists
who are not familiar with operators on spaces of the form
$L^p (X, \mu).$

We are grateful to
Joe Diestel,
Ilijas Farah,
Coenraad Labuschagne,
and
Volker Runde
for useful email discussions and for providing references.
We are especially grateful to Bill Johnson
for extensive discussions about Banach spaces,
and to Guillermo Corti\~{n}as
and Mar\'{\i}a Eugenia Rodr\'{\i}guez,
who carefully read an early draft
and whose comments led to numerous
corrections and improvements.
Some of this work was carried out during a visit to
the Instituto Superior T\'{e}cnico,
Universidade T\'{e}cnica de Lisboa,
and during an extended research visit to
the Centre de Recerca Matem\`{a}tica (Barcelona).
I also thank the Research Institute for Mathematical Sciences
of Kyoto University for its support through a visiting professorship.
I am grateful to all these institutions for their hospitality.

\section{Leavitt and Cohn algebras}\label{Sec:LCAlgs}

\indent
In this section,
we define Leavitt algebras and some of their relatives.
We describe a grading,
a linear involution,
and a conjugate linear involution.
We give some useful computational lemmas.

\begin{dfn}\label{D:Leavitt}
Let $d \in \{ 2, 3, 4, \ldots \}.$
We define the {\emph{Leavitt algebra}} $L_d$
to be the universal complex associative algebra on
generators $s_1, s_2, \ldots, s_d, t_1, t_2, \ldots, t_d$
satisfying the relations
\begin{equation}\label{Eq:Leavitt1}
t_j s_j = 1
\,\,\,\,\,\,\,\,\,\,\,\,
{\mbox{for $j \in \{ 1, 2, \ldots, d \},$}}
\end{equation}
\begin{equation}\label{Eq:Leavitt2}
t_j s_k = 0
\,\,\,\,\,\,\,\,\,\,\,\,
{\mbox{for $j, k \in \{ 1, 2, \ldots, d \}$ with $j \neq k,$}}
\end{equation}
and
\begin{equation}\label{Eq:Leavitt3}
\sum_{j = 1}^d s_j t_j = 1.
\end{equation}
\end{dfn}

These algebras
were introduced in Section~3 of~\cite{Lv}
(except that the base field there is $\Z / 2 \Z$),
and they are simple
(Theorem~2 of~\cite{Lv2}, with an arbitrary choice of the field).

\begin{dfn}\label{D:Cohn}
Let $d \in \{ 2, 3, 4, \ldots \}.$
We define the {\emph{Cohn algebra}} $C_d$
to be the universal complex associative algebra on
generators $s_1, s_2, \ldots, s_d, t_1, t_2, \ldots, t_d$
satisfying the relations
(\ref{Eq:Leavitt1}) and~(\ref{Eq:Leavitt2})
(but not~(\ref{Eq:Leavitt3})).
\end{dfn}

These algebras are a special case of algebras
introduced at the beginning of Section~5 of~\cite{Coh}.
What we have called $C_d$ is called $U_{1, d}$ in~\cite{Coh},
and also in~\cite{Coh} the field is allowed to be arbitrary.
Our notation, and the name ``Cohn algebra'',
follow Definition~1.1 of~\cite{AF},
except that we specifically take the field to be~$\C$
and suppress it in the notation.

\begin{dfn}\label{D:LInfty}
Let $d \in \{ 2, 3, 4, \ldots \}.$
We define the {\emph{(infinite) Leavitt algebra}} $L_{\I}$
to be the universal complex associative algebra on
generators $s_1, s_2, \ldots, t_1, t_2, \ldots$
satisfying the relations
\begin{equation}\label{Eq:ILeavitt1}
t_j s_j = 1
\,\,\,\,\,\,\,\,\,\,\,\,
{\mbox{for $j \in \N$}}
\end{equation}
and
\begin{equation}\label{Eq:ILeavitt2}
t_j s_k = 0
\,\,\,\,\,\,\,\,\,\,\,\,
{\mbox{for $j, k \in \N$ with $j \neq k.$}}
\end{equation}
When it is necessary to distinguish the generators of $L_{\I}$
from those of $L_d$ and~$C_d,$
we denote them by
$s_1^{(\I)}, s_2^{(\I)}, \ldots,
   t_1^{(\I)}, t_2^{(\I)}, \ldots.$
\end{dfn}

This algebra is simple, by Example 3.1(ii) of~\cite{AA2}.

\begin{rmk}\label{R-GraphAlg}
The algebras $L_d,$ $C_d,$ and $L_{\I}$ are all examples of
Leavitt path algebras.
For $L_d$ see Example 1.4(iii) of~\cite{AA1},
for $C_d$ see Section~1.5 of~\cite{AAS},
and for $L_{\I}$ see Example 3.1(ii) of~\cite{AA2}.
(Warning:
There are two possible conventions for the choice of the direction
of the arrows in the graph, and both are in common use.)
\end{rmk}

\begin{lem}\label{L-CdLdPlus1}
Let $C_d$ be as in Definition~\ref{D:Cohn}
and let $L_d$ be as in Definition~\ref{D:Leavitt},
with the generators named as there
(using the same names in both kinds of algebras).
For $d_1, d_2 \in \{ 2, 3, 4, \ldots, \I \}$ with $d_1 < d_2,$
there is a unique \hm\  $\io_{d_1, d_2} \colon C_{d_1} \to L_{d_2}$
such that $\io_{d_1, d_2} (s_j) = s_j$
and $\io_{d_1, d_2} (t_j) = t_j$ for $j \in \{ 1, 2, \ldots, d_1 \}.$
Moreover, $\io_{d_1, d_2}$~is injective.
\end{lem}

\begin{proof}
Existence and uniqueness of~$\io_{d_1, d_2}$ are immediate from
the definitions of the algebras as universal algebras on
generators and relations.

We prove injectivity.
We presume that there is a purely algebraic proof,
but one can easily see this by comparing with the \ca s,
following Remark~\ref{R:HCuntz} below.
Let $E_{d_1}$ be the extended Cuntz algebra,
as in Remark~\ref{R:HCuntz}.
There is a commutative diagram
\[
\begin{CD}
C_{d_1} @>{\io_{d_1, d_2}}>> L_{d_2}   \\
@VVV                         @VVV            \\
E_{d_1} @>>>                 \OA{d_2}.
\end{CD}
\]
The left vertical map is injective
by Theorem~7.3 of~\cite{Tm} and Remark~\ref{R-GraphAlg},
and the bottom horizontal map is well known to be injective.
Therefore $\io_{d_1, d_2}$ is injective.
\end{proof}

\begin{lem}\label{L:Invol}
Let $A$ be any of $L_d$ (Definition~\ref{D:Leavitt}),
$C_d$ (Definition~\ref{D:Cohn}),
or~$L_{\I}$ (Definition~\ref{D:LInfty}).
\begin{enumerate}
\item\label{L:Invol-Star}
There exists a unique conjugate linear antimultiplicative
involution $a \mapsto a^*$ on~$A$
such that $s_j^* = t_j$ and $t_j^* = s_j$ for all~$j.$
\item\label{L:Invol-Prime}
There exists a unique linear antimultiplicative
involution $a \mapsto a'$ on~$A$
such that $s_j' = t_j$ and $t_j' = s_j$ for all~$j.$
\end{enumerate}
\end{lem}

The properties of $a \mapsto a^*$ are just the algebraic
properties of the adjoint of a complex matrix:
\[
(a + b)^* = a^* + b^*,
\,\,\,\,\,\,
(\ld a)^* = {\ov{\ld}} a^*,
\,\,\,\,\,\,
(a b)^* = b^* a^*,
\andeqn
(a^*)^* = a
\]
for all $a, b \in A$ and $\ld \in \C.$
The properties of $a \mapsto a'$ are the same,
except that it is linear: $(\ld a)' = \ld a'$
for all $a \in A$ and $\ld \in \C.$

\begin{proof}[Proof of Lemma~\ref{L:Invol}]
See Remark~3.4 of~\cite{Tm},
where explicit formulas, valid for any graph algebra,
are given, and Remark~\ref{R-GraphAlg}.
Both parts may also be easily obtained using the universal
properties of algebras on generators and relations:
$a \mapsto {\overline{a}}$ is a \hm\  from $A$ to its opposite
algebra,
and $a \mapsto a^*$ is the composition of
$a \mapsto {\overline{a}}$ with a \hm\  from $A$ to its
complex conjugate algebra.
\end{proof}

One can get Lemma~\ref{L:Invol}(\ref{L:Invol-Star})
by using the fact
(Theorem~7.3 of~\cite{Tm} and Remark~\ref{R-GraphAlg})
that there are injective maps from $L_d,$ $C_d,$ and~$L_{\I}$
to \ca s
which preserve the intended involution.
(For $L_d$ and $L_{\I},$
injectivity is automatic because the algebras are simple.
See Remark~\ref{R:HCuntz} for definitions of *-representations
on Hilbert spaces.)

\begin{prp}\label{P:ZGrading}
Let $A$ be any of $L_d$ (Definition~\ref{D:Leavitt}),
$C_d$ (Definition~\ref{D:Cohn}),
or~$L_{\I}$ (Definition~\ref{D:LInfty}).
Then there is a unique $\Z$-grading on~$A$ determined
by
\[
\deg (s_j) = 1
\andeqn
\deg (t_j) = - 1
\]
for all~$j.$
\end{prp}

\begin{proof}
The proof is easy.
(See after Definition~3.12 in~\cite{Tm}.)
\end{proof}

We will need some of the finer algebraic structure of~$L_d,$
and associated notation.
We roughly follow the beginning of Section~1 of~\cite{Cu1},
starting with 1.1 of~\cite{Cu1}.

\begin{ntn}\label{N:Words}
Let $d \in \{ 2, 3, 4, \ldots, \I \},$
and let $n \in \Nz.$
For $d < \I,$ we define $W_n^d = \{1, 2, \ldots, d \}^n,$
and we define $W_n^{\I} = (\N)^n.$
Thus, $W_n^d$ is
the set of all sequences
$\af = \big( \af (1), \af (2), \ldots, \af (n) \big)$
with $\af (l) \in \{1, 2, \ldots, d \}$
(or $\af (l) \in \N$ if $d = \I$) for $l = 1, 2, \ldots, n.$
We set
\[
W_{\I}^d = \coprod_{n = 0}^{\I} W_n^d.
\]
We call the elements of $W_{\I}^d$ {\emph{words}}
(on $\{1, 2, \ldots, d \}$ or $\N$ as appropriate).
If $\af \in W_{\I}^d,$ the {\emph{length}} of~$\af,$
written $l (\af),$
is the unique number $n \in \Nz$ such that $\af \in W_n^d.$
Note that there is a unique word of length zero,
namely the empty word,
which we write as~$\E.$
For $\af \in W_m^d$ and $\bt \in W_n^d,$
we denote by $\af \bt$ the concatenation,
a word in $W_{m + n}^d.$
\end{ntn}

\begin{ntn}\label{N:WordsInGens}
Let $A$ be any of $L_d$ (Definition~\ref{D:Leavitt}),
$C_d$ (Definition~\ref{D:Cohn}),
or~$L_{\I}$ (Definition~\ref{D:LInfty}).
Let $n \in \Nz,$ and
let $\af = \big( \af (1), \af (2), \ldots, \af (n) \big) \in W_n^d.$
If $n \geq 1,$ we define $s_{\af}, t_{\af} \in A$ by
\[
s_{\af} = s_{\af (1)} s_{\af (2)}
     \cdots s_{\af (n - 1)} s_{\af (n)}
\andeqn
t_{\af} = t_{\af (n)} t_{\af (n - 1)}
     \cdots t_{\af (2)} t_{\af (1)}.
\]
We take $s_{\E} = t_{\E} = 1.$
\end{ntn}

For emphasis:
in the definition of $t_{\af},$
we take the generators $t_{\af (l)}$
{\emph{in reverse order}}.
We do this to get convenient formulas
in Lemma~\ref{L:PropOfWords}(\ref{L:PropOfWords-3}).
In particular, when working with Cuntz algebras,
one simply uses $s_j^*$ in place of~$t_j,$
and we want to have $s_{\af}^* = t_{\af}.$

\begin{lem}\label{L:PropOfWords}
Let the notation be as in Notation~\ref{N:Words}
and Notation~\ref{N:WordsInGens}.
\begin{enumerate}
\item\label{L:PropOfWords-4}
Let $\af, \bt \in W_{\I}^d.$
Then $s_{\af \bt} = s_{\af} s_{\bt}$
and $t_{\af \bt} = t_{\bt} t_{\af}.$
\item\label{L:PropOfWords-1}
In the $\Z$-grading on~$A$ of Proposition~\ref{P:ZGrading},
we have $\deg (s_{\af}) = l (\af)$ and $\deg (t_{\af}) = - l (\af)$
for all $\af \in W_{\I}^d.$
\item\label{L:PropOfWords-3}
Let $\af \in W_{\I}^d.$
Then the involutions of Lemma~\ref{L:Invol} satisfy
$s_{\af}' = s_{\af}^* = t_{\af}$
and $t_{\af}' = t_{\af}^* = s_{\af}.$
\item\label{L:PropOfWords-2}
Let
\[
a_1, a_2, \ldots, a_n
 \in \{ s_1, s_2, \ldots \} \cup \{ t_1, t_2, \ldots \}.
\]
Suppose $a_1 a_2 \cdots a_n \neq 0.$
Then there exist unique $\af, \bt \in W_{\I}^d$
such that $a_1 a_2 \cdots a_n = s_{\af} t_{\bt}.$
\item\label{L:PropOfWords-X0}
Let $\af, \bt \in W_{\I}^d$ satisfy $l (\af) = l (\bt).$
Then $t_{\bt} s_{\af} = 1$ if $\af = \bt,$
and $t_{\bt} s_{\af} = 0$ otherwise.
\end{enumerate}
\end{lem}

\begin{proof}
Parts (\ref{L:PropOfWords-4}), (\ref{L:PropOfWords-1}),
(\ref{L:PropOfWords-3}),
and~(\ref{L:PropOfWords-X0}) are obvious.
(Part~(\ref{L:PropOfWords-X0}) is also in Lemma 1.2(b) of~\cite{Cu1}.)

Using Part~(\ref{L:PropOfWords-3}),
we see that
Part~(\ref{L:PropOfWords-2}) is Lemma~1.3 of~\cite{Cu1}.
\end{proof}

\begin{lem}\label{L:mSum}
Let $d \in \{ 2, 3, 4, \ldots \},$
let $L_d$ be as in Definition~\ref{D:Leavitt}, and let $m \in \Nz.$
Then the collection $( s_{\af} t_{\bt} )_{\af, \bt \in W_{m}^d}$
is a system of matrix units for a unital subalgebra of $L_d$
isomorphic to $M_{d^m}.$
That is,
identifying $M_{d^m}$ with the linear maps on a vector space
with basis $W_{m}^d,$
with matrix units $e_{\af, \bt}$ for $\af, \bt \in W_{m}^d,$
there is a unique \hm\  %
$\ph_m \colon M_{d^m} \to L_d$
such that $\ph_m (e_{\af, \bt}) = s_{\af} t_{\bt}.$
\end{lem}

\begin{proof}
We prove that $\sum_{\af \in W_{m}^d} s_{\af} t_{\af} = 1,$
by induction on~$m.$
The case $m = 1$ is relation~(\ref{Eq:Leavitt3})
in Definition~\ref{D:Leavitt}.
Assuming the result holds for~$m,$
use this and the case $m = 1$ at the last step to get
\[
\sum_{\af \in W_{m + 1}^d} s_{\af} t_{\af}
  = \sum_{\bt \in W_{m}^d} \sum_{j = 1}^d s_{\bt} s_j t_j t_{\bt}
  = \sum_{\bt \in W_{m}^d} s_{\bt}
        \left( \sssum{j = 1}{d}  s_j t_j \right) t_{\bt}
  = 1.
\]

The statement of the lemma now follows from
Lemma~\ref{L:PropOfWords}(\ref{L:PropOfWords-X0}),
or is Proposition~1.4 of~\cite{Cu1}.
\end{proof}

\begin{lem}\label{L:SameLength}
Let $d \in \{ 2, 3, 4, \ldots \},$
let $m \in \N,$
and let $a_1, a_2, \ldots, a_m \in L_d.$
Then there exist~$n \in \Nz,$
a finite set $F \subset W_{\I}^d,$
and numbers $\ld_{k, \af, \bt} \in \C$
for $k = 1, 2, \ldots, m,$ $\af \in F,$ and $\bt \in W_n^d,$
such that
\begin{equation}\label{Eq:SameLength}
a_k = \sum_{\af \in F} \sum_{\bt \in W_n^d}
         \ld_{k, \af, \bt} s_{\af} t_{\bt}
\end{equation}
for $k = 1, 2, \ldots, m.$
\end{lem}

\begin{proof}
Since $a_1, a_2, \ldots, a_m$ are linear combinations of products
of the $s_j$ and~$t_j,$
it suffices to prove the statement when $a_1, a_2, \ldots, a_m$
are products of the $s_j$ and~$t_j.$
By Lemma~\ref{L:PropOfWords}(\ref{L:PropOfWords-2}),
we may assume $a_k = s_{\af_k} t_{\bt_k}$
with $\af_k, \bt_k \in W^d_{\I}.$
Set
\[
n = \max \big( l (\bt_1), \, l (\bt_2), \, \ldots, \, l (\bt_m) \big).
\]
For $k = 1, 2, \ldots, m,$ set $l_k = n - l (\bt_k).$
Take
\[
F = \bigcup_{k = 1}^m \big\{ \af_k \af \colon \af \in W_{l_k}^d \big\}.
\]
Lemma~\ref{L:mSum} and
Lemma~\ref{L:PropOfWords}(\ref{L:PropOfWords-4}) imply
\[
a_k = s_{\af_k} t_{\bt_k}
    = \sum_{\af \in W_{l_k}^d} s_{\af_k} s_{\af} t_{\af} t_{\bt_k}
    = \sum_{\af \in W_{l_k}^d} s_{\af_k \af} t_{\bt_k \af}.
\]
This expression has the form in~(\ref{Eq:SameLength}).
\end{proof}

\begin{dfn}\label{N-LComb}
Let $A$ be any of $L_d,$
$C_d,$
or~$L_{\I}.$
Let $\ld = (\ld_1, \ld_2, \ldots, \ld_d) \in \C^d.$
(For $A = L_{\I},$ take $d = \I,$
take $\C^d = \bigoplus_{j = 1}^{\I} \C,$
and take $\ld = (\ld_1, \ld_2, \ldots ).$)
Define $s_{\ld}, t_{\ld} \in A$ by
\[
s_{\ld} = \sum_{j = 1}^d \ld_j s_j
\andeqn
t_{\ld} = \sum_{j = 1}^d \ld_j t_j.
\]
\end{dfn}

In principle, this notation conflicts with Notation~\ref{N:WordsInGens},
but no confusion should arise.

\begin{lem}\label{L-LCombProd}
Let the notation be as in Definition~\ref{N-LComb}.
Let $\ld, \gm \in \C^d.$
Then
\[
t_{\ld} s_{\gm}
 = \left( \ssum{j} \ld_j \gm_j \right) \cdot 1.
\]
\end{lem}

\begin{proof}
This is immediate from the relations $t_j s_k = 1$ for $j = k$
and $t_j s_k = 0$ for $j \neq k.$
\end{proof}

\section{Representations on Banach spaces}\label{Sec:Reps}

\indent
In this section, we discuss \rpn s of
Leavitt and Cohn algebras on Banach spaces.
Much of what we say makes sense
for \rpn s on general Banach spaces,
but some only works for \rpn s on
spaces of the form $L^p (X, \mu).$
Some of the constructions work for general algebras,
but some are special to \rpn s
of Leavitt and Cohn algebras.
Some of what we do is intended primarily to
establish notation and conventions.
(For example, \rpn s are required to be unital,
and isomorphisms are required to be surjective.)

All Banach spaces in this article will be over~$\C.$

\begin{ntn}\label{N:LE}
Let $E$ and $F$ be Banach spaces.
We denote by $L (E, F)$ the Banach space of all bounded linear
operators from $E$ to~$F,$
and by $K (E, F) \subset L (E, F)$
the closed subspace of all compact linear
operators from $E$ to~$F.$
When $E = F,$
we get the Banach algebra $L (E)$
and the closed ideal $K (E) \subset L (E).$
\end{ntn}

The following definition summarizes terminology for Banach spaces
that we use.
We will need both isometries and isomorphisms of Banach spaces.

\begin{dfn}\label{D:BSpTerm}
If $E$ and $F$ are Banach spaces,
we say that $a \in L (E, F)$ is an {\emph{isomorphism}}
if $a$ is bijective.
(By the Open Mapping Theorem, $a^{-1}$ is also bounded.)
If an isomorphism exists, we say $E$ and $F$ are {\emph{isomorphic}}.

We say that $a \in L (E, F)$ is an  {\emph{isometry}}
if $\| a \xi \| = \| \xi \|$ for every $\xi \in E.$
(We do not require that $a$ be surjective.)
If there is a surjective isometry from $E$ to~$F,$
we say that $E$ and $F$ are {\emph{isometrically isomorphic}}.

If $A$ and $B$ are Banach algebras,
we say that a \hm\  $\ph \colon A \to B$ is an {\emph{isomorphism}}
if it is \ct\  and bijective.
(By the Open Mapping Theorem, $\ph^{-1}$ is also \ct.)
If an isomorphism exists, we say $A$ and $B$ are {\emph{isomorphic}}.
If in addition $\ph$ is isometric,
we call it an {\emph{isometric isomorphism}}.
If such a map exists, we say $A$ and $B$
are {\emph{isometrically isomorphic}}.
\end{dfn}

For emphasis (because some authors do not use this convention):
{\emph{isomorphisms are required to be surjective}}.

The following notation for duals
is intended to avoid conflict with the notation for adjoints.

\begin{ntn}\label{N:Dual}
Let $E$ be a Banach space.
We denote by $E'$ its dual Banach space $L (E, \C),$
consisting of all bounded linear functionals on~$E.$
If $F$ is another Banach space and $a \in L (E, F),$
we denote by $a'$ the element of $L (F', E')$
defined by $a' (\om) (\xi) = \om (a \xi)$
for $\xi \in E$ and $\om \in F'.$
\end{ntn}

We will also need some notation for specific spaces.

\begin{ntn}\label{N:FDP}
For any set~$S,$ we give $l^p (S)$ the usual
meaning (using counting measure on~$S$),
and we take (as usual) $l^p = l^p (\N).$
For $d \in \N$ and $p \in [1, \I],$
we let $l^p_d = l^p \big( \{1, 2, \ldots, d \} \big).$
We further let $\MP{d}{p} = L \big( l_d^p \big)$
with the usual operator norm,
and we algebraically identify $\MP{d}{p}$ with the algebra $M_d$ of
$d \times d$ complex matrices in the standard way.
\end{ntn}

We warn of a notational conflict.
Many articles on Banach spaces use $L_p (X, \mu)$
rather than $L^p (X, \mu),$
and use $l_p^d$ for what we call~$l_d^p.$
Our convention is chosen
to avoid conflict with the standard notation
for the Leavitt algebra $L_d$ of Definition~\ref{D:Leavitt}.

For fixed~$d,$ the norms on the various $\MP{d}{p}$
are of course equivalent, but they are not equal.
The following example illustrates this.

\begin{exa}\label{E:NormRankOne}
Let $d \in \N.$
Let $\et, \mu \in \C^d,$
and let $\om_{\et} \colon \C^d \to \C$
be the linear functional
$\om_{\et} (\xi) = \sum_{j = 1}^d \et_j \xi_j.$
Let $a \in M_d$ be the rank one operator given by
$a \xi = \om_{\et} (\xi) \mu.$
Let $p, q \in [1, \I]$ satisfy $\frac{1}{p} + \frac{1}{q} = 1,$
and regard~$a$ as an element of $\MP{d}{p}.$
Then one can calculate that $\| a \| = \| \mu \|_p \| \et \|_q.$
\end{exa}

The following terminology and related observation
will be used many times.

\begin{dfn}\label{D:LpSupp}
Let $\XBM$ be a \msp\  and let $p \in [1, \I].$
For a function $\xi \in L^p (X, \mu)$
(or, more generally,
any measurable function on~$X$)
and a subset $E \subset X,$
we will say that $\xi$ is {\emph{supported}} in~$E$
if $\xi (x) = 0$ for almost all $x \in X \SM E.$
If $I$ is a countable set,
and $(\xi_i)_{i \in I}$ is a family of elements of $L^p (X, \mu)$
or measurable functions on~$X,$
we say that the $\xi_i$
have {\emph{disjoint supports}} if there are
disjoint subsets $E_i \subset X$ such that $\xi_i$ is supported in $E_i$
for all $i \in I.$
\end{dfn}

\begin{rmk}\label{R:LpSuppNorm}
Let $\XBM$ be a \msp, let $p \in [1, \I),$
let $I$ be a countable set,
and let $\xi_i \in L^p (X, \mu),$
for $i \in I,$
have disjoint supports.
Then
\[
\left\| \ssum{i \in I} \xi_i \right\|_p^p
  = \sum_{i \in I} \| \xi_i \|_p^p.
\]
\end{rmk}

\begin{dfn}\label{D:Repn}
Let $A$ be any unital complex algebra.
Let $E$ be a nonzero Banach space.
A {\emph{representation}} of $A$ on~$E$
is a unital algebra \hm\  from $A$ to $L (E).$
\end{dfn}

We do not say anything about continuity.
We will mostly be interested in \rpn s of $L_d,$ $C_d,$ and~$L_{\I},$
for which we do not use a topology on the algebra,
or of $M_d,$ for which all \rpn s are \ct\  by
finite dimensionality.

\begin{rmk}\label{R:HCuntz}
The well known representations of $L_d,$ $C_d,$ and~$L_{\I}$
are those on a Hilbert space~$H.$
Choose any $d$ isometries $w_1, w_2, \ldots, w_d \in L (H)$
(or, for the case of $L_{\I},$
isometries $w_1, w_2, \ldots \in L (H)$)
with orthogonal ranges.
Then we obtain a \rpn\  $\rh \colon C_d \to L (H)$
or $\rh \colon L_{\I} \to L (H)$
by setting $\rh (s_j) = w_j$ and $\rh (t_j) = w_j^*$ for all~$j.$
If $d < \I$ and $\sum_{j = 1}^d w_j w_j^* = 1,$
we get a \rpn\  of~$L_d.$
These \rpn s are even *-\rpn s:
making $A$ a *-algebra as in
Lemma~\ref{L:Invol}(\ref{L:Invol-Star}),
we have $\rh (a^*) = \rh (a)^*$ for all $a \in A.$

The closures ${\ov{\rh (A)}}$
do not depend on the choice of~$\rh$
(in case $A = C_d,$ provided
$\rh \left( 1 - \sum_{j = 1}^d s_j s_j^* \right) \neq 0$),
and are the usual Cuntz algebra $\OA{d}$
when $A = L_d$ (including the case $d = \I$),
and the extended Cuntz algebras $E_d$ when $A = C_d.$
See Theorem~1.12 of~\cite{Cu1} for $L_d$ and~$L_{\I},$
and see Lemma~3.1 of~\cite{Cu2} for~$C_d.$
\end{rmk}

Further examples of representations of $L_d,$ $C_d,$ and~$L_{\I}$
will be given in Section~\ref{Sec:ExsOrReps}.

Representations of $L_d,$ $C_d,$ and $L_{\I}$ have a
kind of rigidity property.
It is stronger for $L_d$ than for the others:
a \rpn\  is determined by
the images of the~$s_j$ or by the images of the~$t_j.$

\begin{lem}\label{L:AUniq}
Let $A$ be any of $L_d$ (Definition~\ref{D:Leavitt}),
$C_d$ (Definition~\ref{D:Cohn}),
or~$L_{\I}$ (Definition~\ref{D:LInfty}).
Let $B$ be a unital algebra over~$\C,$
and let $\ph, \ps \colon L_d \to B$ be unital \hm s
such that for all $j$ we have $\ph (s_j) = \ps (s_j)$
and $\ph (s_j t_j) = \ps (s_j t_j).$
Then $\ph = \ps.$
The same conclusion holds if we replace $\ph (s_j) = \ps (s_j)$
with $\ph (t_j) = \ps (t_j).$
\end{lem}

\begin{proof}
Assume $\ph (s_j) = \ps (s_j)$ for all~$j.$
Using the relations $t_j s_j t_j = t_j$ at the first step
and $\ph (t_j) \ph (s_j) = 1$ at the last step,
we calculate:
\[
\ph (t_j)
  = \ph (t_j) \ph (s_j t_j)
  = \ph (t_j) \ps (s_j t_j)
  = \ph (t_j) \ps (s_j) \ps (t_j)
  = \ph (t_j) \ph (s_j) \ps (t_j)
  = \ps (t_j).
\]
The first statement follows.
If instead $\ph (t_j) = \ps (t_j)$ for all~$j,$
similar reasoning (using $s_j t_j s_j = s_j$) gives
\begin{align*}
\ph (s_j)
& = \ph (s_j t_j) \ph (s_j)
    \\
& = \ps (s_j t_j) \ph (s_j)
  = \ps (s_j) \ps (t_j) \ph (s_j)
  = \ps (s_j) \ph (t_j) \ph (s_j)
  = \ps (s_j).
\end{align*}
This completes the proof.
\end{proof}

\begin{lem}\label{L:LdUniq}
Let $d \in \{ 2, 3, 4, \ldots \},$
let $B$ be a unital algebra over~$\C,$
and let $\ph, \ps \colon L_d \to B$ be unital \hm s
such that $\ph (s_j) = \ps (s_j)$ for $j \in \{ 1, 2, \ldots, d \}.$
Then $\ph = \ps.$
The same conclusion holds if we replace $\ph (s_j) = \ps (s_j)$
with $\ph (t_j) = \ps (t_j).$
\end{lem}

\begin{proof}
For $j \in \{ 1, 2, \ldots, d \},$
define idempotents $e_j, f_j \in B$ by
$e_j = \ph (s_j t_j)$ and $f_j = \ps (s_j t_j).$
By Lemma~\ref{L:AUniq}, it suffices to show that
$e_j = f_j$ for all~$j.$

First assume that $\ph (s_j) = \ps (s_j)$ for all~$j.$
Using this statement at the first and third steps,
$t_j s_k = 0$ for $j \neq k$ at the second step,
and $\sum_{k = 1}^d e_k = 1$ at the last step,
we have
\begin{equation}\label{Eq:LdUniq-1}
f_j e_j
  = \ps (s_j t_j s_j) \ph (t_j)
  = \sum_{k = 1}^d \ps (s_j t_j s_k) \ph (t_k)
  = \sum_{k = 1}^d f_j e_k
  = f_j.
\end{equation}
If now $j \neq k,$
then
\[
f_k e_j = f_k e_k e_j = 0.
\]
This equation, together with $\sum_{k = 1}^d f_k = 1,$
gives
\begin{equation}\label{Eq:LdUniq-2}
e_j = \sum_{k = 1}^d f_k e_j = f_j e_j.
\end{equation}
The proof is completed
by combining (\ref{Eq:LdUniq-1}) and~(\ref{Eq:LdUniq-2}).

Now assume that $\ph (t_j) = \ps (t_j)$ for all~$j.$
With similar justifications, we get
\[
f_j e_j
  = \ps (s_j) \ph (t_j s_j t_j)
  = \sum_{k = 1}^d \ps (s_k) \ph (t_k s_j t_j)
  = \sum_{k = 1}^d f_k e_j
  = e_j.
\]
So for $j \neq k$ we have
$f_j e_k = f_j f_k e_k = 0.$
Combining these results gives
$f_j = \sum_{k = 1}^d f_j e_k = f_j e_j,$
whence $e_j = f_j.$
\end{proof}

The analog of Lemma~\ref{L:LdUniq}
for $L_{\I}$ and $C_d$ is false.
See Example~\ref{E:01RepID}.

The following definition gives several natural conditions to
ask of a representation of $L_d,$ $C_d,$ or~$L_{\I}$
on a Banach space~$E.$
The condition in~(\ref{D:Repn-SFI})
is motivated by the following property
of a *-\rpn\  $\rh$ of $L_d$ or $C_d$ on a Hilbert space~$H$
(as in Remark~\ref{R:HCuntz}):
for $\ld_1, \ld_2, \ldots, \ld_d \in \C$
and $\xi \in H,$
we have
\begin{equation}\label{Eq:2Std}
\left\| \rh \left( \sssum{j = 1}{d} \ld_j s_j \right) \xi \right\|
  = \| (\ld_1, \ld_2, \ldots, \ld_d) \|_2 \| \xi \|.
\end{equation}
In Definition~\ref{D:SpatialRep}
and Definition~\ref{D:StdRep},
we will see further conditions on \rpn s
which are natural when $E = L^p (X, \mu).$

\begin{dfn}\label{D:KindsOfReps1}
Let $A$ be any of $L_d$ (Definition~\ref{D:Leavitt}),
$C_d$ (Definition~\ref{D:Cohn})
or~$L_{\I}$ (Definition~\ref{D:LInfty}).
Let $E$ be a nonzero Banach space,
and let $\rh \colon A \to L (E)$ be a representation.
\begin{enumerate}
\item\label{D:Repn-CG}
We say that $\rh$ is
{\emph{contractive on generators}} if for every~$j,$
we have $\| \rh (s_j) \| \leq 1$ and $\| \rh (t_j) \| \leq 1.$
\item\label{D:Repn-FI}
We say that $\rh$ is
{\emph{forward isometric}} if $\rh (s_j)$ is an isometry for every~$j.$
\item\label{D:Repn-SFI}
We say that $\rh$ is
{\emph{strongly forward isometric}} if
$\rh$ is forward isometric and
(following Definition~\ref{N-LComb})
for every $\ld \in \C^d,$
the element $\rh (s_{\ld})$ is a scalar multiple
of an isometry.
\end{enumerate}
\end{dfn}

\begin{rmk}\label{R-cogImpfis}
A \rpn\  which is \cog\  is clearly \fis.
\end{rmk}

A representation of $L_d$ which is contractive on generators
need not be strongly forward isometric.
See Example~\ref{E-LdDblSkew} below.
We will see in Example~\ref{E:01RepID} below
that a strongly forward isometric \rpn\  of $L_{\I}$
need not be contractive on generators.
We do not know whether this can happen for $L_d$
with $d$ finite.

We now describe several ways to make new \rpn s from old ones.
The first two (direct sums and tensoring with the identity
on some other Banach space)
work for representations of general algebras.
They also work for more general choices of norms on the
direct sum and tensor product than we consider here.
For simplicity, we restrict to specific choices
which are suitable for representations
on spaces of the form $L^p (X, \mu).$

\begin{lem}\label{L-pSum}
Let $A$ be a unital complex algebra,
and let $p \in [1, \I].$
Let $n \in \N,$
and for $l = 1, 2, \ldots, n$ let $(X_l, \cB_l, \mu_l)$
be a \sfm\  and
let $\rh_l \colon A \to L (L^p (X_l, \mu_l))$ be a \rpn.
Equip $E = \bigoplus_{l = 1}^n L^p (X_l, \mu_l)$
with the norm
\[
\| (\xi_1, \xi_2, \ldots, \xi_n ) \|
 = \left( \sssum{l = 1}{n} \| \xi_l \|_p^p \right)^{1 / p}.
\]
Then there is a unique \rpn\  $\rh \colon A \to L (E)$
such that
\[
\rh (a) (\xi_1, \xi_2, \ldots, \xi_n )
        = \big( \rh_1 (a) \xi_1, \, \rh_2 (a) \xi_2,
             \, \ldots, \, \rh_n (a) \xi_n \big)
\]
for $a \in L_d$ and $\xi_l \in L^p (X_l, \mu_l)$
for $l = 1, 2, \ldots, n.$
If $A$ is any of $L_d,$ $C_d,$ or~$L_{\I},$
and each $\rh_l$ is \cog\  or \fis,
then so is~$\rh.$
\end{lem}

\begin{proof}
This is immediate.
\end{proof}

\begin{rmk}\label{R-pSumRmk}
The norm used in Lemma~\ref{L-pSum}
identifies $E$ with $L^p \left( \coprod_{l = 1}^n X_l \right),$
using the obvious measure.
We write this space as
\[
L^p (X_1, \mu_1) \oplus_p L^p (X_2, \mu_2)
   \oplus_p \cdots \oplus_p L^p (X_n, \mu_n).
\]
We write the \rpn~$\rh$ as
\[
\rh = \rh_1 \oplus_p \rh_2 \oplus_p \cdots \oplus_p \rh_n,
\]
and call it
the {\emph{$p$-direct sum}} of $\rh_1, \rh_2, \ldots, \rh_n.$
\end{rmk}

Example~\ref{E-LdDblSkew} below shows that if
$\rh_1, \rh_2, \ldots, \rh_n$ are \sfi,
it does not follow that
$\rh$ is \sfi.

One can form a $p$-direct sum $\bigoplus_{i \in I} \rh_i$
over an infinite index set~$I$ provided
$\sup_{i \in I} \| \rh_i (a) \| < \I$ for all $a \in A.$

We now consider tensoring
with the identity on some other Banach space.
This requires the theory of tensor products of Banach spaces
and of operators on them.
We will consider only a very special case.

Fix $p \in [1, \I).$
We need a tensor product, defined on pairs of Banach
spaces both of the form $L^p (X, \mu)$
for \sft\  measures~$\mu,$
for which one has a canonical isometric identification
\[
L^p (X, \mu) \otimes L^p (Y, \nu) = L^p (X \times Y, \, \mu \times \nu).
\]
(One can't reasonably expect something like this for $p = \I.$)
The tensor product
\[
L^p (X, \mu) {\widetilde{\otimes}}_{\Dt_p} L^p (Y, \nu)
\]
described in Chapter~7 of~\cite{DF} will serve our purpose.
To simplify the notation,
we simply write $L^p (X, \mu) \otimes_p L^p (Y, \nu).$

We note that there is a more general construction,
the $M$-norm of~\cite{Ch},
defined before~(4) on page~3 of~\cite{Ch}.
This norm is defined for the tensor product of a Banach lattice
(this includes all spaces $L^p (X, \mu)$ for all~$p$)
and a Banach space.
Theorem~3.2(1) of~\cite{Ch} shows that whenever
$p \in [1, \I)$ and $\XBM$ is a finite measure space,
then the completion of $E \otimes_{\mathrm{alg}} L^p (X, \mu)$
in this norm is isometrically isomorphic to the space of
$L^p$~functions on~$X$ with values in~$E.$
In particular, regardless of the value of~$p,$
this norm gives the properties in Theorem~\ref{T-LpTP}.

\begin{thm}\label{T-LpTP}
Let $\XBM$ and $\YCN$ be \sfm s.
Let $p \in [1, \I).$
Write $L^p (X, \mu) \otimes_p L^p (Y, \nu)$
for the Banach space completed tensor product
$L^p (X, \mu) {\widetilde{\otimes}}_{\Dt_p} L^p (Y, \nu)$
defined in~7.1 of~\cite{DF}.
Then there is a unique isometric isomorphism
\[
L^p (X, \mu) \otimes_p L^p (Y, \nu)
  \cong L^p (X \times Y, \, \mu \times \nu)
\]
which identifies,
for every $\xi \in L^p (X, \mu)$ and $\et \in L^p (Y, \nu),$
the element $\xi \otimes \et$ with the function
$(x, y) \mapsto \xi (x) \et (y)$ on $X \times Y.$
Moreover:
\begin{enumerate}
\item\label{T-LpTP-0}
Under the identification above,
the linear span of all $\xi \otimes \et,$
for $\xi \in L^p (X, \mu)$ and $\et \in L^p (Y, \nu),$
is dense in $ L^p (X \times Y, \, \mu \times \nu).$
\item\label{T-LpTP-1}
$\| \xi \otimes \et \|_p = \| \xi \|_p \| \et \|_p$
for all $\xi \in L^p (X, \mu)$ and $\et \in L^p (Y, \nu).$
\item\label{T-LpTP-2}
The tensor product $\otimes_p$ is commutative and associative.
\item\label{T-LpTP-3a}
Let
\[
(X_1, \cB_1, \mu_1), \,\,\,\,\,\,
(X_2, \cB_2, \mu_2), \,\,\,\,\,\,
(Y_1, \cC_1, \nu_1),
\andeqn
(Y_2, \cC_2, \nu_2)
\]
be \sfm s.
Let
\[
a \in L \big( L^p (X_1, \mu_1), \, L^p (X_2, \mu_2) \big)
\andeqn
b \in L \big( L^p (Y_1, \nu_1), \, L^p (Y_2, \nu_2) \big).
\]
Then there exists a unique
\[
c \in L \big( L^p (X_1 \times Y_1, \, \mu_1 \times \nu_1),
          \, L^p (X_2 \times Y_2, \, \mu_2 \times \nu_2) \big)
\]
such that, making the identification above,
$c (\xi \otimes \et) = a \xi \otimes b \et$
for all $\xi \in L^p (X_1, \mu_1)$ and $\et \in L^p (Y_1, \nu_1).$
We call this operator $a \otimes b.$
\item\label{T-LpTP-3b}
The operator $a \otimes b$ of~(\ref{T-LpTP-3a})
satisfies $\| a \otimes b \| = \| a \| \cdot \| b \|.$
\item\label{T-LpTP-4}
The tensor product of operators defined in~(\ref{T-LpTP-3a})
is associative, bilinear, and satisfies (when the domains
are appropriate)
$(a_1 \otimes b_1) (a_2 \otimes b_2) = a_1 a_2 \otimes b_1 b_2.$
\end{enumerate}
\end{thm}

\begin{proof}
The identification of
$L^p (X, \mu) {\widetilde{\otimes}}_{\Dt_p} L^p (Y, \nu)$
is in 7.2 of~\cite{DF}.
Part~(\ref{T-LpTP-0}) is part of the definition of a tensor product
of Banach spaces.
Part~(\ref{T-LpTP-1}) is in 7.1 of~\cite{DF}.
Part~(\ref{T-LpTP-2}) follows from the corresponding properties
of products of measure spaces.
Parts (\ref{T-LpTP-3a}) and~(\ref{T-LpTP-3b})
are a special case of 7.9 of~\cite{DF},
or of Theorem~1.1 of~\cite{FIP}.
Part~(\ref{T-LpTP-4}) follows from part~(\ref{T-LpTP-3a})
and part~(\ref{T-LpTP-0}) by examining what happens on
elements of the form $\xi \otimes \et.$
\end{proof}

In fact, the statements about tensor products
of operators in
Theorem~\ref{T-LpTP}(\ref{T-LpTP-3a})
and Theorem~\ref{T-LpTP}(\ref{T-LpTP-3b})
are valid in considerably greater generality;
see, for example, Theorem~1.1 of~\cite{FIP}.
We need a slightly more general statement
in the proof of the following lemma.

\begin{lem}\label{L-TPRep}
Let $A$ be a unital complex algebra,
let $p \in [1, \I),$
let $\XBM$ and $\YCN$ be \sfm s,
and let $\rh \colon A \to \LLp$ be a \rpn.
Then there is a unique \rpn\  %
$\rh \otimes_p 1 \colon A \to L^p (X \times Y, \, \mu \times \nu)$
such that, following Theorem~\ref{T-LpTP}(\ref{T-LpTP-3a}),
we have
$(\rh \otimes_p 1) (a) = \rh (a) \otimes 1$ for all $a \in A.$
This \rpn\  satisfies $\| (\rh \otimes_p 1) (a) \| = \| \rh (a) \|$
for all $a \in A.$
If $A$ is any of $L_d,$ $C_d,$ or~$L_{\I},$
and $\rh$ is any of \cog, \fis, or \sfi,
then so is $\rh \otimes_p 1.$
\end{lem}

\begin{proof}
Existence of $\rh \otimes_p 1$
and $\| (\rh \otimes_p 1) (a) \| = \| \rh (a) \|$
follow from parts
(\ref{T-LpTP-3a}), (\ref{T-LpTP-3b}), and~(\ref{T-LpTP-4})
of Theorem~\ref{T-LpTP}.
If $A$ is one of $L_d,$ $C_d,$ or~$L_{\I}$
and $\rh$ is \cog,
then the norm equation implies that
$\rh \otimes_p 1$ is also \cog.

To prove that $\rh \otimes_p 1$ is \fis\  or \sfi\  when
$\rh$ is,
it suffices to prove that if $s \in \LLp$ is isometric
(not necessarily surjective),
then so is
$s \otimes 1 \in L \big( L^p (X \times Y, \, \mu \times \nu) \big).$
Let $E \subset L^p (X, \mu)$ be the range of~$s.$
Let $t \colon E \to L^p (X, \mu)$ be the inverse of
the corestriction of~$s.$
Then $\| t \| = 1.$
Let $F \subset L^p (X \times Y, \, \mu \times \nu)$ be the closed
linear span of all $\xi \otimes \et$ with $\xi \in E$
and $\et \in L^p (Y, \nu).$
Then Theorem~1.1 of~\cite{FIP}
implies that
$t \otimes 1 \colon F \to L^p (X \times Y, \, \mu \times \nu)$
is defined and satisfies $\| t \otimes 1 \| = 1.$
Since $\| s \otimes 1 \| = 1$ and
$(t \otimes 1) (s \otimes 1) = 1,$
this implies that $s \otimes 1$ is isometric.
\end{proof}

In the proof of Lemma~\ref{L-TPRep},
we could also have used Lamperti's Theorem
(Theorem~\ref{T:Lamperti} below)
instead of Theorem~1.1 of~\cite{FIP}
to show that $\| t \otimes 1 \| = 1.$

Finally, we present constructions of new representations
that are special to the kinds of algebras we consider.
They will play important technical roles.

\begin{lem}\label{L:MultLd}
Let $d \in \{ 2, 3, 4, \ldots \},$
let $E$ be a nonzero Banach space,
and let $\rh \colon L_d \to L (E)$ be a representation.
Let $u \in L (E)$ be invertible.
Then there is a unique \rpn\  $\rh^u \colon L^d \to L (E)$
such that for $j = 1, 2, \ldots, d$ we have
\[
\rh^u (s_j) = u \rh (s_j)
\andeqn
\rh^u (t_j) = \rh (t_j) u^{-1}.
\]
Assume further that $u$ is isometric.
If $\rh$ is \cog, \fis, or \sfi\  %
(Definition~\ref{D:KindsOfReps1}),
then so is $\rh^u.$
\end{lem}

\begin{proof}
For the first part, we check the relations
(\ref{Eq:Leavitt1}), (\ref{Eq:Leavitt2}), and~(\ref{Eq:Leavitt3})
in Definition~\ref{D:Leavitt}.
For $j \in \{ 1, 2, \ldots, d \}$
we have
\[
[\rh (t_j) u^{-1}] [u \rh (s_j)] = \rh (t_j) \rh (s_j) = 1,
\]
for distinct $j, k \in \{ 1, 2, \ldots, d \}$
we have
\[
[\rh (t_j) u^{-1}] [u \rh (s_k)] = \rh (t_j) \rh (s_k) = 0,
\]
and we have
\[
\sum_{j = 1}^d [u \rh (s_j)] [\rh (t_j) u^{-1}]
  = u \left( \sssum{j = 1}{d} s_j t_j \right) u^{-1}
  = u \cdot 1 \cdot u^{-1}
  = 1.
\]

The second part follows directly from the definitions of the
conditions on the \rpn.
\end{proof}

For \rpn s of $C_d$ and~$L_{\I},$
we only need one sided invertibility.

\begin{lem}\label{L:MultCd}
Let $d \in \{ 2, 3, 4, \ldots \},$
let $E$ be a nonzero Banach space,
and let $\rh \colon C_d \to L (E)$ be a representation.
Let $u, v \in L (E)$ satisfy
$v u = 1.$
Then there is a unique \rpn\  $\rh^{u, v} \colon C^d \to L (E)$
such that for $j = 1, 2, \ldots, d$ we have
\[
\rh^{u, v} (s_j) = u \rh (s_j)
\andeqn
\rh^{u, v} (t_j) = \rh (t_j) v.
\]
If $u$ is an isometry
and $\rh$ is \fis\   or \sfi,
then so is $\rh^{u, v}.$
If $\| u \|, \, \| v \| \leq 1$ and $\rh$ is \cog,
then so is $\rh^{u, v}.$
\end{lem}

\begin{proof}
The proof is essentially the same as that of Lemma~\ref{L:MultLd},
using the relations
(\ref{Eq:Leavitt1}) and (\ref{Eq:Leavitt2})
in Definition~\ref{D:Leavitt}
(see Definition~\ref{D:Cohn}).
Since no relation involves $s_j t_j,$
we do not need to have $u v = 1.$
\end{proof}

\begin{lem}\label{L:MultLI}
Let $\rh \colon L_{\I} \to L (E)$ be a representation.
Let $u, v \in L (E)$ satisfy
$v u = 1.$
Then there is a unique \rpn\  $\rh^{u, v} \colon L_{\I} \to L (E)$
such that for $j = 1, 2, \ldots$ we have
\[
\rh^{u, v} (s_j) = u \rh (s_j)
\andeqn
\rh^{u, v} (t_j) = \rh (t_j) v.
\]
If $u$ is an isometry
and $\rh$ is \fis\   or \sfi,
then so is $\rh^{u, v}.$
If $\| u \|, \, \| v \| \leq 1$ and $\rh$ is \cog,
then so is $\rh^{u, v}.$
\end{lem}

\begin{proof}
The proof is the same as that of Lemma~\ref{L:MultCd},
using the relations
(\ref{Eq:ILeavitt1}) and (\ref{Eq:ILeavitt2})
in Definition~\ref{D:LInfty}.
\end{proof}

Examples \ref{E:01RepI2} and~\ref{E:01RepID}
illustrate this construction.

\begin{lem}\label{L:TransposeRep}
Let $A$ be any of $L_d$ (Definition~\ref{D:Leavitt}),
$C_d$ (Definition~\ref{D:Cohn}),
or~$L_{\I}$ (Definition~\ref{D:LInfty}).
Let $E$ be a nonzero Banach space,
and let $\rh \colon A \to L (E)$ be a representation.
Using Notation~\ref{N:Dual}
and the notation of Lemma~\ref{L:Invol}(\ref{L:Invol-Prime}),
define a function $\rh' \colon A \to L (E')$
by $\rh' (a) = \rh (a')'.$
Then $\rh'$ is a \rpn\  of~$A,$
called the dual of~$\rh.$
\end{lem}

\begin{proof}
This follows from Lemma~\ref{L:Invol}(\ref{L:Invol-Prime}).
\end{proof}

\section{Examples of representations}\label{Sec:ExsOrReps}

\indent
In this section, we give a number of examples of representations
of Leavitt algebras on Banach spaces.
These examples illustrate some of the possible behavior
of \rpn s.
We begin with basic examples of \rpn s on~$l^p.$

\begin{exa}\label{E:LdRepOnlp}
Fix $p \in [1, \I].$
We take $l^p = l^p (\N).$
Let $d \in \{ 2, 3, 4, \ldots \}.$
Define functions $f_1, f_2, \ldots, f_d \colon \N \to \N$
by
\[
f_{d, j} (n) = d (n - 1) + j
\]
for $n \in \N.$
The functions $f_j$ are injective
and have disjoint ranges whose union is~$\N.$
For $j = 1, 2, \ldots, d,$
define $v_{d, j}, w_{d, j} \in L (l^p)$ by,
for $\xi = \big( \xi (1), \, \xi (2), \, \ldots \big) \in l^p$
and $n \in \N,$
\[
(v_{d, j} \xi) (n)
 = \begin{cases}
   \xi \big( f_{d, j}^{- 1} (n) \big)
                 & n \in {\mathrm{ran}} ( f_{d, j} )
        \\
   0             & n \not\in {\mathrm{ran}} ( f_{d, j} )
\end{cases}
\,\,\,\,\,\,
\andeqn
\,\,\,\,\,\,
(w_{d, j} \xi) (n)
 = \xi ( f_{d, j} (n) ).
\]
Then $v_{d, j}$ is isometric, $w_{d, j}$ is contractive,
and there is a unique \rpn\  $\rh \colon L_d \to L (l^p)$
such that
\[
\rh (s_j) = v_{d, j}
\andeqn
\rh (t_j) = w_{d, j}
\]
for $j = 1, 2, \ldots, d.$

If we let $\dt_n \in l^p$ be the sequence given by $\dt_n (n) = 1$
and $\dt_n (m) = 0$ for $m \neq n,$
then
\[
v_{d, j} \dt_n = \dt_{f_{d, j} (n)}
\andeqn
w_{d, j} \dt_n = \begin{cases}
   \dt_{f_{d, j}^{-1} (n)} & n \in {\mathrm{ran}} (f_{d, j})
       \\
   0                       & n \not\in {\mathrm{ran}} (f_{d, j})
\end{cases}
\]
for all $n \in \N.$
For $p \in [1, \I),$
these formulas determine $v_{d, j}, w_{d, j} \in L (l^p)$ uniquely.

The \rpn~$\rh$ is clearly
\cog\  (Definition~\ref{D:KindsOfReps1}(\ref{D:Repn-CG})).
One can check directly that $\rh$ is
\sfi\  (Definition~\ref{D:KindsOfReps1}(\ref{D:Repn-SFI})),
with
\[
\left\| \rh \left( \sssum{j = 1}{d} \ld_j s_j \right) \xi \right\|
  = \| (\ld_1, \ld_2, \ldots, \ld_d) \|_p \| \xi \|.
\]
(Or use Lemma~\ref{T-SpatialRepsL} below.)
\end{exa}

\begin{exa}\label{E:CdRepOnlp}
Let the notation be as in Example~\ref{E:LdRepOnlp}.
Then there is a unique \rpn\  $\pi \colon C_d \to L (l^p)$
such that
\[
\pi (s_j) = v_{d + 1, \, j}
\andeqn
\pi (t_j) = w_{d + 1, \, j}
\]
for $j = 1, 2, \ldots, d.$
Since $\sum_{j = 1}^d w_{d + 1, \, j} v_{d + 1, \, j} \neq 1,$
this \rpn\  does not descend to a \rpn\  of~$L_d.$
\end{exa}

\begin{exa}\label{E:LIRepOnlp}
Let $p \in [1, \I].$
We obtain a \rpn\  $\rh$ of $L_{\I}$ on $l^p$
by the same method as in Example~\ref{E:LdRepOnlp},
taking $v_{\I, j}$ and $w_{\I, j}$ to come from the functions
\[
f_{\I, j} (n) = 2^j n - 2^{j - 1}
\]
for $j, n \in \N.$
Again, this \rpn\  is \cog\  and \sfi.
\end{exa}

In Example~\ref{E:LIRepOnlp},
the ranges of the $\rh (s_j)$ span a dense subspace
of~$l^p,$
except when $p = \I.$
The following example shows that this need not be the case.

\begin{exa}\label{E:01RepI2}
Let $p \in [1, \I].$
Let $\rh \colon L_{\I} \to L (l^p)$
be as in Example~\ref{E:LIRepOnlp}.
Then there exists a unique \rpn\  %
$\pi \colon L_{\I} \to L (l^p)$
such that for $\pi (s_j) = \rh (s_{j + 1})$
and $\pi (t_j) = \rh (t_{j + 1})$ for all $j \in \N.$
Like~$\rh,$ this \rpn\  is \cog\  and \sfi.

This example can be obtained from~$\rh$ by the method
of Lemma~\ref{L:MultLI}.
Define $u, v \in L (l^p)$ by
\[
(u \xi) (n) = \begin{cases}
   \xi (n / 2) & {\mbox{$n$ is even}}
        \\
   0           & {\mbox{$n$ is odd}}
\end{cases}
\,\,\,\,\,\,
\andeqn
\,\,\,\,\,\,
(v \xi) (n) = \xi (2 n)
\]
for $\xi \in l^p$ and $n \in \N.$
Using the notation of Example~\ref{E:LdRepOnlp},
for $p \neq \I$ these are determined by
\[
u \dt_n = \dt_{2 n}
\,\,\,\,\,\,
\andeqn
\,\,\,\,\,\,
v \dt_n
  = \begin{cases}
   \dt_{n / 2} & {\mbox{$n$ is even}}
        \\
   0           & {\mbox{$n$ is odd}}.
\end{cases}
\]
Then $\pi = \rh^{u, v}.$
\end{exa}

Let $\rh$ be as in Example~\ref{E:LIRepOnlp},
and let $\pi$ be as in Example~\ref{E:01RepI2}.
We do not know whether
the Banach algebras ${\ov{\rh (L_{\I})}}$ and ${\ov{\pi (L_{\I})}}$
are isometrically isomorphic,
or even whether they are isomorphic.

The \rpn s of Example~\ref{E:LIRepOnlp} and Example~\ref{E:01RepI2}
are both very strongly tied to the structure of
$l^p$ as a space of functions.
(They are spatial in the sense of
Definition~\ref{D:SpatialRep}(\ref{D:SpatialRep-Sp}) below.)
The following example is less regular.

\begin{exa}\label{E:01RepID}
Let $p \in [1, \I].$
Let the notation be as in Example~\ref{E:LIRepOnlp}
and Example~\ref{E:01RepI2}.
Define $y \in L (l^p)$ by
\[
(y \xi) (n)
 = \begin{cases}
   \xi (2 n)            & {\mbox{$n$ is even}}
        \\
   \xi (2 n) + \xi (n)  & {\mbox{$n$ is odd}}.
\end{cases}
\]
Thus, in the notation of Example~\ref{E:LdRepOnlp}, we have
\[
y \dt_n
  = \begin{cases}
   \dt_{n / 2} & {\mbox{$n$ is even}}
        \\
   \dt_n       & {\mbox{$n$ is odd}}.
\end{cases}
\]
We have $y u = 1,$
so, using Lemma~\ref{L:MultLI},
we obtain a \rpn\  $\sm = \rh^{u, y}$ of~$L_{\infty}.$

We have $\sm (s_j) = \pi (s_j)$ for all~$j,$
but $\sm (t_1) \neq \pi (t_1).$
Indeed,
$\pi (t_1) (\dt_1 + \dt_2) = \dt_1,$
but $y (\dt_1 + \dt_2) = 2 \dt_1,$
so $\sm (t_1) (\dt_1 + \dt_2) = 2 \dt_1.$
Thus the analog of Lemma~\ref{L:LdUniq}
for $L_{\I}$ is false.
By restriction, it also fails for~$C_d.$

The \rpn\  $\sm$ is \sfi,
since $\pi$ is.
For $p = 1,$
it is easy to check that $\| y \| \leq 1,$
from which it follows that $\sm$ is \cog.
Now suppose $p \neq 1.$
We saw above that $\sm (t_1) (\dt_1 + \dt_2) = 2 \dt_1.$
Taking $\frac{1}{p} = 0$ when $p = \I,$
we have
\[
\| \sm (t_1) (\dt_1 + \dt_2) \|
  = 2
  > 2^{1 / p}
  = \| \dt_1 + \dt_2 \|.
\]
So $\sm$ is not \cog.
In particular,
a \sfi\  \rpn\  of~$L_{\I}$ need not be \cog.
We do not know whether this can happen for \rpn s of~$L_d$
when $d$ is finite.
\end{exa}

\begin{exa}\label{E:01RepCd}
The restrictions of the \rpn s
of Examples \ref{E:LIRepOnlp}, \ref{E:01RepI2}, and~\ref{E:01RepID}
to the standard copy of $C_d$ in~$L_{\I}$
(see Lemma~\ref{L-CdLdPlus1})
are \rpn s of $C_d$ on~$l^p.$

In particular, \sfi\  does not imply \cog\  for \rpn s of~$C_d.$
\end{exa}

\begin{exa}\label{E-Mult}
Let $\rh \colon L_d \to L (l^p)$
be as in Example~\ref{E:LdRepOnlp}.
One immediately checks that there is
a \rpn\  $\sm \colon L_d \to L (l^p)$
such that for $j = 1, 2, \ldots, d,$
\[
\sm (s_j) = \tfrac{1}{2} \rh (s_j)
\andeqn
\sm (t_j) = 2 \rh (t_j).
\]
Then $\sm$ has the property
(part of the definition of being \sfi,
Definition~\ref{D:KindsOfReps1}(\ref{D:Repn-SFI}))
that for every $\ld \in \C^d,$
the element $\sm (s_{\ld})$ is a scalar multiple
of an isometry.
Moreover, $\| \sm (s_j) \| \leq 1$ for $j = 1, 2, \ldots, d.$
Also, the values of $\sm$ on $s_j t_j \in L_d$ are the same
as for the \sfi\  \rpn~$\rh.$
However, $\sm$ is not \sfi.

The Banach algebras ${\overline{\rh (L_d)}}$
and ${\overline{\sm (L_d)}}$
are isometrically isomorphic--in fact, they are equal.

Taking $u = \tfrac{1}{2},$
and following Lemma~\ref{L:MultLd},
we can write $\sm = \rh^u.$
\end{exa}

\begin{exa}\label{E-SumMult}
Let $\rh \colon L_d \to L (l^p)$
be as in Example~\ref{E:LdRepOnlp},
and let $\sm \colon L_d \to L (l^p)$
be as in Example~\ref{E-Mult}.
Let $\pi = \rh \oplus_p \sm$ be as in Remark~\ref{R-pSumRmk}.
Then $\| \pi (s_j) \| = 1$ for $j = 1, 2, \ldots, d,$
but $\pi$ is not \fis\  and not \cog.

It seems unlikely that
the Banach algebras ${\overline{\rh (L_d)}}$
and ${\overline{\pi (L_d)}}$
are isomorphic.

Taking
$v = \diag \big( 1, \tfrac{1}{2} \big)$
and using the notation of Lemma~\ref{L:MultLd},
we can write $\pi = ( \rh \oplus_p \rh)^v.$
\end{exa}

Example~\ref{E-SumMult}
was obtained using Lemma~\ref{L:MultLd}
with a diagonal scalar matrix.
In the following example,
we instead use a nondiagonalizable scalar matrix.
We do not know how the Banach algebras
obtained as the closures of the ranges differ.

\begin{exa}\label{E-Triang}
Let $\rh \colon L_d \to L (l^p)$
be as in Example~\ref{E:LdRepOnlp}.
Let $\rh \oplus_p \rh$ be as in Remark~\ref{R-pSumRmk}.
Set
\[
w = \left( \begin{matrix}
  1     &  1        \\
  0     &  1
\end{matrix} \right)
  \in L (l^p \oplus_p l^p).
\]
Let $\ta$ be the \rpn\  $\ta = (\rh \oplus_p \rh)^w$
of Lemma~\ref{L:MultLd}.
We have
\[
\ta (s_j) = \left( \begin{matrix}
  \rh (s_j)     &  \rh (s_j)        \\
          0     &  \rh (s_j)
\end{matrix} \right)
\andeqn
\ta (t_j) = \left( \begin{matrix}
  \rh (t_j)     & - \rh (t_j)        \\
          0     &  \rh (t_j)
\end{matrix} \right)
\]
for $j = 1, 2, \ldots, d.$
\end{exa}

We now turn to a different kind of modification
of our basic example,
using automorphisms of $L_d$ rather than Lemma~\ref{L:MultLd}.

\begin{exa}\label{E-LdSkew}
Let $d \in \{ 2, 3, 4, \ldots \}$ and let $p \in [1, \I].$
Set $\om = \exp (2 \pi i / d).$
Let $q$ be the conjugate exponent,
that is,
$\frac{1}{p} + \frac{1}{q} = 1.$
For $k = 1, 2, \ldots, d,$
define elements of $L_d$ by
\[
v_k = d^{- 1 / p} \sum_{j = 1}^d \om^{j k} s_j
\andeqn
w_k = d^{- 1 / q} \sum_{j = 1}^d \om^{- j k} t_j.
\]
(We take $d^{- 1 / p} = 1$ when $p = \I$ and $d^{- 1 / q} = 1$
when $p = 1.$)
It follows from Lemma~\ref{L-LCombProd}
that $w_j v_k = 1$ for $j = k$
and $w_j v_k = 0$ for $j \neq k.$
A computation shows that $\sum_{k = 1}^d v_k w_k = 1.$
Therefore there is an endomorphism $\ph \colon L_d \to L_d$
such that $\ph (s_k) = v_k$ and $\ph (t_k) = w_k$
for $k = 1, 2, \ldots, d.$
A computation shows that
\[
s_k = d^{- 1 / q} \sum_{j = 1}^d \om^{- j k} v_j
\andeqn
t_k = d^{- 1 / p} \sum_{j = 1}^d \om^{j k} w_j.
\]
Therefore $\ph$ is bijective.

It follows that whenever $\rh$ is a \rpn\  of~$L_d$
on a Banach space~$E,$
then $\rh \circ \ph$ is also a \rpn.
Moreover,
${\ov{(\rh \circ \ph) (L_d)}} = {\ov{\rh (L_d)}}.$

Now take $\rh$ to be as in Example~\ref{E:LdRepOnlp}.
We claim that $\rh \circ \ph$ is \cog\  and \sfi.
To prove this, it is convenient to introduce the matrix
\[
u = \big( u_{j, k} \big)_{j, k = 1}^d
  = d^{- 1 / p} \left( \begin{matrix}
  \om       &  \om^2      &  \ldots &  \om^{d - 1} &  1        \\
  \om^2     &  \om^4      &  \ldots &  \om^{d - 2} &  1        \\
  \vdots    &  \vdots     &  \ddots &  \vdots      &  \vdots   \\
\om^{d - 1} & \om^{d - 2} &  \ldots &  \om         &  1        \\
  1         &  1          &  \ldots &  1           &  1
\end{matrix} \right).
\]
We then have
\[
\ph (s_k) = \sum_{j = 1}^d u_{j, k} s_j.
\]
Therefore, for $\ld = (\ld_1, \ld_2, \ldots, \ld_d) \in \C^d,$
and following Definition~\ref{N-LComb},
\[
\ph (s_{\ld})
  = \sum_{k = 1}^d \sum_{j = 1}^d \ld_k u_{j, k} s_j
  = \sum_{j = 1}^d (u \ld)_j s_j
  = s_{u \ld}.
\]
So, for $\ld \in \C^d$ and $\xi \in l^p,$
Example~\ref{E:LdRepOnlp} implies that
\[
\| (\rh \circ \ph) (s_{\ld}) \xi \|_p = \| u \ld \|_p \| \xi \|_p.
\]
In particular, taking $\ld = \dt_j,$ for $p \neq \I$ we get
\[
\| (\rh \circ \ph) (s_{j}) \xi \|_p
  = \| u \dt_j \|_p \| \xi \|_p
  = \left( d^{- 1}
           \sssum{k = 1}{d} \big| \om^{j k} \big|^p \right)^{1/p}
         \| \xi \|_p
  = \| \xi \|_p.
\]
One also checks that $\| u \dt_j \|_p = 1$ when $p = \I.$
We conclude that $\rh \circ \ph$ is \sfi.

It remains to prove that $\rh \circ \ph$ is \cog.
Assume $p \neq 1, \I.$
Let $\xi = (\xi (1), \, \xi (2), \, \ldots ) \in l^p.$
Let $w_{d, j}$ be as in Example~\ref{E:LdRepOnlp}.
Then
\[
\sum_{j = 1}^d \| w_{d, j} \xi \|_p^p
 = \sum_{j = 1}^d \sum_{n = 1}^{\I} \big| \xi (d (n - 1) + j ) \big|^p
 = \sum_{m = 1}^{\I} | \xi (m) |^p
 = \| \xi \|_p^p.
\]
Now, for $k \in \{ 1, 2, \ldots, d \},$
we get,
using H\"{o}lder's inequality at the second step
and $| \om^{- j k} | = 1$ at the third step,
\begin{align*}
\| (\rh \circ \ph) (t_k) \xi \|_p^p
 & = d^{- p / q}
    \sssum{m = 1}{\I}
        \left| \sssum{j = 1}{d} \om^{- j k} (w_{d, j} \xi) (m) \right|^p
     \\
 & \leq d^{- p / q}
    \sssum{m = 1}{\I}
         \left( \sssum{j = 1}{d} | \om^{- j k} |^q \right)^{p / q}
         \left( \sssum{j = 1}{d}
                  \big| (w_{d, j} \xi) (m) \big|^p \right)
      \\
 & = d^{- p / q} d^{p / q} \sum_{j = 1}^d \| w_{d, j} \xi \|_p^p
   = \| \xi \|_p^p.
\end{align*}
Thus $\rh$ is \cog.
Easier calculations show that
$\rh$ is \cog\  when $p = 1$
and $p = \I$ as well.

For $p \in [1, \I) \SM \{ 2 \},$ we claim that,
unlike the \rpn\  $\rh$ of Example~\ref{E:01RepLd},
we have in general
$\| (\rh \circ \ph) (s_{\ld}) \xi \|_p \neq \| \ld \|_p \| \xi \|_p.$
Since we have already seen that
$\| (\rh \circ \ph) (s_{\ld}) \xi \|_p = \| u \ld \|_p \| \xi \|_p,$
it suffices to find some $\ld \in \C^d$
such that $\| u \ld \|_p \neq \| \ld \|_p.$

For an explicit easily checked example,
take $d = 2,$ $p = 3,$ and $\ld = (1, 2).$
Then
\[
\| \ld \|_p^p = 1 + 2^p = 9,
\,\,\,\,\,\,
u \ld = 2^{ - 1/ 3} (1, 3),
\andeqn
\| u \ld \|_p^p = \tfrac{1}{2} (1 + 3^p) = 14.
\]

For arbitrary $p \in [1, \I) \SM \{ 2 \}$ and arbitrary~$d,$
define $\sm \colon \MP{d}{p} \to L \big( l_d^p \big)$
by $\sm (a) = u a u^{-1}$
for $a \in \MP{d}{p}.$
Let $e_{j, k} \in \MP{d}{p}$ be the usual matrix unit.
Then one checks that $\sm (e_{1, 1})$ is not multiplication by
any characteristic function.
This violates condition~(\ref{T:SpatialRepsMd-5})
in Theorem~\ref{T:SpatialRepsMd} below.
Therefore $\sm$ is not isometric.
So $u$ is not isometric.
\end{exa}

In the following examples derived from Example~\ref{E-LdSkew},
we exclude $p = 2$ and $p = \I.$
When $p = 2,$ we do not get new behavior for the closure
of the image of~$L_d.$
We have not checked what happens when $p = \I.$

\begin{exa}\label{E-LdDblSkew}
Let $d \in \{ 2, 3, 4, \ldots \}$
and let $p \in [1, \I) \setminus \{ 2 \}.$
Let $\rh$ be as in Example~\ref{E:LdRepOnlp},
and let $\ph$ and $\rh \circ \ph$ be as in Example~\ref{E-LdSkew}.
Let $\pi = \rh \oplus_p (\rh \circ \ph),$
as in Remark~\ref{R-pSumRmk}.
Then $\pi$ is \fis\  and \cog\  because $\rh$ and $\rh \circ \ph$ are.
However, for $p \neq 2,$
Example~\ref{E-SumMult} and Example~\ref{E:LdRepOnlp}
show that there is $\ld \in \C^d$
and $\xi \in l^p$ such that
$\| \rh (s_{\ld} ) \xi \| \neq \| (\rh \circ \ph) (s_{\ld} ) \xi \|.$
Therefore $\pi (s_{\ld} )$ is not a scalar multiple of an isometry,
so $\pi$ is not \sfi.

We do not know whether the Banach algebras
${\ov{\rh (L_d)}}$ and ${\ov{\pi (L_d)}}$ are isomorphic.
\end{exa}

There are more complicated versions of Example~\ref{E-LdSkew}
which also give \rpn s
which are \sfi\  %
and \cog.
For example, one might split the generators into families
and treat each family separately
in the manner of Example~\ref{E-LdSkew}.
We give three special cases which are easy to write down
and which we want for specific purposes.

\begin{exa}\label{E-L3PartSkew}
Let $d = 3,$ let $p \in [1, \I) \setminus \{ 2 \},$
and let $\rh$ be as in Example~\ref{E:LdRepOnlp}
with this choice of~$d.$
There is a unique automorphism $\ps$ of $L_3$
such that
\[
\ps (s_1) = s_1,
\,\,\,\,\,\,
\ps (s_2) = 2^{- 1 / p} (s_2 + s_3),
\,\,\,\,\,\,
\ps (s_3) = 2^{- 1 / p} (s_2 - s_3),
\]
and
\[
\ps (t_1) = t_1,
\,\,\,\,\,\,
\ps (t_2) = 2^{- 1 / q} (t_2 + t_3),
\,\,\,\,\,\,
\ps (t_3) = 2^{- 1 / q} (t_2 - t_3).
\]
Then $\rh \circ \ps$
is \sfi,
with
\[
\| (\rh \circ \ps) (s_{\ld}) \xi \|_p
  = \big( | \ld_1 |^p + \tfrac{1}{2} | \ld_2 + \ld_3 |^p
           + \tfrac{1}{2} | \ld_2 - \ld_3 |^p \big)^{1/p}
                     \| \xi \|_p,
\]
and \cog.
Moreover, $\sm = \rh \oplus_p (\rh \circ \ps),$
as in Remark~\ref{R-pSumRmk},
is \fis\  and \cog,
but not \sfi.

Letting $\pi$ be as in Example~\ref{E-LdDblSkew} with $d = 3,$
we do not know whether the Banach algebras
${\ov{\sm (L_3)}}$ and ${\ov{\pi (L_3)}}$ are isomorphic.
We do note that ${\ov{\pi (L_3)}}$ has an isometric
automorphism which cyclically permutes the elements $\pi (s_j),$
but there is no isometric
automorphism of ${\ov{\sm (L_3)}}$ which does this.
Possibly there isn't even any continuous
automorphism of ${\ov{\sm (L_3)}}$ which does this.
\end{exa}

\begin{exa}\label{E-LdGpSkew}
Let $p \in [1, \I) \setminus \{ 2 \}.$
Let $d_0, n \in \{ 2, 3, \ldots \}.$
Set $d = n d_0.$
In $L_d,$ call the standard generators $s_{j, m}$ and $t_{j, m}$
for $j = 1, 2, \ldots, d_0$ and $m = 1, 2, \ldots, n.$
By reasoning similar to that of Example~\ref{E-LdSkew},
there is an automorphism $\af$ of $L_d$
such that, for $j = 1, 2, \ldots, d_0$ and $l = 1, 2, \ldots, n,$
we have
\begin{equation}\label{Eq:GpSkew}
\af (s_{k, m}) = d^{- 1 / p} \sum_{j = 1}^d \om^{j k} s_{j, m}
\andeqn
\af (t_{k, m}) = d^{- 1 / q} \sum_{j = 1}^d \om^{- j k} t_{j, m}.
\end{equation}
Take $\rh$ to be as in Example~\ref{E:LdRepOnlp}.
Then $\rh \circ \af$ is a \rpn\  of $L_d$ which is \sfi\  and \cog,
and for which one has
${\ov{(\rh \circ \af) (L_d)}} = {\ov{\rh (L_d)}}$
but, in general,
$\| (\rh \circ \af) (s_{\ld}) \xi \|_p \neq \| \ld \|_p \| \xi \|_p.$

We presume that $\rh \oplus_p (\rh \circ \af)$ is essentially
different from both $\rh$ and
the \rpn\  $\rh \oplus_p (\rh \circ \ph)$ of Example~\ref{E-LdDblSkew}.
However, we do not have a proof of anything like this.
\end{exa}

\begin{exa}\label{E-CdOpenSkew}
Let the notation be as in Example~\ref{E-LdGpSkew}.
Define a \hm\  $\bt \colon C_{d_0} \to L_d$ by
$\bt (s_j) = s_{j, 1}$ and $\bt (t_j) = t_{j, 1}$
for $j = 1, 2, \ldots, d_0.$
Then $\rh \circ \bt$ is a \rpn\  of $C_{d_0}$ which is \sfi\  and \cog,
and for which one has
$\| (\rh \circ \bt) (s_{\ld}) \xi \|_p = \| \ld \|_p \| \xi \|_p$
for $\ld \in \C^{d_0}$ and $\xi \in l^p.$
\end{exa}

The next example is the analog of Example~\ref{E-LdGpSkew}
for $L_{\I}.$

\begin{exa}\label{E-LInSkew}
Let $p \in [1, \I) \setminus \{ 2 \}$
and let $d_0 \in \{ 2, 3, \ldots \}.$
In $L_{\I},$ call the standard generators $s_{j, m}$ and $t_{j, m}$
for $j = 1, 2, \ldots, d_0$ and $m \in \N.$
Then there is an automorphism $\af$ of $L_{\I}$
defined by the formula~(\ref{Eq:GpSkew}),
but now for $j = 1, 2, \ldots, d_0$ and $m \in \N.$

Take $\rh$ to be as in Example~\ref{E:LIRepOnlp}.
Then $\rh \circ \af$ is a \rpn\  of $L_{\I}$ which is \sfi\  and \cog,
and for which one has
${\ov{(\rh \circ \af) (L_{\I})}} = {\ov{\rh (L_{\I})}}$
but, in general,
$\| (\rh \circ \af) (s_{\ld}) \xi \|_p \neq \| \ld \|_p \| \xi \|_p.$

We can then form the direct sum \rpn\  %
$\pi = \rh \oplus_p (\rh \circ \af)$
as in Remark~\ref{R-pSumRmk}.
We do not know whether ${\ov{\pi (L_{\I}) }}$
is isomorphic to ${\ov{\rh (L_{\I}) }}$
as a Banach algebra.
\end{exa}

\begin{exa}\label{E-LIOpenSkew}
Let the notation be as in Example~\ref{E-LInSkew}.
Define a \hm\  $\bt$ from $L_{\I}$
(with conventionally named generators)
to $L_{\I}$
(with generators named as in Example~\ref{E-LInSkew})
by $\bt (s_j) = s_{j, 1}$ and $\bt (t_j) = t_{j, 1}$
for $j = 1, 2, \ldots, d_0.$
Then $\rh \circ \bt$ is a \rpn\  of $L_{\I}$ which is \sfi\  and \cog,
and for which one has
$\| (\rh \circ \bt) (s_{\ld}) \xi \|_p = \| \ld \|_p \| \xi \|_p$
for $\ld \in \C^{\I}$ and $\xi \in l^p.$
We do not know whether ${\ov{(\rh \circ \bt) (L_{\I}) }}$
is isomorphic to ${\ov{\rh (L_{\I}) }}$
as a Banach algebra.
\end{exa}

\begin{rmk}\label{R-Fourier}
Example~\ref{E-LdSkew}
is based on the Fourier transform from functions on $\Z_d$
to functions on $\Z_d.$
We have not investigated the possibility of
using other finite abelian groups.
\end{rmk}

We finish this section with basic examples
of \rpn s on $L^p ([0, 1]).$

\begin{exa}\label{E:01RepLd}
Let $d \in \{ 2, 3, 4, \ldots \}$ and let $p \in [1, \I].$
For $j = 1, 2, \ldots, d,$ define a function
\[
g_j \colon [0, 1]
   \to \left[ \frac{j - 1}{d}, \, \frac{j}{d} \right]
\]
by
$g_j (x) = d^{-1} (j + x - 1).$
Then we claim that there is a unique \rpn\  %
$\rh \colon L_{d} \to L (L^p ([0, 1]))$
such that for $j = 1, 2, \ldots, d,$
$x \in [0, 1],$ and $\xi \in L^p ([0, 1])$
we have
\[
\big( \rh (s_j) \xi \big) (x)
   = \begin{cases}
     d^{1/p} \xi (g_j^{-1} (x))   &
            x \in \left[ \frac{j - 1}{d}, \, \frac{j}{d} \right]
               \\
     0   &  {\mbox{otherwise}}
    \end{cases}
\]
and
\[
\big( \rh (t_j) \xi \big) (x)
   = d^{- 1 / p} \xi (g_j (x)).
\]

The proof of the claim is a straightforward verification
that the proposed operators $\rh (s_j)$ and $\rh (t_j)$
satisfy the defining relations for $L_{d}$ in
Definition~\ref{D:Leavitt}.

It is easy to check that $\rh$ is \cog\  %
and is \fis.
In fact, $\rh$ is \sfi.
This follows from general theory.
(See Lemma~\ref{L-CnsqOfSp} below.)
But it is also easily checked,
using Remark~\ref{R:LpSuppNorm},
that for $\ld \in \C^d$ and $\xi \in L^p ([0, 1]),$
we have $\| \rh (s_{\ld}) \xi \|_p = \| \ld \|_p \| \xi \|_p.$
\end{exa}

\begin{exa}\label{E:01RepI1}
Let $p \in [1, \I].$
We give an example of a \rpn\  of $L_{\I}$ on $L^p ([0, 1]).$
In the following, if $p = \I$ then expressions with $p$
in the denominator are taken to be zero.

For $j \in \N,$ define a function
\[
f_j \colon [0, 1]
   \to \left[ \frac{1}{2^j}, \, \frac{1}{2^{j - 1}} \right]
\]
by
$ f_j (x) = 2^{- j} (1 + x).$
Then we claim that there is a unique \rpn\  %
$\rh \colon L_{\I} \to L (L^p ([0, 1]))$
such that for $j \in \N,$ $x \in [0, 1],$ and $\xi \in L^p ([0, 1])$
we have
\[
\big( \rh (s_j) \xi \big) (x)
   = \begin{cases}
     2^{j/p} \xi \big( f_j^{-1} (x) \big)   &
            x \in \left[ \frac{1}{2^j}, \, \frac{1}{2^{j - 1}} \right]
               \\
     0   &  {\mbox{otherwise}}
    \end{cases}
\]
and
\[
\big( \rh (t_j) \xi \big) (x)
   = 2^{- j / p} \xi (f_j (x)).
\]
The proof of the claim is a straightforward verification
that the proposed operators $\rh (s_j)$ and $\rh (t_j)$
satisfy the defining relations for $L_{\I}$ in
Definition~\ref{D:LInfty}.

One can check that $\rh$ is \cog\  and \fis.
(We will see in Lemma~\ref{L-CnsqOfSp} below
that it is in fact \sfi.)
\end{exa}

\section{Boolean $\sigma$-algebras}\label{Sec:Lp1}

\indent
In this section and the next,
we describe the background for the characterization,
due to Lamperti,
of isometries on $L^p (X, \mu).$
Parts of this material can be found in
Section~1 of Chapter~X of~\cite{Db}
and in Lamperti's paper~\cite{Lp},
but these references contain only enough to state and
prove the characterization theorem,
not enough to make serious use of it.
A somewhat more systematic presentation can be found in
Chapter~15 of~\cite{Ry},
especially Section~2,
and we use the terminology from there,
but we need more than is there,
and the form in which it is presented there makes citation
of specific results difficult.

For this section,
on abstract Boolean $\sm$-algebras,
we follow~\cite{Kp}
and just state the basic results.
(However, we use notation more suggestive of
unions and intersections than the notation of~\cite{Kp}:
our $E \vee F$ is $E + F$ there,
our $E \wedge F$ is $E \cdot F$ there,
and our $E'$ is $- E$ there.)

\begin{dfn}[Definition~1.1 of~\cite{Kp}]\label{D:BooleanAlg}
A {\emph{Boolean algebra}} is a set $\cB$ with two
commutative associative
binary operations $(E, F) \mapsto E \vee F$
and $(E, F) \mapsto E \wedge F,$
a unary operation $E \mapsto E',$
and distinguished elements $0$ and~$1,$
satisfying the following for all $E, F, G \in \cB$:
\begin{enumerate}
\item\label{D:BooleanAlg-01}
$E \vee (E \wedge F) = E$ and $E \wedge (E \vee F) = E.$
\item\label{D:BooleanAlg-03}
$E \wedge (F \vee G) = (E \wedge F) \vee (E \wedge G)$
and
$E \vee (F \wedge G) = (E \vee F) \wedge (E \vee G).$
\item\label{D:BooleanAlg-05}
$E \vee E' = 1$ and $E \wedge E' = 0.$
\end{enumerate}
\end{dfn}

Subalgebras and \hm s of Boolean algebras
have the obvious meanings.
(See Definitions 1.7 and~1.3 of~\cite{Kp}.)

\begin{exa}\label{R:Boolean}
The standard example is the power set ${\mathcal{P}} (X)$ of a set~$X,$
with
\[
E \vee F = E \cup F,
\,\,\,\,\,\,
E \wedge F = E \cap F,
\,\,\,\,\,\,
E' = X \SM E,
\,\,\,\,\,\,
0 = \E.
\andeqn
1 = X.
\]
\end{exa}

We therefore refer to the operations in a Boolean algebra as
union, intersection, and complementation.

Proofs that the axioms imply the other expected properties
can be found in Section~1.5 of~\cite{Kp}.
However, for the purposes of finite algebraic manipulations,
it is easier to rely on the following theorem,
which implies that all finite identities which hold among sets
also hold in any Boolean algebra.

\begin{thm}[Theorem~2.1 of~\cite{Kp}]\label{T-BoolRep}
Let $\cB$ be a Boolean algebra.
Then there exists a set $X$ and an isomorphism of $\cB$
with a Boolean subalgebra of ${\mathcal{P}} (X).$
\end{thm}

\begin{dfn}\label{D:BooleanOrder}
Let $\cB$ be a Boolean algebra,
and let $E, F \in \cB.$
We define $E \leq F$ to mean $E \wedge F = E.$
We say that $E$ and $F$ are {\emph{disjoint}}
if $E \wedge F = 0.$
We define
the {\emph{symmetric difference}} of $E$ and $F$ to be
\[
E \bigtriangleup F = (E \wedge F') \vee (E' \wedge F).
\]
\end{dfn}

In ${\mathcal{P}} (X),$
disjointness means that
the intersection is empty,
$E \leq F$ means $E \subset F,$
and symmetric difference has its usual meaning.
Thus, by Theorem~\ref{T-BoolRep},
the relation of Definition~\ref{D:BooleanOrder}
is a partial order on~$\cB,$
in which $0$ is the least element and $1$ is the greatest element.
In particular, if $E \leq F$ and $F \leq E,$ then $E = F.$

\begin{dfn}[Section~15.2 of~\cite{Ry}]\label{D:BSMA}
A {\emph{Boolean $\sm$-algebra}} is a Boolean algebra~$\cB$
in which whenever $E_1, E_2, \ldots \in \cB,$
then there is a least element $E \in \cB$
such that $E_n \leq E$ for all~$n \in \N.$
\end{dfn}

In Definition~1.28 of~\cite{Kp},
a Boolean $\sm$-algebra is called a $\sm$-complete Boolean algebra.
(It is also required that greatest lower bounds
of countable collections exist,
but this follows from Definition~\ref{D:BSMA}
by complementation.)

\begin{dfn}\label{D:CtblUnionNtn}
Let $\cB$ be a Boolean $\sm$-algebra.
The element $E$ in Definition~\ref{D:CtblUnionNtn}
is denoted $\bigvee_{n = 1}^{\I} E_n.$
We call it the {\emph{union}} of the~$E_n.$
We further define
$\bigwedge_{n = 1}^{\I} E_n
  = \left( \bigvee_{n = 1}^{\I} E_n' \right)',$
and call it the {\emph{intersection}} of the~$E_n.$
\end{dfn}

The operations $\bigvee_{n = 1}^{\I} E_n$
and $\bigwedge_{n = 1}^{\I} E_n$
behave as expected.
(This is not proved in~\cite{Ry},
but is proved in~\cite{Kp}.)

\begin{lem}\label{L:PropInfUnion}
Let $\cB$ be a Boolean $\sm$-algebra.
\begin{enumerate}
\item\label{L:PropInfUnion-0}
Let $E_1, E_2, \ldots \in \cB.$
Then
\[
\left( \bigvee_{n = 1}^{\I} E_n \right)'
  = \bigwedge_{n = 1}^{\I} E_n'
\andeqn
\left( \bigwedge_{n = 1}^{\I} E_n \right)'
  = \bigvee_{n = 1}^{\I} E_n'.
\]
\item\label{L:PropInfUnion-1}
Let $E_1, E_2, \ldots, F \in \cB.$
Then
\[
F \vee \bigvee_{n = 1}^{\I} E_n = \bigvee_{n = 1}^{\I} (F \vee E_n)
\andeqn
F \wedge \bigvee_{n = 1}^{\I} E_n = \bigvee_{n = 1}^{\I} (F \wedge E_n),
\]
and
\[
F \vee \bigwedge_{n = 1}^{\I} E_n = \bigwedge_{n = 1}^{\I} (F \vee E_n)
\andeqn
F \wedge \bigwedge_{n = 1}^{\I} E_n
 = \bigwedge_{n = 1}^{\I} (F \wedge E_n).
\]
\item\label{L:PropInfUnion-2}
Let $E_{m, n} \in \cB$ for $m, n \in \N.$
Then
\[
\bigvee_{m = 1}^{\I} \bigvee_{n = 1}^{\I} E_{m, n}
  = \bigvee_{m = 1}^{\I} \bigvee_{n = 1}^{\I} E_{n, m}
\andeqn
\bigwedge_{m = 1}^{\I} \bigwedge_{n = 1}^{\I} E_{m, n}
  = \bigwedge_{m = 1}^{\I} \bigwedge_{n = 1}^{\I} E_{n, m}.
\]
\end{enumerate}
\end{lem}

\begin{proof}
These statements all follow easily from the various parts
of Lemma 1.33 of~\cite{Kp},
or from their duals (stated afterwards).
\end{proof}

\begin{exa}\label{E:SGAIsBoolean}
Let $X$ be a set.
Then a \sga\  of subsets of~$X$ is a Boolean \sga.
\end{exa}

We introduce $\sm$-homomorphisms and $\sm$-ideals.
In~\cite{Kp},
they are called $\sm$-complete homomorphisms
(Definition~5.1 of~\cite{Kp})
and $\sm$-complete ideals
(Definition~5.19 of~\cite{Kp}).

\begin{dfn}\label{D:shm}
Let $\cB$ and $\cC$ be Boolean \sga s.
A {\emph{$\sm$-homomorphism}} from $\cB$ to $\cC$
is a function $S \colon \cB \to \cC$ which is a
\hm\  of Boolean algebras in the obvious sense,
and moreover such that
whenever $E_1, E_2, \ldots \in \cB$
then
\[
S \left( \bigvee_{n = 1}^{\I} E_n \right)
 = \bigvee_{n = 1}^{\I} S (E_n).
\]
\end{dfn}

\begin{lem}\label{L:ShmProp}
The \shm s of Definition~\ref{D:shm} have the following properties.
\begin{enumerate}
\item\label{L:ShmProp-1}
The composition of two \shm s is a \shm.
\item\label{L:ShmProp-2}
A \shm\  preserves order and disjointness
(as in Definition~\ref{D:BooleanOrder}).
\item\label{L:ShmProp-3}
A \shm\  $S$ is injective \ifo\  $S (E) = 0$ implies $E = 0.$
\end{enumerate}
\end{lem}

\begin{proof}
The first part is obvious,
and the second is easy.
The nontrivial direction of the third part
is proved by considering symmetric differences.
(See Lemma~5.3 of~\cite{Kp}.)
\end{proof}

\begin{dfn}\label{D:Ideal}
Let $\cB$ be a Boolean $\sigma$-algebra.
A {\emph{$\sigma$-ideal}} in ${\mathcal{B}}$ is a subset
${\mathcal{N}} \subset {\mathcal{B}}$ which is closed under
countable unions,
such that $0 \in {\mathcal{N}},$
and such that whenever $E \in {\mathcal{B}}$
and $F \in {\mathcal{N}}$ satisfy $E \subset F,$
then $E \in {\mathcal{N}}.$
\end{dfn}

The standard example is as follows.

\begin{exa}\label{E:Null}
Let $X$ be a set,
and let ${\mathcal{B}}$ be a $\sigma$-algebra on~$X.$
Let $\mu$ be a measure with domain~${\mathcal{B}}.$
Then
\[
{\mathcal{N}} (\mu) = \{ E \in {\mathcal{B}} \colon \mu (E) = 0 \}
\]
is a $\sigma$-ideal in ${\mathcal{B}}.$
\end{exa}

\begin{dfn}\label{D:Quot}
Let $\cB$ be a Boolean \sga.
If ${\mathcal{N}} \subset {\mathcal{B}}$ is a $\sigma$-ideal,
we define
$E \sim F$ to mean $E \bigtriangleup F \in {\mathcal{N}}.$
We define ${\mathcal{B}} / {\mathcal{N}}$ to be the
quotient set of ${\mathcal{B}}$ by~$\sim.$
If $E \in {\mathcal{B}},$ we write $[E]$ for its
image in ${\mathcal{B}} / {\mathcal{N}}.$
\end{dfn}

\begin{lem}\label{L:Quot}
Let the notation be as in Definition~\ref{D:Quot}.
Then $\sim$ is an equivalence relation,
the obvious induced operations on ${\mathcal{B}} / {\mathcal{N}}$
are well defined and
make ${\mathcal{B}} / {\mathcal{N}}$ a Boolean \sga,
and $E \mapsto [E]$ is a surjective \shm.
\end{lem}

\begin{proof}
This is outlined in Section~15.2 of~\cite{Ry},
and is contained in
Lemma~5.22 of~\cite{Kp} and the remark after its proof.
\end{proof}

We finish this section with a lemma which will be needed later.

\begin{lem}\label{L:InvImDisj}
Let $\cB$ be a Boolean \sga,
and let ${\mathcal{N}} \subset {\mathcal{B}}$ be a $\sigma$-ideal.
Suppose $E_1, E_2, \ldots \in {\mathcal{B}} / {\mathcal{N}}$
are pairwise disjoint (Definition~\ref{D:BooleanOrder}).
Then there exist disjoint elements
$E_1^{(0)}, E_2^{(0)}, \ldots \in {\mathcal{B}}$
such that $\big[ E_n^{(0)} \big] = E_n$ for all $n \in \N.$
\end{lem}

\begin{proof}
We first show that if $E, F \in {\mathcal{B}} / {\mathcal{N}}$
are disjoint,
then there exist disjoint $E_0, F_0 \in \cB$
such that $[E_0] = E$ and $F_0 = F.$
Choose any $E_0 \in \cB$ such that $[E_0] = E$ and
any $K \in \cB$ such that $[K] = F.$
Then $[K \wedge E_0] = 0,$
so we can take $F_0 = K \wedge E_0'.$

Now we prove the statement.
For $m \neq n$ choose disjoint elements
$S_{m, n}, T_{m, n} \in \cB$ such that
$[S_{m, n}] = E_m$ and $[T_{m, n}] = E_n.$
Then take
\[
E_n^{(0)}
 = \bigwedge_{k \in \N \SM \{ n \} } (S_{n, k} \wedge T_{k, n}).
\]
This completes the proof.
\end{proof}

\section{Measurable set transformations}\label{Sec:MST}

In this section,
we consider \msp s
$\XBM$ and $\YCN,$
and a suitable $\sm$-homomorphism
$S \colon
 {\mathcal{B}} / {\mathcal{N}} (\mu)
  \to {\mathcal{C}} / {\mathcal{N}} (\nu).$
We describe how to use $S$ to produce maps on
measurable functions mod equality almost everywhere
and on measures.
The idea is not new;
it can be found in~\cite{Lp}
and (for functions, under stronger hypotheses)
in Chapter~X of~\cite{Db}.
However, we need much more than can be found in these references.

First,
we give some definitions and notation which will be
frequently used later.

\begin{dfn}\label{D:L0}
Let $(X, {\mathcal{B}}, \mu)$ be a measure space.
We denote by $L^0 (X, \mu)$ the vector space of all
complex valued measurable functions on~$X,$
mod the functions which vanish almost everywhere.
We follow the usual convention of treating elements
of $L^0 (X, \mu)$ as functions when convenient.
If $E \in {\mathcal{B}},$
we denote by $\ch_E$ the characteristic function of~$E,$
which is a well defined element of~$L^0 (X, \mu).$
\end{dfn}

\begin{dfn}\label{D:SupportQuot}
Let $\XBM$ be a \msp,
and let ${\mathcal{N}} (\mu)$ be as in Example~\ref{E:Null}.
For $\xi \in L^0 (X, \mu),$
we define the {\emph{support}} of $\xi$
to be the element of ${\mathcal{B}} / {\mathcal{N}} (\mu)$
given as follows.
Choose any actual function $\xi_0 \colon X \to \C$
whose class in $L^0 (X, \mu)$ is~$\xi,$
and set
\[
\supp (\xi)
  = \big[ \big\{ x \in X \colon \xi_0 (x) \neq 0 \big\} \big].
\]

Further, for $E \in {\mathcal{B}} / {\mathcal{N}} (\mu),$
define $\ch_E \in L^0 (X, \mu)$ to be the class of $\ch_{E_0}$
for any $E_0 \in \cB$ with $[E_0] = E.$
\end{dfn}

Definition~\ref{D:SupportQuot} is a natural kind of strengthening
of Definition~\ref{D:LpSupp}.

\begin{rmk}\label{R:SuppWellDef}
In Definition~\ref{D:SupportQuot},
it is easy to check that $\supp (\xi)$ and $\ch_E$
are well defined.
Clearly $\xi$ is supported in~$E,$
in the sense of Definition~\ref{D:LpSupp},
\ifo\  $[E] \geq \supp (\xi).$
\end{rmk}

The following definition and constructions based on it are adapted from
the beginning of Section~1 in Chapter~X of~\cite{Db}.

\begin{dfn}\label{D:SetTrans}
Let $(X, {\mathcal{B}}, \mu)$ and $(Y, {\mathcal{C}}, \nu)$
be measure spaces,
and let ${\mathcal{N}} (\mu)$ and ${\mathcal{N}} (\nu)$
be as in Example~\ref{E:Null}.
A {\emph{measurable set transformation}} from
$(X, {\mathcal{B}}, \mu)$ to $(Y, {\mathcal{C}}, \nu)$
is a \shm\  (in the sense of Definition~\ref{D:shm})
$S \colon {\mathcal{B}} / {\mathcal{N}} (\mu)
            \to {\mathcal{C}} / {\mathcal{N}} (\nu).$
By abuse of notation, we write
$S \colon \XBM \to \YCN.$
For $E \in \cB,$
we also write, by abuse of notation,
$S (E)$ for some choice of $F \in \cC$ such that $[F] = S ([E]).$
Injectivity and surjectivity always refer to properties of the map
$S \colon {\mathcal{B}} / {\mathcal{N}} (\mu)
            \to {\mathcal{C}} / {\mathcal{N}} (\nu).$
We denote by ${\mathrm{ran}}_Y (S)$
the collection of all subsets $F \in {\mathcal{C}}$ such that
$[F] = S ([E])$ for some $E \in {\mathcal{B}}.$
\end{dfn}

In~\cite{Db},
a set transformation is taken to be a multivalued map
from ${\mathcal{B}}$ to~${\mathcal{C}},$
with possible values differing only up to a set of measure zero,
and which preserves the appropriate set operations,
that is, which defines a \shm\  from
${\mathcal{B}} / {\mathcal{N}} (\mu)$
to ${\mathcal{C}} / {\mathcal{N}} (\nu).$
Also, $(X, {\mathcal{B}}, \mu) = (Y, {\mathcal{C}}, \nu),$
and the map is required to be measure preserving.
In our situation, if $S$ is injective,
then at least the type of transformation in~\cite{Db}
sends sets of nonzero measure to sets of nonzero measure.

At the beginning of Section~3 of~\cite{Lp},
what we call an injective \shm\  %
is called a regular set isomorphism.
The definition specifies a map of sets modulo null sets,
although without formally defining an appropriate domain.
It omits the requirement (implicit above) that $S (X) = Y,$
but this is easily restored by replacing $Y$ by $S (X).$
We do not use the term ``regular set isomorphism''
because such maps need not be surjective,
and we need to consider cases in which they are not.
Also, we are more careful with the formalism because
we need to use such maps systematically.

\begin{lem}\label{L:RanIsSgAlg}
Let the notation be as in Definition~\ref{D:SetTrans}.
Then ${\mathrm{ran}}_Y (S)$ is a sub-$\sm$-algebra of~${\mathcal{B}}.$
\end{lem}

\begin{proof}
The proof is easy, and is omitted.
\end{proof}

\begin{prp}\label{P:SStar}
Let $(X, {\mathcal{B}}, \mu)$ and $(Y, {\mathcal{C}}, \nu)$
be measure spaces.
Let
$S \colon (X, {\mathcal{B}}, \mu) \to (Y, {\mathcal{C}}, \nu)$
be a \mst\  %
(Definition~\ref{D:SetTrans}).
Then (following the notation of Definition~\ref{D:L0})
there is a unique linear map
$S_* \colon L^0 (X, \mu) \to L^0 (Y, \nu)$
such that:
\begin{enumerate}
\item\label{P:SStar-1}
$S_* (\ch_E) = \ch_{S (E)}$
for all $E \in {\mathcal{B}} / {\mathcal{N}} (\mu).$
\item\label{P:SStar-2}
Whenever $(\xi_n)_{n \in \N}$ is a sequence of measurable
functions on~$X$ which converges pointwise almost everywhere~$[\mu]$
to~$\xi,$
then $S_* (\xi_n) \to S_* (\xi)$ pointwise almost everywhere~$[\nu].$
\setcounter{TmpEnumi}{\value{enumi}}
\end{enumerate}
Moreover:
\begin{enumerate}
\setcounter{enumi}{\value{TmpEnumi}}
\item\label{P:SStar-2x}
Let $n \in \N,$
let $f \colon \C^n \to \C$ be \ct,
and let $\xi_1, \xi_2, \ldots, \xi_n \in L^0 (X, \mu).$
Set
$\et (x) = f \big( \xi_1 (x), \, \xi_2 (x), \, \ldots, \xi_n (x) \big)$
for $x \in X.$
Then
\[
S_* (\et) (y)
 = f \big( S_* (\xi_1) (y), \, S_* (\xi_2) (y), \, \ldots,
      \, S_* (\xi_n) (y) \big)
\]
for almost all $y \in Y.$
(In particular, $S_*$ preserves products and preserves
arbitrary positive powers of the absolute value of a function.)
\item\label{P:SStar-3}
The range of $S_*$ is
$L^0 \big( Y, \, \nu |_{{\mathrm{ran}} (S)} \big).$
\item\label{P:SStar-4}
$S_*$ is injective \ifo\  $S$ is injective.
\item\label{P:SStar-5}
If $\xi \in L^0 (X, \mu)$ and $B \subset \C$ is a Borel set,
then $S ([\xi^{-1} (B)]) = \big[ S_* (\xi)^{-1} (B) \big].$
\item\label{P:SStar-6}
Let $(Z, {\mathcal{D}}, \ld)$ be another measure space,
and let $T$ be a \mst\  from
$(Y, {\mathcal{C}}, \nu)$
to $(Z, {\mathcal{D}}, \ld).$
Then $(T \circ S)_* = T_* \circ S_*.$
\end{enumerate}
\end{prp}

For the proof, we use the following well known lemma.

\begin{lem}\label{L-OrdSets}
Let $\XBM$ be a measure space,
and let $(E_{\af})_{\af \in \Q}$ be a family of measurable sets
such that $E_{\af} \subset E_{\bt}$ whenever $\af \leq \bt,$
and such that
\[
\bigcap_{\af \in \Q} E_{\af} = \E
\andeqn
\bigcup_{\af \in \Q} E_{\af} = X.
\]
Then there exists a unique measurable function
$\xi \colon X \to \R$ such that for all $\af \in \Q$
we have
\[
\{ x \in X \colon \xi (x) < \af \}
 \subset E_{\af}
 \subset \{ x \in X \colon \xi (x) \leq \af \}.
\]
Moreover, for $\af \in \R$ we have
\[
\{ x \in X \colon \xi (x) < \af \}
 = \bigcup_{\bt \in \Q \cap (- \I, \af)} E_{\bt}
\]
and
\[
\{ x \in X \colon \xi (x) > \af \}
 = \bigcup_{\bt \in \Q \cap (\af, \I)} (X \SM E_{\bt}).
\]
\end{lem}

\begin{proof}
The first part of the statement is Lemma~9 in Chapter~11 of~\cite{Ry}.
The last two equations follow easily.
\end{proof}

\begin{proof}[Proof of Proposition~\ref{P:SStar}]
We prove uniqueness.
Let $T_1, T_2 \colon L^0 (X, \mu) \to L^0 (Y, \nu)$
be linear and
satisfy (\ref{P:SStar-1}) and~(\ref{P:SStar-2}).
It follows from linearity and~(\ref{P:SStar-1})
that whenever $\xi \in L^0 (X, \mu)$ is a simple function,
we have $T_1 (\xi) = T_2 (\xi)$ almost everywhere~$[\nu].$
Since every measurable function is a pointwise limit of
simple functions,
this conclusion in fact holds for all $\xi \in L^0 (X, \mu).$

We now prove existence.
We follow the construction on page 454 of~\cite{Db},
which is done under stronger hypotheses.
The verification of linearity will be done as
a special case of~(\ref{P:SStar-2x}).

First let $\xi \in L^0 (X, \mu)$ be real valued.
For $\af \in \Q$ define
\begin{equation}\label{Eq:EA}
E_{\af} = \{ x \in X \colon \xi (x) < \af \}.
\end{equation}
Then the sets $E_{\af}$ satisfy the hypotheses of
Lemma~\ref{L-OrdSets}.
For each $\af \in \Q,$
choose a set $F_{\af}^{(0)} \in \cC$
such that
\begin{equation}\label{Eq:FAZ}
\big[ F_{\af}^{(0)} \big] = S ( [E_{\af}] ).
\end{equation}
Set
\begin{equation}\label{Eq:223D}
D_0 = \left( \bigcup_{\af < \bt}
         F_{\bt}^{(0)} \cap \big( Y \SM F_{\af}^{(0)} \big) \right)
       \cup \left( \bigcap_{\af} F_{\af}^{(0)} \right)
       \cup \left( Y \SM \bigcup_{\af} F_{\af}^{(0)} \right).
\end{equation}
Then $\nu (D_0) = 0$
and the sets $F_{\af}^{(0)} \cap (Y \SM D_0) \subset Y$ satisfy
the the hypotheses of
Lemma~\ref{L-OrdSets}
with $\big( Y \SM D_0, \, \cC |_{Y \SM D_0}, \, \nu |_{Y \SM D_0} \big)$
in place of~$\XBM.$

Choose any set $D \in \cC$ such that $D_0 \subset D$
and $\nu (D) = 0.$
For $\af \in \Q$ define
\begin{equation}\label{Eq:FA}
F_{\af} =  \begin{cases}
   F_{\af}^{(0)} \cup D & \af > 0
       \\
   F_{\af}^{(0)} \cap (Y \SM D) & \af \leq 0.
\end{cases}
\end{equation}
Then the sets $F_{\af} \subset Y$ satisfy
the the hypotheses of
Lemma~\ref{L-OrdSets}
with $\YCN$
in place of~$\XBM.$
So Lemma~\ref{L-OrdSets}
provides a measurable function $\et \colon Y \to \R$
such that for all $\af \in \Q,$ we have
\begin{equation}\label{Eq:223St}
\{ y \in Y \colon \et (y) < \af \}
 \subset F_{\af}
 \subset \{ y \in Y \colon \et (y) \leq \af \}.
\end{equation}
We define $S_* (\xi) = \et.$

We claim that, up to equality almost everywhere~$[\nu],$
the function $\et$ does not depend on the choices made
in its construction.
First, if we replace $D$ with~${\widetilde{D}},$
and call the new function~${\widetilde{\et}},$
then ${\widetilde{\et}} = \et$ off $D \cup {\widetilde{D}},$
and $\nu \big( D \cup {\widetilde{D}} \big) = 0.$
Now suppose that $F_{\af}^{(0)}$ is replaced by some other set
${\widetilde{F}}_{\af}^{(0)} \in \cC$
such that $\big[ {\widetilde{F}}_{\af}^{(0)} \big] = S ( [E_{\af}] ).$
Then
$\nu \big( F_{\af}^{(0)}
       \bigtriangleup {\widetilde{F}}_{\af}^{(0)} \big)
     = 0$
for all $\af \in \Q.$
Let $D_0$ be as in~(\ref{Eq:223D}),
and define ${\widetilde{D}}_0$ analogously,
using ${\widetilde{F}}_{\af}^{(0)}$ in place of $F_{\af}^{(0)}.$
Let $\et$ and ${\widetilde{\et}}$
be the functions resulting from our construction,
with the choice
\[
D = {\widetilde{D}}
 = D_0 \cup {\widetilde{D}}_0
    \cup \bigcup_{\af \in \Q} \big( F_{\af}^{(0)}
                   \bigtriangleup {\widetilde{F}}_{\af}^{(0)} \big).
\]
We have
\[
F_{\af}^{(0)} \cap (Y \SM D)
 = {\widetilde{F}}_{\af}^{(0)} \cap (Y \SM D)
\]
for all $\af \in \Q,$
so also
\[
F_{\af}^{(0)} \cup D = {\widetilde{F}}_{\af}^{(0)} \cup D
\]
for all $\af \in \Q.$
It follows that $\et = {\widetilde{\et}}.$
This completes the proof of the claim.

We next claim that if $\xi = \sum_{k = 1}^n \gm_k \ch_{B_k},$
with $B_1, B_2, \ldots, B_n \in \cB,$
then (using the abuse of notation from Definition~\ref{D:SetTrans})
we have
\[
S_* (\xi) = \sum_{k = 1}^n \gm_k \ch_{S (B_k)}.
\]
To prove this, first suppose that
$B_1, B_2, \ldots, B_n$ are disjoint.
Using Lemma~\ref{L:InvImDisj},
choose disjoint sets $C_1, C_2, \ldots, C_n \in \cC$
such that $[C_k] = S ([B_k])$ for $k = 1, 2, \ldots, n.$
For $\af \in \Q$ choose
$F_{\af}^{(0)} = \bigcup_{\gm_k < \af} C_k.$
Then the set $D_0$ in the construction of~$\et$ is empty.
Taking $D = \E$ gives
\[
S_* (\xi) = \sum_{k = 1}^n \gm_k \ch_{C_k},
\]
as desired.
The general case is easily reduced to the disjoint case by
using instead of the sets $B_k$ all possible nonempty intersections
$E_1 \cap E_2 \cap \cdots \cap E_n$
in which $E_k$ is either $B_k$ or $X \SM B_k.$
This proves the claim.

It follows immediately that (\ref{P:SStar-2x})~holds
when $\xi_1, \xi_2, \ldots, \xi_n$ are real simple functions
and $f$ is a \cfn\  from $\R^n$ to~$\R.$

We now prove~(\ref{P:SStar-2}).
Let $(\xi_n)_{n \in \N}$ be a sequence of real measurable functions
on~$X,$
and suppose that $\xi_n (x) \to \xi (x)$ almost everywhere~$[\mu].$
Changing the $\xi_n$ and $\xi$ on a set of measure zero,
we may assume that $\xi_n (x) \to \xi (x)$ for all $x \in X.$
For $n \in \N$ and $\af \in \Q,$
let $E_{n, \af},$ $F_{n, \af}^{(0)},$ and $D_n^{(0)}$
be the sets of (\ref{Eq:EA}), (\ref{Eq:FAZ}), and~(\ref{Eq:223D})
for the construction of $S_* (\xi_n),$
and let $E_{\af},$ $F_{\af}^{(0)},$ and $D_0$
be the corresponding sets for~$\xi.$
Let $\af, \bt \in \Q$ satisfy $\af > \bt.$
Since $\xi (x) < \af$ implies $\limsup_{n \to \I} \xi_n (x) < \af$
and $\xi (x) \geq \af$ implies $\liminf_{n \to \I} \xi_n (x) > \bt,$
we get
\[
E_{\af} \subset \bigcup_{n = 1}^{\I} \bigcap_{m = n}^{\I} E_{m, \af}
\andeqn
X \SM E_{\af}
  \subset \bigcup_{n = 1}^{\I} \bigcap_{m = n}^{\I} (X \SM E_{m, \bt}).
\]
Set
\[
B_{\af} = F_{\af}^{(0)} \cap \left( Y \SM \bigcup_{n = 1}^{\I}
          \bigcap_{m = n}^{\I} F_{m, \af}^{(0)} \right)
\]
and
\[
C_{\af, \bt} = \big( Y \SM F_{\af}^{(0)} \big)
    \cap \left( Y \SM \bigcup_{n = 1}^{\I}
          \bigcap_{m = n}^{\I}
               \big( Y \SM F_{m, \bt}^{(0)} \big) \right).
\]
Since $S$ is a $\sm$-\hm,
$\nu (B_{\af}) = 0$ and $\nu ( C_{\af, \bt} ) = 0.$

Set
\[
D = D_0 \cup \left( \bigcup_{n = 1}^{\I} D_n^{(0)} \right)
        \cup \left( \bigcup_{\af \in \Q} B_{\af} \right)
        \cup \left( \bigcup_{\af > \bt} C_{\af, \bt} \right).
\]
Then $\nu (D) = 0.$
Using this choice of~$D,$
define sets $F_{n, \af}$ as in~(\ref{Eq:FA}) for $\xi_n$
and $F_{\af}$ for~$\xi,$
and let $\et_n$ and $\et$ be the representatives the construction
gives for $S_* (\xi_n)$ and $S_* (\xi).$

Let $y \in Y \SM D.$
We claim that $\et_n (y) \to \et (y).$
This will complete the proof of~(\ref{P:SStar-2}) for real functions.

To prove the claim, let $\ep > 0,$
and choose $\af, \bt \in \Q$ such that
\[
\et (y) - \ep < \bt < \et (y) < \af < \et (y) + \ep.
\]
Then $y \in F_{\af}$ by~(\ref{Eq:223St}).
Since $y \not\in B_{\af},$
we have
\[
y \in \bigcup_{n = 1}^{\I} \bigcap_{m = n}^{\I} F_{m, \af},
\]
so~(\ref{Eq:223St}) for $\et_n$ implies that there is $n_1 \in \N$
such that for all $m \geq n_1$ we have $\et_m (y) \leq \af.$
Then also $\et_m (y) < \et (y) + \ep.$
Furthermore,
(\ref{Eq:223St})~implies $y \not\in F_{\bt},$
and $y \not\in C_{\af, \bt},$
so
\[
y \in \bigcup_{n = 1}^{\I} \bigcap_{m = n}^{\I} ( Y \SM F_{m, \bt} ).
\]
Using~(\ref{Eq:223St}) for $\et_n,$
we get $n_2 \in \N$ such that for all $m \geq n_2$
we have $\et_m (y) \geq \bt > \et (y) - \ep.$
We have thus shown that $\et_n (y) \to \et (y),$ as desired.

It now follows that
(\ref{P:SStar-2x})~holds
whenever $\xi_1, \xi_2, \ldots, \xi_n$ are real measurable functions
and $f$ is a \cfn\  from $\R^n$ to~$\R,$
because $\xi_1, \xi_2, \ldots, \xi_n$
are pointwise limits of real simple functions.

We define $S_*$ on complex functions~$\xi$
by $S_* (\xi)
 = S_* ( {\mathrm{Re}} (\xi)) + i S_* ( {\mathrm{Im}} (\xi)).$
It is easy to check that $S_*$
satisfies (\ref{P:SStar-1}) and~(\ref{P:SStar-2}).
For~(\ref{P:SStar-2x}),
first let $f \colon \C^n \to \R$ be \ct.
Define
$g \colon \R^{2 n} \to \R$ by
\[
g (s_1, t_1, s_2, t_2, \ldots, s_n, t_n)
  = f (s_1 + i t_1, \, s_2 + i t_2, \, \ldots, \, s_n + i t_n)
\]
for $s_1, t_1, s_2, t_2, \ldots, s_n, t_n \in \R.$
Set
$\et (x) = f \big( \xi_1 (x), \, \xi_2 (x), \, \ldots, \xi_n (x) \big)$
for $x \in X.$
Then one checks that
\begin{align*}
S_* (\et) (y)
& = g \big( S_* ({\mathrm{Re}} (\xi_1)) (y),
     \, S_* ({\mathrm{Im}} (\xi_1)) (y),
     \, \ldots,
     \, S_* ({\mathrm{Re}} (\xi_n)) (y),
     \, S_* ({\mathrm{Im}} (\xi_n)) (y) \big)
         \\
& = f \big( S_* (\xi_1) (y), \, S_* (\xi_2) (y), \, \ldots,
      \, S_* (\xi_n) (y) \big)
\end{align*}
for $y \in Y,$
as desired.
The extension to \cfn s $f \colon \C^n \to \C$
is now easy.

We now prove~(\ref{P:SStar-3}).
It suffices to show that a real valued measurable function
$\ld \in L^0 \big( Y, \, \nu |_{{\mathrm{ran}} (S)} \big)$
is in the range of~$S_*.$
For $\af \in \Q,$
define
\[
G_{\af} = \{ y \in Y \colon \ld (y) < \af \}
        \in {\mathrm{ran}} (S).
\]
Choose $E_{\af}^{(0)} \in \cB$ such that
$S \big( \big[ E_{\af}^{(0)} \big] \big) = [G_{\af}].$
Define
\[
D = \left( X \SM \bigcup_{\af \in \Q} E_{\af}^{(0)} \right)
     \cup \left( \bigcap_{\af \in \Q} E_{\af}^{(0)} \right),
\]
and set
\[
E_{\af} = \begin{cases}
   D \cup \bigcup_{\bt < \af} E_{\bt}^{(0)} & \af > 0
       \\
   (X \SM D) \cap \bigcup_{\bt < \af} E_{\bt}^{(0)}
                                            & \af \leq 0.
\end{cases}
\]
These sets satisfy the hypotheses of
Lemma~\ref{L-OrdSets}.
Let $\xi$ be the function obtained from Lemma~\ref{L-OrdSets}.
One easily checks that
$E_{\af} = \bigcup_{\bt < \af} E_{\bt}$
for all $\af \in \Q,$
so the second part of Lemma~\ref{L-OrdSets}
implies that
$E_{\af} = \{ x \in X \colon \xi (x) < \af \}$ for all $\af \in \Q.$

Since $S$ is a $\sm$-\hm\  %
and $G_{\af} = \bigcup_{\bt < \af} G_{\bt}$ for all $\af \leq \bt,$
one easily checks that $S ([E_{\af}]) = [G_{\af}]$
for all $\af \in \Q.$
In the construction of $S (\xi)$ at the beginning of the proof,
we may therefore take $F_{\af}^{(0)} = G_{\af}$ for all $\af \in \Q,$
giving $D_0 = \E,$
and then we may take $D = \E.$
Thus $F_{\af} = G_{\af}$ for all $\af \in \Q.$
So $\et = S_* (\xi)$ satisfies~(\ref{Eq:223St})
for all $\af \in \Q.$
Such a function is unique by Lemma~\ref{L-OrdSets},
and $\ld$ is such a function, so $S_* (\xi) = \ld.$
This completes the proof of~(\ref{P:SStar-3}).

We next prove~(\ref{P:SStar-5}).
Suppose first that $\xi$ is real valued.
For $\af \in \Q,$
let
\[
E_{\af} = \{ x \in X \colon \xi (x) < \af \}
\]
(as in~(\ref{Eq:EA})).
Choose $G_{\af} \S Y$ such that $[G_{\af}] = S ([E_{\af}])$
(as in~(\ref{Eq:FAZ}),
where the set is called~$F_{\af}^{(0)}$).
Set
\[
H = \left( Y \SM \bigcup_{\af \in \Q} G_{\af} \right)
      \cup \left( \bigcap_{\af \in \Q} G_{\af} \right),
\]
and set
\[
F_{\af}^{(0)}
 = \begin{cases}
   H \cup \bigcup_{\bt < \af} G_{\bt} & \af > 0
       \\
   (Y \SM H) \cap \bigcup_{\bt < \af} G_{\bt}
                                            & \af \leq 0.
\end{cases}
\]
Since $E_{\af} = \bigcup_{\bt < \af} E_{\bt}$ for all $\af \in \Q$
and $S$ is a $\sm$-\hm,
we have $\big[ F_{\af}^{(0)} \big] = S ([E_{\af}])$
for all $\af \in \Q.$
Therefore we may use the sets $F_{\af}^{(0)}$
in the
construction of $S_* (\xi).$
This gives $D_0 = \E.$
Take $D = \E.$
So $F_{\af} = F_{\af}^{(0)}.$
One easily checks that
$F_{\af}^{(0)} = \bigcup_{\bt < \af} F_{\bt}^{(0)}$
for all $\af \in \Q.$
The second part of Lemma~\ref{L-OrdSets}
therefore implies that
$F_{\af} = \{ y \in Y \colon S_* (\xi) (y) < \af \}$
for all $\af \in \Q.$
We have verified that
\[
S \big( \big[ \{ x \in X \colon \xi (x) < \af \} \big] \big)
  = \big[ \{ y \in Y \colon S_* (\xi) (y) < \af \} \big]
\]
for all $\af \in \Q.$
Since the collection $\big\{ (- \I, \af) \colon \af \in \Q \big\}$
generates the Borel $\sm$-algebra
and $S$ is a $\sm$-\hm,
we have~(\ref{P:SStar-5}) for real valued $\xi$ and all
Borel subsets of~$\R.$

By considering real and imaginary parts separately,
one sees that~(\ref{P:SStar-5}) holds for complex $\xi$
whenever $B$ is the product of two Borel subsets of~$\R.$
Such products generate the Borel subsets of~$\C,$
so~(\ref{P:SStar-5}) holds for arbitrary~$B.$

For~(\ref{P:SStar-4}),
first suppose that $S$ is not injective.
Then there is $E \in \cB$ such that $\mu (E) \neq 0$
but $S ([E]) = [ \E ].$
So $\ch_E \neq 0$ but $S_* (\ch_E) = 0.$

On the other hand,
suppose $S_*$ is not injective.
Then there is a nonzero $\xi \in L^0 (X, \mu)$
such that $S_* (\xi) = 0.$
By considering the positive or negative part of the real or imaginary
part of~$\xi,$
we may assume that $\xi$ is nonnegative.
Since $\xi \neq 0,$
there is $\ep > 0$ such that the set
$E = \{ x \in X \colon \xi (x) > \ep \}$
satisfies $\mu (E) \neq 0.$
Using part~(\ref{P:SStar-5}) at the first step, we get
\[
[S (E)] = \big[ \{ y \in Y \colon S_* (\xi) > \ep \} \big]
        = [ \E ].
\]
Therefore $S$ is not injective.

Part~(\ref{P:SStar-6})
follows from uniqueness,
since $T_* \circ S_*$ and $(T \circ S)_*$ are both linear
and satisfy (\ref{P:SStar-1}) and~(\ref{P:SStar-2}).
\end{proof}

\begin{cor}\label{C-SurjS}
Let the notation be as in Proposition~\ref{P:SStar},
and assume in addition that $\mu$ and $\nu$ are \sft.
Then \tfae:
\begin{enumerate}
\item\label{C-SurjS-1}
$S$ is surjective.
\item\label{C-SurjS-2x}
$S_* \colon L^0 (X, \mu) \to L^0 (Y, \nu)$ is surjective.
\item\label{C-SurjS-2}
The range $S_* \big( L^0 (X, \mu) \big)$
contains $\ch_F$ for every $F \in \cC$
such that $\nu (F) < \I.$
\item\label{C-SurjS-3}
For every $F \in \cC$ there are disjoint sets
$E_1, E_2, \ldots \in \cB$ with finite measure such that,
with $E = \bigcup_{n = 1}^{\I} E_n,$
we have $S_* (\ch_E) = \ch_F.$
\end{enumerate}
\end{cor}

\begin{proof}
That (\ref{C-SurjS-1}) implies~(\ref{C-SurjS-2x})
is Proposition~\ref{P:SStar}(\ref{P:SStar-3}).
That (\ref{C-SurjS-2x}) implies~(\ref{C-SurjS-2})
is trivial.

We prove that (\ref{C-SurjS-2}) implies~(\ref{C-SurjS-3}).
Since $\nu$ is \sft\  and $S_*$ preserves
pointwise almost everywhere limits
(Proposition~\ref{P:SStar}(\ref{P:SStar-2})),
it suffices to consider sets $F$ with $\nu (F) < \I.$
Choose $\xi \in L^0 (X, \mu)$ such that $S_* (\xi) = \ch_F.$
Set $E = \{ x \in X \colon \xi (x) = 1 \}.$
Then $S ([E]) = [F]$ by Proposition~\ref{P:SStar}(\ref{P:SStar-5}).
Now use \sft ness of $\mu$ to write $E = \coprod_{n = 1}^{\I}$
with $\mu (E_n) < \I$ for all~$n.$

Finally, we prove that (\ref{C-SurjS-3}) implies~(\ref{C-SurjS-1}).
Let $F \in \cC.$
Then~(\ref{C-SurjS-3}) implies that there is $\xi \in L^0 (X, \mu)$
such that $S_* (\xi) = \ch_F.$
Set $E = \{ x \in X \colon \xi (x) = 1 \}.$
Then $E \in \cB$ and
$S ([E]) = [F]$ by Proposition~\ref{P:SStar}(\ref{P:SStar-5}).
\end{proof}

\begin{lem}\label{L:SUpStarMu}
Let the notation be as in Definition~\ref{D:SetTrans}.
Let $\ld$ be a measure on ${\mathcal{C}}$ such that
$\ld << \nu.$
Then there exists a unique measure $S^* (\ld)$
on $\cB$
such that whenever $E \in {\mathcal{B}}$ and
$F \in {\mathcal{C}}$ satisfy $[F] = S ([E]),$
then $S^* (\ld) (E) = \ld (F).$
Moreover:
\begin{enumerate}
\item\label{L:SUpStarMu-1}
$S^* (\ld) << \mu.$
\item\label{L:SUpStarMu-1a}
If $\sm$ is another measure on ${\mathcal{C}}$ and $\sm << \ld,$
then $S^* (\sm) << S^* (\ld).$
\item\label{L:SUpStarMu-2}
If $S^* (\ld)$ is $\sm$-finite, then
$\ld$ is $\sigma$-finite.
\item\label{L:SUpStarMu-3}
For every nonnegative function $\xi \in L^0 (X, \mu),$
and for every function $\xi \in L^1 (X, \, S^* (\ld) ),$
we have
\[
\int_X \xi \, d (S^* (\ld)) = \int_Y S_* (\xi) \, d \ld.
\]
\item\label{L:SUpStarMu-4x}
If $\XBM = \YCN$ and $S$ is the identity map,
then $S^* (\ld) = \ld.$
\item\label{L:SUpStarMu-4}
Let $(Z, {\mathcal{D}}, \rh)$ be another measure space,
and let $T$ be a \mst\  from
$(Y, {\mathcal{C}}, \nu)$
to $(Z, {\mathcal{D}}, \rh).$
Suppose $\sm$ is a measure on ${\mathcal{D}}$ such that
$\sm << \rh.$
Then $(T \circ S)^* (\sm) = S^* ( T^* (\sm)).$
\end{enumerate}
\end{lem}

\begin{proof}
Uniqueness of $S^* (\ld)$ is obvious.
For existence,
we have to prove that $S^* (\ld)$ is well defined,
satisfies $S^* (\E) = 0,$
and is countably additive.
The first follows from $\ld << \nu,$
the second is immediate, and the third follows from the first
together with Lemma~\ref{L:InvImDisj}
and Lemma~\ref{L:ShmProp}(\ref{L:ShmProp-2}).

Parts (\ref{L:SUpStarMu-1}), (\ref{L:SUpStarMu-1a}),
(\ref{L:SUpStarMu-4x}),
and~(\ref{L:SUpStarMu-4}) are clear.

To prove~(\ref{L:SUpStarMu-2}), write
$X = \bigcup_{n = 1}^{\I} E_n$
with $S^* (\ld) (E_n) < \I$ for all $n \in \N.$
Choose $F_n \in \cC$ such that $S ([E_n]) = [F_n].$
Then $\ld (F_n) < \I.$
Therefore $F = \bigcup_{n = 1}^{\I} F_n$ is a subset
of~$Y$ on which $\ld$ is \sft.
Also
$[Y \SM F] = [ \E ].$
Since $\ld << \nu,$
this implies that $\ld (Y \SM F) = 0.$
Therefore $\ld$ is \sft.

It remains to verify~(\ref{L:SUpStarMu-3}).
It suffices to do so when $\xi$ is nonnegative.
The result is immediate for simple functions.
Choose a sequence $(\xi_n)_{n \in \N}$
of nonnegative simple functions
such that $\xi_n \to \xi$ pointwise
and $\xi_1 \leq \xi_2 \leq \cdots$ pointwise.
Proposition~\ref{P:SStar}(\ref{P:SStar-2})
implies that
$S_* (\xi_n) \to S_* (\xi)$ pointwise almost everywhere~$[\nu],$
and Proposition~\ref{P:SStar}(\ref{P:SStar-5})
implies that
$S_* (\xi_1) \leq S_* (\xi_2) \leq \cdots$
pointwise almost everywhere~$[\nu].$
Using $\ld << \nu,$ $S^* (\ld) << \mu,$
and the Monotone Convergence Theorem (twice),
we get
\[
\int_X \xi \, d S^* (\ld)
  = \limi{n} \int_X \xi_n \, d S^* (\ld)
  = \limi{n} \int_Y S_* (\xi_n) \, d \ld
  = \int_Y S_* (\xi) \, d \ld.
\]
This completes the proof.
\end{proof}

The converse
to Lemma~\ref{L:SUpStarMu}(\ref{L:SUpStarMu-2}) is false.
Take $X = \Z,$ take $\mu$ to be counting measure,
take $Y = \Z \times \Z,$
take $\nu$ to be counting measure, take $\ld = \nu,$
and take $S (E) = E \times \Z$ for $E \subset X.$

We really need to push measures forwards rather than pull them back.

\begin{dfn}\label{D:SDnStarMu}
Let the notation be as in Definition~\ref{D:SetTrans},
and assume that $S$ is injective.
Let $\ld$ be a measure defined on ${\mathcal{B}}$ which is
absolutely continuous with respect to~$\mu.$
Then we define $S_* (\ld)$ to be the measure
$(S^{-1})^* (\ld)$ on the \sga\  ${\mathrm{ran}} (S).$
\end{dfn}

The measure $S_* (\mu)$ is called $\mu^*$
in the statement of Theorem~3.1 of~\cite{Lp}.

We need some notation.

\begin{ntn}\label{N:RestrSgAlg}
Let $\XBM$ be a \msp,
and let $E \in \cB.$
We denote by $\cB |_E$ the \sga\  on~$E$
consisting of all $F \in \cB$ such that $F \subset E.$
We call it the {\emph{restriction}} of~$\cB.$

If $\ld$ is a measure defined on a \sga~$\cB,$
and $\cC \subset \cB$ is a \sga\  on a set $E \in \cB,$
we write $\ld |_{\cC}$ for the restriction of~$\ld.$
If $\cC = \cB |_E,$ we just write $\ld |_E.$
When no confusion can arise, we often just write~$\ld.$
For example, for a \msp\  $\XBM,$
we often write $L^p (E, \mu)$ rather than $L^p (E, \mu |_E).$
Also, we identify without comment $L^p (E, \mu)$ as the subspace
of $L^p (X, \mu)$ consisting of those functions which
vanish off~$E.$
\end{ntn}

\begin{lem}\label{L:SDnStarMu}
In the situation of Definition~\ref{D:SDnStarMu},
the measures $S_* (\mu)$ and $\nu |_{{\mathrm{ran}} (S)}$
are mutually absolutely \ct.
\end{lem}

\begin{proof}
We have $S_* (\mu) << \nu |_{{\mathrm{ran}} (S)}$
by Lemma~\ref{L:SUpStarMu}(\ref{L:SUpStarMu-1}).
Also,
$S_* ( S^* ( \nu |_{{\mathrm{ran}} (S)} ) )
    = \nu |_{{\mathrm{ran}} (S)}$
by Parts (\ref{L:SUpStarMu-4x}) and~(\ref{L:SUpStarMu-4})
of Lemma~\ref{L:SUpStarMu},
and $S^* ( \nu |_{{\mathrm{ran}} (S)} ) << \mu$
by Part~(\ref{L:SUpStarMu-1}) of Lemma~\ref{L:SUpStarMu},
so $\nu |_{{\mathrm{ran}} (S)} << S_* (\mu)$
by Part~(\ref{L:SUpStarMu-1a}) of Lemma~\ref{L:SUpStarMu}.
\end{proof}

We summarize some standard computations.

\begin{cor}\label{C-Computation}
Let $\XBM$ and $\YCN$ be \sft\  measure spaces,
and let $S$ be a bijective \mst\  from $\XBM$ to $\YCN.$
Set
\[
h = \left[ \frac{d (S_* (\mu))}{d \nu} \right]
  \in L^0 (Y, \nu).
\]
Then $h$ has values in $(0, \I)$ almost everywhere~$[\nu],$
and
\[
\left[ \frac{d ( (S^{-1})_* (\nu))}{d \mu} \right]
  = \frac{1}{(S^{-1})_* (h)}
  = (S^{-1})_* \left( \frac{1}{h} \right).
\]

Moreover, whenever $\xi \in L^0 (X, \mu)$ is nonnegative,
or one of the integrals in the following exists
(in which case they all do),
we have
\[
\int_X \xi \, d \mu
 = \int_Y S_* (\xi) \, d (S_* (\mu))
 = \int_Y S_* (\xi) h \, d \nu.
\]
\end{cor}

\begin{proof}
This follows from Lemma~\ref{L:SUpStarMu}(\ref{L:SUpStarMu-3})
for $S$ and $S^{-1},$
combined with standard properties of Radon-Nikodym derivatives
and Proposition~\ref{P:SStar}(\ref{P:SStar-2x}).
\end{proof}

\begin{cor}\label{C-429RN}
Let $\XBM$ and $\YCN$ be \sft\  measure spaces,
and let $S$ be a bijective \mst\  from $\XBM$ to $\YCN.$
\begin{enumerate}
\item\label{C-429RN-1}
Let $\ld$ be a $\sm$-finite measure on $(X, \cB)$
such that $\ld << \mu.$
Then $S_* (\ld)$ is $\sm$-finite and $S_* (\ld) << \nu.$
\item\label{C-429RN-2}
Let $\ld$ and $\sm$ be $\sm$-finite measures on $(X, \cB)$
such that $\sm,$ $\ld,$ and~$\mu$ are all mutually absolutely
continuous.
Then $S_* (\sm),$ $S_* (\ld),$ and $S_* (\mu)$
are all mutually absolutely
continuous,
and
\[
\left[ \frac{d S_* (\sm)}{d S_* (\ld)} \right]
 = S_* \left( \left[ \frac{d \sm}{d \ld} \right] \right)
\]
almost everywhere $[S_* (\ld)].$
\end{enumerate}
\end{cor}

\begin{proof}
For~(\ref{C-429RN-1}),
set $\ta = S_* (\ld).$
Then $\ta = (S^{-1})^* (\ld),$
so $S^* (\ta) = \ld$
by Lemma~\ref{L:SUpStarMu}(\ref{L:SUpStarMu-4}).
Therefore Lemma~\ref{L:SUpStarMu}(\ref{L:SUpStarMu-2})
implies that $\ta$ is \sft,
and it follows from Lemma~\ref{L:SUpStarMu}(\ref{L:SUpStarMu-1}),
applied to $S^{-1},$
that $S_* (\ld) << \mu.$

We prove~(\ref{C-429RN-2}).
Mutual absolute continuity
follows from Lemma~\ref{L:SUpStarMu}(\ref{L:SUpStarMu-1a}),
applied to $S^{-1}.$
In particular,
we can use Corollary~\ref{C-Computation}
with $\sm$ and~$\ld$ in place of~$\mu.$
Next, since the measures are \sft\  by~(\ref{C-429RN-1}),
the required Radon-Nikodym derivatives exist.
We then prove the result by showing that
\[
\int_Y \et \cdot S_* \left( \left[ \frac{d\sm}{d \ld} \right] \right)
    \, d S_* (\ld)
 = \int_Y \et \, d S_* (\sm)
\]
for all nonnegative $\et \in L^0 (Y, \nu).$
By Corollary~\ref{C-SurjS}(\ref{C-SurjS-2x}),
we may assume that $\et = S_* (\xi)$ for some $\xi \in L^0 (X, \mu).$
We may take $\xi$ to be nonnegative
by Proposition~\ref{P:SStar}(\ref{P:SStar-5}).
Using Corollary~\ref{C-Computation}
for $(X, {\mathcal{B}}, \sm)$
and Proposition~\ref{P:SStar}(\ref{P:SStar-2x})
at the first step,
and Corollary~\ref{C-Computation}
for $(X, {\mathcal{B}}, \ld)$
at the third step,
we have
\[
\int_Y S_* (\xi) S_* \left( \left[ \frac{d\sm}{d \ld} \right] \right)
    \, d S_* (\ld)
  = \int_X \xi \left[ \frac{d\sm}{d \ld} \right] \, d \ld
  = \int_X \xi \, d \sm
  = \int_Y S_* (\xi) \, d S_* (\sm).
\]
This completes the proof.
\end{proof}

There is a more general version of
Corollary~\ref{C-429RN}(\ref{C-429RN-2}),
in which we only assume $\sm << \ld << \mu.$
Since it has a more complicated statement
and we don't need it, we omit it.

\section{Spatial partial isometries and Lamperti's
  Theorem}\label{Sec:SPI}

\indent
We will need a systematic theory of isometries and
partial isometries on $L^p$~spaces.
The main result is Lamperti's Theorem~\cite{Lp},
according to which, for $p \in (0, \I) \SM \{ 2 \},$
every isometry between $L^p$~spaces is ``semispatial''
in a sense which we describe below.
The material we need
in order to to make effective use of Lamperti's Theorem
seems not to be in the literature,
so we describe it here.

There are two choices of how to formulate the definitions,
giving different results on $L^2 (X, \mu).$
We adopt the stricter choice.

We will need the following three computations.

\begin{lem}\label{L-PrelimComp}
Let $\XBM$ and $\YCN$ be \msp s,
let $p \in [1, \I],$
let $S$ be an injective measurable set transformation
from $\XBM$ to $\YCN$
such that $\nu |_{ {\mathrm{ran}} (S) }$ is $\sm$-finite,
and let $g$ be a measurable function on~$Y$
such that $| g (y) | = 1$ for almost all $y \in Y.$
\begin{enumerate}
\item\label{L-PrelimComp-0}
The Radon-Nikodym derivative
\[
h =
 \left[ \frac{d S_* (\mu)}{d (\nu |_{ {\mathrm{ran}} (S) }) } \right]
   \in L^0 (Y, \, \nu |_{ {\mathrm{ran}} (S) } )
\]
exists
and satisfies $h (y) \in (0, \I)$ for almost all~$y \in Y$
with respect to $\nu |_{ {\mathrm{ran}} (S) }.$
\item\label{L-PrelimComp-1}
Let $h$ be as in~(\ref{L-PrelimComp-0}).
Let $\xi$ be a measurable function on~$X.$
Define a measurable function $\et$ on~$Y$ by
\[
\et = g h^{1/p} S_* (\xi).
\]
Then $\| \et \|_p = \| \xi \|_p.$
\item\label{L-PrelimComp-2}
Let $\xi$ be a measurable function on~$X,$
let $\et$ be as in~(\ref{L-PrelimComp-1}),
and suppose that $S$ is bijective.
Then for almost all $x \in X,$ we have
\[
\xi = (S^{-1})_* \left( \frac{1}{g} \right)
    \left[ \frac{d (S^{-1})_* (\nu)}{d \mu} \right]^{1/p}
                               (S^{-1})_* (\et).
\]
\end{enumerate}
\end{lem}

\begin{rmk}\label{R-NSmf}
Lemma~\ref{L-PrelimComp}(\ref{L-PrelimComp-1})
says that $\xi \mapsto  g h^{1/p} S_* (\xi)$
defines an isometry in $L \big( L^p (X, \mu), \, L^p (Y, nu) \big).$
(This will be made explicit in Lemma~\ref{L-BasicWSPI} below.)
The measure
$\nu |_{ {\mathrm{ran}} (S) }$ need not be $\sm$-finite,
and in this case we do not get
an element of $L \big( L^p (X, \mu), \, L^p (Y, \nu) \big).$
Example: take $X$ to consist of one point~$x,$
take $Y = \N,$
take $\mu$ and $\nu$ to be counting measure,
take $S ( \E) = \E$ and $S (X) = Y,$
and take $g = 1.$
\end{rmk}

\begin{proof}[Proof of Lemma~\ref{L-PrelimComp}]
We prove~(\ref{L-PrelimComp-0}).
The measures $S_* (\mu)$ and $\nu |_{ {\mathrm{ran}} (S) }$
are mutually absolutely \ct\  %
by Lemma~\ref{L:SDnStarMu}.
The measure $\nu |_{ {\mathrm{ran}} (S) }$
is $\sm$-finite by hypothesis.

We claim that $S_* (\mu)$ is $\sm$-finite.
Write $X = \bigcup_{n = 1}^{\I} X_n$ with $\mu (X_n) < \I.$
Choose $Y_n \in \cC$ such that $S ([X_n]) = [Y_n].$
Then $S_* (\mu) (Y_n) = \mu (X_n) < \I.$
Since $\left[ Y \SM \bigcup_{n = 1}^{\I} Y_n \right] = S ([\E]),$
we have $S_* (\mu) \left( Y \SM \bigcup_{n = 1}^{\I} Y_n \right) = 0.$
The claim follows,
as does~(\ref{L-PrelimComp-0}).

For the remaining two parts,
we present the proof for $p \neq \I.$
(The case $p = \I$ is simpler.)
For~(\ref{L-PrelimComp-1}),
we have, using $| g | = 1$ almost everywhere at the first step
and Corollary~\ref{C-Computation}
at the third step,
\[
\| \et \|_p^p
  = \int_Y h | S_* (\xi) |^p \, d \nu
  = \int_Y | S_* (\xi) |^p \, d S_* (\mu)
  = \int_X | \xi |^p \, d \mu
  = \| \xi \|_p^p.
\]
For~(\ref{L-PrelimComp-2}),
use Corollary~\ref{C-Computation} at the first step,
Proposition~\ref{P:SStar}(\ref{P:SStar-2x}) at the second step,
and Proposition~\ref{P:SStar}(\ref{P:SStar-6}) at the third step,
to get, with equalities almost everywhere~$[\mu],$
\begin{align*}
\lefteqn{(S^{-1})_* \left( \frac{1}{g} \right)
    \left[ \frac{d (S^{-1})_* (\nu)}{d \mu} \right]^{1/p}
                               (S^{-1})_* (\et)}
                       \\
& \hspace{2em}
   = (S^{-1})_* \left( \frac{1}{g} \right)
    (S^{-1})_* \left( \frac{1}{h} \right)^{1/p} (S^{-1})_* (\et)
   = (S^{-1})_* (S_* (\xi))
   = \xi.
\end{align*}
This completes the proof.
\end{proof}

\begin{dfn}\label{D-WSpSys}
Let $\XBM$ and $\YCN$ be \sfm s.
A {\emph{semispatial system}} for $\XBM$ and $\YCN$
is a quadruple
$(E, F, S, g)$
in which $E \in \cB,$ in which $F \in \cC,$
in which $S$ is an injective measurable set transformation
from $(E, \cB |_E, \mu |_E)$
to $\big( F, \cC_0 |_F, \nu_F \big)$
such that $\nu |_{ {\mathrm{ran}} (S) }$ is $\sm$-finite,
and in which $g$ is a $\cC$-measurable function on~$F$
such that $| g (y) | = 1$ for almost all $y \in F.$

We say that $(E, F, S, g)$ is a {\emph{spatial system}}
if, in addition, $S$ is bijective.
\end{dfn}

\begin{dfn}\label{D-WSPI}
Let $\XBM$ and $\YCN$ be \sfm s,
and let $p \in [1, \I].$
A linear map $s \in L \big( L^p (X, \mu), \, L^p (Y, \nu) \big)$
is called a {\emph{semispatial partial isometry}} if there
exists a semispatial system $(E, F, S, g)$
such that, for every $\xi \in L^p (X, \mu),$
we have
\[
(s \xi) (y)
 = \begin{cases}
     g (y) \left( \left[ \frac{d S_* (\mu |_E)}{
              d ( \nu |_{{\mathrm{ran}} (S)} )}
                                   \right] (y) \right)^{1/p}
                               S_* (\xi |_E) (y) &   y \in F  \\
     0                                           &   y \not\in F.
    \end{cases}
\]
(When $p = \I,$ we take
\[
\left[
 \frac{d S_* (\mu |_E)}{d (\nu |_{{\mathrm{ran}} (S)})} \right]^{1/p}
\]
to be the constant function~$1.$)

We call $s$ a {\emph{spatial partial isometry}}
if $(E, F, S, g)$ is in fact a spatial system.

We call $(E, F, S, g)$ the {\emph{(semi)spatial system}} of~$s.$
(We will see in Lemma~\ref{L-WPiecesUniq} below
that it is essentially unique.)
We call $E$ and $F$
the {\emph{domain support}} and the {\emph{range support}} of~$s.$
The sub-$\sm$-algebra ${\mathrm{ran}} (S) \subset \cC |_F$
is called the
{\emph{range $\sm$-algebra}},
the measurable set transformation $S$ is called
the {\emph{(semi)spatial realization}} of~$s,$
and $g$ is called the {\emph{phase factor}}.
If $\mu (X \SM E) = 0,$
we call $s$ a {\emph{(semi)spatial isometry}}.
\end{dfn}

\begin{lem}\label{L-BasicWSPI}
Let $\XBM$ and $\YCN$ be \msp s,
let $p \in [1, \I],$
and let $(E, F, S, g)$ be a semispatial system for $\XBM$ and $\YCN.$
\begin{enumerate}
\item\label{L-BasicWSPI-1}
There exists a unique semispatial partial isometry
$s \in L \big( L^p (X, \mu), \, L^p (Y, \nu) \big)$
whose semispatial system is $(E, F, S, g).$
\item\label{L-BasicWSPI-2}
Let $s$ be as in~(\ref{L-BasicWSPI-1}).
Then $\| s \xi \|_p = \| \xi |_E \|_p$ for every $\xi \in L^p (X, \mu),$
and $\| s \| \leq 1.$
\item\label{L-BasicWSPI-3}
Let $s$ be as in~(\ref{L-BasicWSPI-1}).
Then for any set $B \in \cC,$
the range of $s$ is contained in $L^p (Y, \nu |_B)$
\ifo\  $B$ contains $F$ up to a set of measure zero.
\item\label{L-BasicWSPI-5}
Suppose $s$ as in~(\ref{L-BasicWSPI-1}) is a semispatial isometry.
Then $s$ is isometric as a linear map,
that is,
$\| s \xi \|_p = \| \xi \|_p$ for every $\xi \in L^p (X, \mu).$
\end{enumerate}
\end{lem}

\begin{proof}
Lemma~\ref{L-PrelimComp}(\ref{L-PrelimComp-1}),
applied with $E$ in place of $X$ and $F$ in place of~$Y,$
implies existence of~$s$ in~(\ref{L-BasicWSPI-1}).
Uniqueness is obvious.
Part~(\ref{L-BasicWSPI-2}) follows
from Lemma~\ref{L-PrelimComp}(\ref{L-PrelimComp-1}).
Part~(\ref{L-BasicWSPI-5}) is a special case of~(\ref{L-BasicWSPI-2}).

For~(\ref{L-BasicWSPI-3}),
it is obvious that if $B$ contains $F$ up to a set of measure zero,
then ${\mathrm{ran}} (s) \subset L^p (Y, \nu |_B).$
Now suppose ${\mathrm{ran}} (s) \subset L^p (Y, \nu |_B).$
Choose $\xi \in L^p (E, \mu)$
such that $\xi (x) \neq 0$ for all $x \in X.$
It follows from Proposition~\ref{P:SStar}(\ref{P:SStar-5})
that $S_* (\xi |_E)$ is nonzero almost everywhere on~$F,$
from the hypotheses that $g$ is nonzero almost everywhere on~$F,$
and from Lemma~\ref{L-PrelimComp}(\ref{L-PrelimComp-0})
that
\[
\left[ \frac{d S_* (\mu |_E)}{
              d ( \nu |_{{\mathrm{ran}} (S)} )} \right]
\]
is nonzero almost everywhere on~$F.$
Therefore $s \xi$ is nonzero almost everywhere on~$F,$
whence $B$ contains $F$ up to a set of measure zero.
\end{proof}

\begin{lem}\label{L-WPiecesUniq}
Let $\XBM$ and $\YCN$ be \sfm s,
let $p \in [1, \I],$
and let $s \in L \big( L^p (X, \mu), \, L^p (Y, \nu) \big)$
be a semispatial partial isometry.
Then its spatial system is unique up to changes in
the domain and range supports by sets of measure zero
and up to equality almost everywhere~$[\nu]$ for the phase factors.
\end{lem}

\begin{rmk}\label{R-103-56}
In Lemma~\ref{L-WPiecesUniq},
we identify $\cB |_E / {\mathcal{N}} (\mu |_E)$
with a subset of $\cB / {\mathcal{N}} (\mu).$
This subset does not change if $E$ is modified by a set of measure zero.
We treat $\cC |_F / {\mathcal{N}} (\nu |_F)$ similarly.
With these identifications,
it makes sense to say that the measurable set transformation
component is actually unique.
In particular,
he part about the domain and range supports says they are
uniquely determined as elements of
$\cB / {\mathcal{N}} (\mu)$ and $\cC / {\mathcal{N}} (\nu).$
\end{rmk}

\begin{proof}[Proof of Lemma~\ref{L-WPiecesUniq}]
Suppose that for $j = 1, 2,$
the sets $E_j \subset X$ and $F_j \subset Y$
are measurable,
$S_j$ is an injective measurable set transformation
from $(E_j, \cB |_{E_j}, \mu |_{E_j})$
to $(F_j, \cC |_{F_j}, \nu |_{F_j}),$
and $g_j$ is a measurable function on~$F_j$
with $| g_j (y) | = 1$ for almost all $y \in F_j.$
Let $s_1$ and $s_2$ be the semispatial partial isometries
obtained from Definition~\ref{D-WSPI}.
We have to prove that if $s_1 = s_2,$
then
\[
[E_1] = [E_2],
\,\,\,\,\,\,
[F_1] = [F_2],
\,\,\,\,\,\,
S_1 = S_2,
\andeqn
g_1 = g_2 \,\,\,
{\mbox{almost everywhere~$[\nu]$}}.
\]

It follows from Lemma~\ref{L-BasicWSPI}(\ref{L-BasicWSPI-3})
that $[F_1] = [F_2].$
It is clear from
Lemma~\ref{L-BasicWSPI}(\ref{L-BasicWSPI-2})
that we must have $[E_1] = [E_2].$

For $j = 1, 2,$ set
\[
h_j = \left[ \frac{d (S_j)_* (\mu |_E)}{
              d ( \nu |_{{\mathrm{ran}} (S_j)} )} \right].
\]

Let $B \in \cB |_{E_1}$ satisfy $\mu (B) < \I.$
Then $\ch_B \in L^p (X, \mu).$
{}From the formula for $s_1 (\ch_B),$
and since $h_1$ and $g_1$ are nonzero almost everywhere on~$F,$
it follows that $S_1 ( [B] ) = [C]$ \ifo\  $s_1 (\ch_B)$ is
nonzero almost everywhere on~$C$
and zero almost everywhere on $Y \SM C.$
Of course, the same applies to $s_2$ and~$S_2.$
Since $\mu$ is \sft\   and $S_1$ and $S_2$ are
$\sm$-\hm s,
it follows that $S_1 = S_2,$
and also that $h_1 = h_2.$

It remains to prove that $g_1 = g_2$ almost everywhere~$[\nu].$
Choose $\xi \in L^p (E, \mu)$
such that $\xi (x) \neq 0$ for all $x \in X.$
It follows from Proposition~\ref{P:SStar}(\ref{P:SStar-5})
that $(S_1)_* (\xi |_E)$ is nonzero almost everywhere on~$F_1,$
and from Lemma~\ref{L-PrelimComp}(\ref{L-PrelimComp-0})
that $h_1$ is nonzero almost everywhere on~$F_1.$
{}From $s_1 \xi = s_2 \xi$ almost everywhere,
we therefore get $g_1 = g_2$ almost everywhere on~$F_1,$
as desired.
\end{proof}

\begin{rmk}\label{R-SpatialCF}
Let $\XBM$ and $\YCN$ be \msp s,
let $p \in [1, \I],$
and let $s \in L \big( L^p (X, \mu), \, L^p (Y, \nu) \big)$
be a semispatial partial isometry
with domain support $E \subset X,$
range support $F \subset Y,$
semispatial realization~$S,$
and range $\sm$-algebra~$\cC_0.$
Then $s$ is spatial \ifo\  $\cC_0 = \cC |_F.$
If this is the case,
then $s$ is a spatial isometry \ifo\  $\mu (X \SM E) = 0.$
\end{rmk}

Using the terminology we have introduced, we can now state
Lamperti's Theorem as follows.

\begin{thm}[Lamperti]\label{T:Lamperti}
Let $\XBM$ and $\YCN$ be \sfm s,
let $p \in [1, \I) \SM \{ 2 \},$
and let $s \in L \big( L^p (X, \mu), \, L^p (Y, \nu) \big)$
be an isometric (not necessarily surjective) linear map.
Then $s$ is a semispatial isometry
in the sense of Definition~\ref{D-WSPI}.
\end{thm}

\begin{proof}
The proof is the same as that of Theorem~3.1 of~\cite{Lp}.
In~\cite{Lp} it is assumed that
$(X, {\mathcal{B}}, \mu) = (Y, {\mathcal{C}}, \nu),$
but this is never used in the proof.
\end{proof}

\begin{rmk}\label{R-PhaseNotMb}
Let $p \in [1, \I) \SM \{ 2 \},$
and let $s \in L \big( L^p (X, \mu), \, L^p (Y, \nu) \big)$
be an isometry.
Theorem~\ref{T:Lamperti} states that there exists
a semispatial system $(E, F, S, g),$
with $E = X,$ such that the
construction of Definition~\ref{D-WSPI} yields~$s.$
The function~$g$ was only required to be measurable with respect
to $\cC |_F,$
not \wrt\  ${\mathrm{ran}} (S).$
In general, the stronger condition fails.
Take $X$ to consist of one point~$x$
and $Y$ to consist of two points $y_1$ and~$y_2.$
Let $\mu$ and $\nu$ be the counting measures.
Identify $L^p (X, \mu)$ with $\C$ and $L^p (Y, \nu)$ with $\C^2$
in the obvious way.
Define $s$ by $s (\ld) = 2^{- 1 / p} (\ld, \, - \ld).$
The semispatial system must be $(X, Y, S, g),$
with $S (\E) = \E$ and $S (X) = Y,$
and with $g (y_1) = 1$ and $g (y_2) = - 1.$
The function $g$ is then not ${\mathrm{ran}} (S)$-measurable.
\end{rmk}

We are primarily interested in partial isometries which are
spatial rather than merely semispatial.

We will need notation
for multiplication operators on $L^p (X, \mu).$

\begin{ntn}\label{N:MultOps}
Let $\XBM$ be a \msp,
and let $p \in [1, \I].$
We let $m_{X, \mu, p}$
(or, when no confusion should arise, just $m_X$ or $m$)
be the \hm\  $m_{X, \mu, p} \colon L^{\I} (X, \mu) \to \LLp$
defined by
\[
\big( m_{X, \mu, p} (f) \xi \big) (x) = f (x) \xi (x)
\]
for $f \in L^{\infty} (X, \mu),$
$\xi \in L^p (X, \mu)$
and $x \in X.$
\end{ntn}

\begin{lem}\label{L-BasicSPI}
Let $\XBM$ and $\YCN$ be \msp s,
let $p \in [1, \I],$
and let $(E, F, S, g)$ be a spatial system for $\XBM$ and $\YCN$
(in the sense of Definition~\ref{D-WSpSys}).
\begin{enumerate}
\item\label{L-BasicSPI-1}
There exists a unique spatial partial isometry
$s \in L \big( L^p (X, \mu), \, L^p (Y, \nu) \big)$
whose spatial system is $(E, F, S, g).$
\item\label{L-BasicSPI-3}
Let $s$ be as in~(\ref{L-BasicSPI-1}).
Then the range of $s$ is $L^p (F, \nu) \subset L^p (Y, \nu).$
\item\label{L-BasicSPI-4}
Let $s$ be as in~(\ref{L-BasicSPI-1}).
There exists a unique spatial partial isometry
$t \in L \big( L^p (Y, \nu), \, L^p (X, \mu) \big)$
whose semispatial system is
$\big( F, \, E, \, S^{-1}, \, (S^{-1})_* (g)^{-1} \big).$
Moreover, using Notation~\ref{N:MultOps},
we have
$t s = m ( \ch_E)$ and $s t = m (\ch_F).$
\item\label{L-BasicSPI-5}
Let $s$ be as in~(\ref{L-BasicSPI-1})
and let $t$ be as in~(\ref{L-BasicSPI-4}).
Let $u \in L \big( L^p (Y, \nu), \, L^p (X, \mu) \big)$
satisfy $u s = m ( \ch_E )$
and $u |_{L^p (Y \SM F, \, \nu)} = 0.$
Then $u = t.$
\end{enumerate}
\end{lem}

\begin{proof}
Part~(\ref{L-BasicSPI-1}) follows
from Lemma~\ref{L-BasicWSPI}(\ref{L-BasicWSPI-1})
and Definition~\ref{D-WSPI}.
The existence of~$t$ in~(\ref{L-BasicSPI-4})
follows from part~(\ref{L-BasicSPI-1}),
and the formulas for $t s$ and $s t$ in~(\ref{L-BasicSPI-4})
follow from Lemma~\ref{L-PrelimComp}(\ref{L-PrelimComp-2}).
Lemma~\ref{L-BasicWSPI}(\ref{L-BasicWSPI-5})
and the formula for $s t$ imply~(\ref{L-BasicSPI-3}).

For~(\ref{L-BasicSPI-5}),
let $\et \in L^p (Y, \nu).$
Write $\et = \et_1 + \et_2$
with $\et_1 \in L^p (F, \nu)$
and $\et_2 \in L^p (Y \SM F, \, \nu).$
We show that $u \et_1 = t \et_1$ and $u \et_2 = t \et_2.$
Clearly $u \et_2$ and $t \et_2$ are both zero.
Also,
using ${\mathrm{ran}} (t) = L^p (E, \mu)$ at the last step,
we have
\[
u \et_1 = u m (\ch_F) \et_1
        = u s t \et_1
        = m (\ch_E) t \et_1
        = t \et_1.
\]
This completes the proof.
\end{proof}

\begin{dfn}\label{D-Reverse}
Let $\XBM$ and $\YCN$ be \msp s,
let $p \in [1, \I],$
and let $s \in L \big( L^p (X, \mu), \, L^p (Y, \nu) \big)$
be a spatial partial isometry.
The spatial partial isometry $t$ of
Lemma~\ref{L-BasicSPI}(\ref{L-BasicSPI-4})
is called the {\emph{reverse}} of~$s.$
\end{dfn}

\begin{rmk}\label{R-RevAdj}
When $p = 2,$ the reverse of~$s$ is of course~$s^*.$
However, we can't define it this way when $p \neq 2.$
\end{rmk}

\begin{lem}\label{L-CndForSp}
Let $\XBM$ and $\YCN$ be \sfm s,
let $p \in [1, \I],$
and let $s \in L \big( L^p (X, \mu), \, L^p (Y, \nu) \big)$
be a semispatial partial isometry with domain support $E \subset X,$
range support $F \subset Y,$
and range $\sm$-algebra~$\cC_0.$
Then \tfae:
\begin{enumerate}
\item\label{L-CndForSp-1}
$s$ is spatial.
\item\label{L-CndForSp-2}
${\mathrm{ran}} (s) = L^p (F, \mu).$
\item\label{L-CndForSp-3}
$\cC_0 = \cC |_F.$
\end{enumerate}
\end{lem}

\begin{proof}
That (\ref{L-CndForSp-3}) implies~(\ref{L-CndForSp-1})
is clear from the definitions,
and (\ref{L-CndForSp-1}) implies~(\ref{L-CndForSp-2})
by Lemma~\ref{L-BasicSPI}(\ref{L-BasicSPI-3}).
So assume~(\ref{L-CndForSp-2}).
Let $S$ be the semispatial realization of~$s.$
Thus $S$ is an injective measurable set transformation
from $(E, \cB |_{E}, \mu |_{E})$
to $(F, \cC |_{F}, \nu |_{F})$
such that ${\mathrm{ran}} (S_*)$
contains $\ch_B$
for every $B \in \cC |_{F}$ with $\nu (B) < \I.$
It follows from Corollary~\ref{C-SurjS}
that $S$ is surjective, which is~(\ref{L-CndForSp-3}).
\end{proof}

\begin{lem}\label{L-PiIsSp}
Let $p \in [1, \I) \SM \{ 2 \},$
let $\XBM$ and $\YCN$ be \sfm s,
let $E \subset X$ and $F \subset Y$ be measurable subsets,
and let $s \in L \big( L^p (X, \mu), \, L^p (Y, \nu) \big)$
satisfy the following conditions:
\begin{enumerate}
\item\label{L-PiIsSp-1}
The range of $s$ is $L^p (F, \nu) \subset L^p (Y, \nu).$
\item\label{L-PiIsSp-2}
$s |_{L^p (E, \mu)}$ is isometric.
\item\label{L-PiIsSp-3}
$s |_{L^p (X \SM E, \, \mu)} = 0.$
\end{enumerate}
Then $s$ is a spatial partial isometry
in the sense of Definition~\ref{D-WSPI},
and has domain support~$E$ and range support~$F.$
\end{lem}

\begin{proof}
Apply Theorem~\ref{T:Lamperti} with $E$ in place of~$X,$
to conclude that $s$ is a semispatial partial isometry.
Then $s$ is spatial by Lemma~\ref{L-CndForSp}.
\end{proof}

In the rest of this section,
we describe some operations which give new spatial partial isometries
from old ones.
Most of the statements have analogs for semispatial partial isometries,
but we don't need them and don't prove them.

We begin by showing that the product of two spatial partial isometries
is again a spatial partial isometry.
On a Hilbert space,
the product of two partial isometries
is usually not a partial isometry,
unless the range projection of the second commutes with
the domain projection of the first.
With spatial partial isometries,
this commutation relation is automatic.

\begin{lem}\label{L-CompSPI}
Let $(X_1, \cB_1, \mu_1),$ $(X_2, \cB_2, \mu_2),$ and
$(X_3, \cB_3, \mu_3)$ be \sfm s.
Let $p \in [1, \I].$
Let
\[
s \in L \big( L^p (X_1, \mu_1), \, L^p (X_2, \mu_2) \big)
\andeqn
u \in L \big( L^p (X_2, \mu_2), \, L^p (X_3, \mu_3) \big)
\]
be spatial partial isometries,
with reverses $t$ and~$w,$
and with spatial systems $(E_1, E_2, S, g)$ and $(F_2, F_3, V, h).$
Then $v s$ is a spatial partial isometry.
Its domain support is $E = S^{-1} (E_2 \cap F_2),$
its range support is $F = V (E_2 \cap F_2),$
and its reverse is $t w.$
Its spatial realization is the composite $V_0 \circ S_0$
of the restriction $S_0$ of $S$ to $\cB_1 |_{S^{-1} (E_2 \cap F_2)}$
and the restriction $V_0$ of $V$ to $\cB_2 |_{E_2 \cap F_2}.$
Its phase factor is
$k = (V_0)_* (g |_{E_2 \cap F_2} ) ( h |_{V (E_2 \cap F_2)} ).$
\end{lem}

\begin{proof}
It is immediate that $(E, \, F, \ V_0 \circ S_0, \ k)$
is a spatial system for $(X_1, \cB_1, \mu_1)$
and $(X_3, \cB_3, \mu_3).$
Let $y \in L \big( L^p (X_1, \mu_1), \, L^p (X_3, \mu_3) \big)$
be the corresponding spatial partial isometry.
We claim that $y = v s,$
and it is enough to prove that $y \xi = v s \xi$
for $\xi$ in each of the three spaces
$L^p (X_1 \SM E_1, \, \mu_1),$
$L^p (E_1 \SM E, \, \mu_1),$
and $L^p (E, \mu_1).$
In the first case, $y \xi = 0$ and $s \xi = 0.$
In the second case, $y \xi = 0.$
Also $s \xi \in L^p (X_2 \SM F_2, \, \mu_2),$
so $v s \xi = 0.$

So let $\xi \in L^p (E, \mu_1).$
Then
\[
y \xi
 = (V_0)_* ( g |_{E_2 \cap F_2})
   ( h |_{V (E_2 \cap F_2)})
   \left[
    \frac{d (V_0 \circ S_0)_* (\mu_1 |_E) }{d (\mu_3 |_F)} \right]^{1/p}
   (V_0 \circ S_0)_* (\xi)
\]
and
\[
v s \xi
 = (h |_F)
   \left[
    \frac{d (V_0)_* (\mu_2 |_{E_2 \cap F_2}) }{d (\mu_3 |_F)}
         \right]^{1/p}
   (V_0)_* \left( (g |_{E_2 \cap F_2})
         \left[
    \frac{d (S_0)_* (\mu_1 |_E) }{d (\mu_2 |_{E_2 \cap F_2})}
                     \right]^{1/p}
       (S_0)_* (\xi) \right).
\]
One sees that these are equal by combining
Corollary~\ref{C-429RN} with several applications
of Proposition~\ref{P:SStar}(\ref{P:SStar-2x})
and standard properties of Radon-Nikodym derivatives.
This completes the proof that $y = v s,$
and thus the proof that $v s$ is as claimed.

It remains to identify the reverse.
We use Lemma~\ref{L-BasicSPI}(\ref{L-BasicSPI-5}).
First observe that $w \et = 0$
for $\et \in L^p (X_3 \SM F_3, \, \mu_3)$
and $t w \et = 0$
for $\et \in L^p (F_3 \SM F, \, \mu_3).$
Also,
for $\xi \in L^p (X_1, \mu_1),$
the first paragraph of the proof shows that
$v s \xi = v s m (\ch_E) \xi.$
It is easy to check that
$s m (\ch_E) = m ( \ch_{E_2 \cap F_2} ) s.$
Therefore
\[
t w v s
  = t (w v) s m ( \ch_{E} )
  = t m (\ch_{F_2}) m ( \ch_{E_2 \cap F_2} ) s
  = t s m (\ch_E)
  = m (\ch_E).
\]
This completes the proof.
\end{proof}

\begin{lem}\label{L-SPIdem}
Let $\XBM$ be a \sfm,
let $p \in [1, \I],$
and let $e \in \LLp.$
Then \tfae:
\begin{enumerate}
\item\label{L-SPIdem-1}
$e$ is an idempotent spatial partial isometry.
\item\label{L-SPIdem-2}
$e$ is a spatial partial isometry,
and there is $E \in \cB$ such that the spatial system
of $e$ is $(E, E, \id_{\cB |_E}, \ch_E).$
\item\label{L-SPIdem-3}
There is $E \in \cB$ such that $e = m (\ch_E).$
\end{enumerate}
\end{lem}

\begin{proof}
It is clear that (\ref{L-SPIdem-2}) and~(\ref{L-SPIdem-3})
are equivalent,
and that (\ref{L-SPIdem-2}) implies~(\ref{L-SPIdem-1}).
So assume~(\ref{L-SPIdem-1}).
Let the spatial system
of $e$ be $(E, F, S, g),$
and let $f$ be the reverse of~$e$
(Definition~\ref{D-Reverse}).
Lemma~\ref{L-CompSPI}
implies that $f^2 = f.$
Therefore
\begin{equation}\label{Eq:L-SPIdem-E1}
m (\ch_F) m (\ch_E)
  = (e f) (f e)
  = e f e
  = e m (\ch_E)
  = e.
\end{equation}
Multiplying~(\ref{Eq:L-SPIdem-E1}) on the left
by $m (\ch_E)$ gives
\begin{equation}\label{Eq:L-SPIdem-E2}
m (\ch_E) m (\ch_F) m (\ch_E)
  = m (\ch_E) e
  = (f e) e
  = f e
  = m (\ch_E).
\end{equation}
The left hand sides of (\ref{Eq:L-SPIdem-E1}) and~(\ref{Eq:L-SPIdem-E2})
are clearly equal,
so $e = m (\ch_E).$
This is~(\ref{L-SPIdem-3}).
\end{proof}

We now consider the tensor product of two spatial partial isometries.
We need a particular case of the product of two $\sm$-\hm s,
which we can get from Lamperti's Theorem.
Quite possibly something more general is true,
but we don't need it.

\begin{lem}\label{L-ProdSHM}
Let
$(X_1, \cB_1, \mu_1),$
$(X_2, \cB_2, \mu_2),$
$(Y_1, \cC_1, \nu_1),$
and $(Y_2, \cC_2, \nu_2)$
be \sfm s.
Let
\[
S \colon \cB_1 / {\mathcal{N}} (\mu_1) \to \cB_2 / {\mathcal{N}} (\mu_2)
\andeqn
V \colon \cC_1 / {\mathcal{N}} (\nu_1) \to \cC_2 / {\mathcal{N}} (\nu_2)
\]
be bijective \shm s.
Let $\cB_1 \times \cC_1$ and $\cB_2 \times \cC_2$ be the
product $\sm$-algebras on $X_1 \times Y_1$ and $X_2 \times Y_2,$
or their completions.
Then there is a unique \shm\  %
\[
S \times V \colon
(\cB_1 \times \cC_1) / {\mathcal{N}} (\mu_1 \times \nu_1)
  \to (\cB_2 \times \cC_2) / {\mathcal{N}} (\mu_2 \times \nu_2)
\]
such that, whenever $E_1 \in \cB_1,$ $E_2 \in \cB_2,$
$F_1 \in \cC_1,$ and $F_2 \in \cC_2$
satisfy $S ( [E_1] ) = [E_2]$ and $V ([F_1]) = [F_2],$
we have
\begin{equation}\label{Eq:ProdSHM}
(S \times V) ( [ E_1 \times F_1 ] ) = [ E_2 \times F_2 ].
\end{equation}
\end{lem}

\begin{proof}
It does not matter whether we use the product $\sm$-algebras
or their completions,
because the Boolean $\sm$-algebras
\[
(\cB_1 \times \cC_1) / {\mathcal{N}} (\mu_1 \times \nu_1)
\andeqn
(\cB_2 \times \cC_2) / {\mathcal{N}} (\mu_2 \times \nu_2)
\]
are the same with either choice.

Uniqueness of $S \times V$ follows from the fact that the
measurable rectangles generate the product $\sm$-algebra.

We prove existence.
Fix any $p \in [1, \I) \SM \{ 2 \}.$
Let $s$ and $v$ be the spatial isometries
with spatial systems $(X_1, X_2, S, 1)$
and $(Y_1, Y_2, V, 1).$
Then $s$ and $v$ are isometric bijections.
Applying Theorem~\ref{T-LpTP}(\ref{T-LpTP-3b})
to $s$ and~$v,$
and to their inverses,
and applying Theorem~\ref{T-LpTP}(\ref{T-LpTP-4}),
we see that
\[
s \otimes v \colon
 L^p (X_1 \times Y_1, \, \mu_1 \times \nu_1) \to
              L^p (X_2 \times Y_2, \, \mu_2 \times \nu_2)
\]
is an isometric bijection.
Lemma~\ref{L-PiIsSp} implies that it is spatial.
We take $S \times V$ to be its spatial realization.
The relation~(\ref{Eq:ProdSHM}) is easily deduced from
$(s \otimes v) (\xi \otimes \et) = s \xi \otimes v \et,$
Fubini's Theorem,
and \sft ness of all the measures involved.
\end{proof}

In the next lemma,
we exclude $p = \I$ because we use Theorem~\ref{T-LpTP}.

\begin{lem}\label{L-TensorSPI}
Let
$(X_1, \cB_1, \mu_1),$
$(X_2, \cB_2, \mu_2),$
$(Y_1, \cC_1, \nu_1),$
and $(Y_2, \cC_2, \nu_2)$
be \sfm s.
Let $p \in [1, \I).$
Let
\[
s \in L \big( L^p (X_1, \mu_1), \, L^p (X_2, \mu_2) \big)
\andeqn
v \in L \big( L^p (Y_1, \nu_1), \, L^p (Y_2, \nu_2) \big).
\]
be spatial partial isometries,
with reverses $t$ and~$w,$
and with spatial systems $(E_1, E_2, S, g)$ and $(F_2, F_3, V, h).$
Then
\[
s \otimes v
 \in L \big( L^p (X_1 \times Y_1, \, \mu_1 \times \nu_1), \,
              L^p (X_2 \times Y_2, \, \mu_2 \times \nu_2) \big)
\]
(as in Theorem~\ref{T-LpTP})
is a spatial partial isometry.
With $S \times V$ as in Lemma~\ref{L-ProdSHM}
and $g_1 \otimes g_2$ as in Theorem~\ref{T-LpTP},
its spatial system is
\begin{equation}\label{Eq:TensorSPI}
\big( E_1 \times F_1, \, E_2 \times F_2, \, S \times V, \,
    g_1 \otimes g_2 \big),
\end{equation}
and its reverse is $t \otimes w.$
\end{lem}

\begin{proof}
Let $y$ be the spatial partial isometry
with spatial system given by~(\ref{Eq:TensorSPI}),
and let $z$ be its reverse.

It is clear that 
$S_* (\mu_1) \times V_* (\nu_1)$
and $(S \times V)_* (\mu_1 \times \nu_1)$
agree on measurable rectangles.
The measures $S_* (\mu_1)$ and $V_* (\nu_1)$ are $\sm$-finite
by Corollary~\ref{C-429RN}(\ref{C-429RN-1}).
The product of $\sm$-finite measures is uniquely
determined by its values on measurable rectangles.
(In Chapter 12 of~\cite{Ry},
see Theorem~8 and the discussion after Lemma~14.)
Therefore
$S_* (\mu_1) \times V_* (\nu_1) = (S \times V)_* (\mu_1 \times \nu_1).$
Set
\[
k = 
\left[ \frac{d (S \times V)_* (\mu_1 |_{E_1} \times \nu_1 |_{F_1})}{
      d (\mu_2 |_{E_2} \times \nu_2 |_{F_2})} \right]
\andeqn
l = \left[ \frac{d  S_* (\mu_1 |_{E_1})}{d (\mu_2 |_{E_2}) } \right]
   \otimes
    \left[ \frac{d V_* (\nu_1 |_{F_1}) }{d (\nu_2 |_{F_2}) } \right].
\]
Then for every measurable rectangle $R \S E_2 \times F_2,$
we have
\[
\int_R k \, d (\mu_2 |_{E_2} \times \nu_2 |_{F_2})
  = \int_R l \, d (\mu_2 |_{E_2} \times \nu_2 |_{F_2}).
\]
Since
\[
G \mapsto \int_G l \, d (\mu_2 |_{E_2} \times \nu_2 |_{F_2})
\]
is a product of $\sm$-finite measures
on $E_2 \times F_2,$
it follows from uniqueness of product measures
(as above)
that $k = l$
almost everywhere $[\mu_2 |_{E_2} \times \nu_2 |_{F_2}].$
One now checks that $s \otimes v = y$
and $t \otimes w = z$ by showing, in each case,
that they agree on elementary tensors.
\end{proof}

\begin{lem}\label{L-DualSPI}
Let $\XBM$ and $\YCN$ be \sfm s,
let $p \in [1, \I),$
and let $s \in L \big( L^p (X, \mu), \, L^p (Y, \nu) \big)$
be a spatial partial isometry with
spatial system $(E, F, S, g)$ and with reverse~$t.$
Let $q \in (1, \I]$ satisfy $\frac{1}{p} + \frac{1}{q} = 1,$
and for any \sfm\  $(Z, {\mathcal{D}}, \ld),$
identify $L^p (Z, \ld)'$ with $L^q (Z, \ld)$ in the usual way.
Then $s' \in L \big( L^q (Y, \nu), \, L^q (X, \mu) \big)$
is a spatial partial isometry with
spatial system $\big( F, \, E, \, S^{-1}, \, (S^{-1})_* (g) \big)$
and with reverse~$t'.$
\end{lem}

\begin{proof}
To simplify the notation, let
\[
h = \left[ \frac{d S_* (\mu)}{d \nu} \right]
\andeqn
k = \left[ \frac{d (S^{-1})_* (\nu)}{d \mu} \right].
\]
Then $h = S_* (k)^{-1}$ by Corollary~\ref{C-Computation}.
Let $u \in L \big( L^q (Y, \nu), \, L^q (X, \mu) \big)$
be the spatial partial isometry
with the spatial system specified for~$s'.$
Then for $\xi \in L^p (X, \mu)$ and $\et \in L^q (Y, \nu),$
we have,
using Corollary~\ref{C-Computation}
at the second step,
\begin{align*}
\int_X \xi \cdot u \et \, d \mu
& = \int_X \xi (S^{-1})_* (g) k^{1 / q} (S^{-1})_* (\et) \, d \mu
  = \int_Y S_* (\xi) g S_* (k)^{1 / q} \et h \, d \nu
    \\
& = \int_Y S_* (\xi) g h^{- 1 / q} \et h \, d \nu
  = \int_Y S_* (\xi) g h^{1 / p} \et h \, d \nu
  = \int_Y s \xi \cdot \et \, d \nu.
\end{align*}
Thus $s' = u.$

To identify the reverse of~$s',$
we calculate:
\[
t' s' = (s t)'
      = m_{Y, \nu, p} (\ch_F)'
      = m_{Y, \nu, q} (\ch_F)
\]
and
\[
t' m_{X, \mu, q} (\ch_{X \SM E})
 = \big[ m_{X, \mu, p} (\ch_{X \SM E}) t \big]'
 = 0.
\]
Now apply Lemma~\ref{L-BasicSPI}(\ref{L-BasicSPI-5}).
\end{proof}

We finish this section with a lemma on homotopies
that will be needed later.
We do not know whether surjectivity is necessary in
the hypotheses.

\begin{lem}\label{L:LampertiHomotopy}
Let $p \in [1, \I) \SM \{ 2 \}.$
Let $(X, {\mathcal{B}}, \mu)$ and $(Y, {\mathcal{C}}, \nu)$
be $\sm$-finite measure spaces,
and let
$\ld \mapsto s_{\ld} \in L \big( L^p (X, \mu), \, L^p (Y, \nu) \big),$
for $\ld \in [0, 1],$
be a norm \ct\  path of surjective isometries.
Let $S_{\ld}$ be the spatial realization of $s_{\ld}.$
Then $S_{0} = S_{1}.$
\end{lem}

\begin{proof}
We prove
that if $v_0$ and $v_1$ are surjective spatial isometries
whose spatial realizations $V_0$ and $V_1$ are distinct,
then $\| v_0 - v_1 \| \geq 1.$
(It follows that the spatial realization must be constant
along a homotopy.)
The \mst s $V_0$ and $V_1$ are bijective by Lemma~\ref{L-CndForSp}.

Choose a set $E \in \cB$ such that $V_0 ([E]) \neq V_1 ([E]).$
\Wolog\  we may assume that $V_0 ([E])$ and $V_1 ([E])$
have representatives $F_0, F_1 \in \cC$
such that $F_0$ does not contain $F_1$ up to sets of
measure zero.
That is, there is $F \subset Y$ such that
\[
\nu (F) > 0,
\,\,\,\,\,\,
F \subset F_0,
\andeqn
F \cap F_1 = \E.
\]
Since $V_0$ is bijective,
the set $Q = V_0^{-1} (F)$ satisfies $\mu (Q) > 0.$
Since $\mu$ is \sft,
replacing $F$ by a suitable subset
allows us to also assume that $\mu (Q) < \I.$
Correcting by a set of measure zero,
we may further assume that $Q \S E.$

Define
$\xi = \mu (Q)^{-1 / p} \ch_{Q} \in L^p (X, \mu).$
Then $\| \xi \|_p = 1.$
Moreover,
$v_0 \xi$ is supported in $F$ and,
since $Q \subset E,$
the function $v_1 \xi$ is supported in~$F_1.$
Therefore
\[
\| v_0 \xi - v_1 \xi \|_p^p
  = \| v_0 \xi \|_p^p + \| v_1 \xi \|_p^p
  = 2,
\]
so $\| v_0 - v_1 \| \geq 2^{1 / p} \geq 1.$
\end{proof}

\section{Spatial representations}\label{Sec:SpatialReps}

\indent
In this section,
we define and characterize spatial representations on spaces
of the form $L^p (X, \mu),$
first of~$M_d$ and then of~$L_d$ for finite~$d.$
In each case, we give a number of equivalent conditions
for a \rpn\  to be spatial,
some of them quite different from each other.
In particular,
some characterizations are primarily in terms of
how the \rpn\  interacts with~$X,$
while others make sense for a \rpn\  on any Banach space.
We consider $M_d$ first because we use the results about~$M_d$
in the theorem for~$L_d.$

We will see in Theorem~\ref{T:SpatialRepsMd}
that, for fixed~$p,$
any two spatial \rpn s of~$M_d$ determine the same norm on~$M_d,$
and we will see in Theorem~\ref{T:SpatialIsSame}
(in the next section)
that, for for fixed~$p$ and when $d < \I,$
any two spatial \rpn s of~$L_d$ determine the same norm on~$L_d.$

\begin{dfn}\label{D:M_dSpatialRep}
Let $d \in \N,$
and let $(e_{j, k})_{j, k = 1}^d$ be the standard system
of matrix units in~$M_d.$
Let $p \in [1, \I],$
let $\XBM$ be a \sfm,
and let $\rh \colon M_d \to \LLp$ be a \rpn.
(Recall that, by convention, \rpn s are unital.
See Definition~\ref{D:Repn}.)
We say that $\rh$ is {\emph{spatial}} if
$\rh (e_{j, k})$ is a spatial partial isometry,
in the sense of Definition~\ref{D-WSPI},
for $j, k = 1, 2, \ldots, d.$
\end{dfn}

\begin{thm}\label{T:SpatialRepsMd}
Let $d \in \N,$
let $p \in [1, \I) \SM \{ 2 \},$
let $\XBM$ be a \sfm,
and let $\rh \colon M_d \to \LLp$ be a \rpn.
Let $(e_{j, k})_{j, k = 1}^d$ be the standard system
of matrix units in~$M_d.$
Then \tfae:
\begin{enumerate}
\item\label{T:SpatialRepsMd-1}
$\rh$ is spatial.
\item\label{T:SpatialRepsMd-2}
For $j, k = 1, 2, \ldots, d,$
the operator $\rh (e_{j, k})$ is a spatial partial isometry
with reverse $\rh (e_{k, j}).$
\item\label{T:SpatialRepsMd-3}
$\rh$ is isometric as a map from $\MP{d}{p}$
(as in Notation~\ref{N:FDP}) to $\LLp.$
\item\label{T:SpatialRepsMd-4}
$\rh$ is contractive as a map from $\MP{d}{p}$ to $\LLp.$
\item\label{T:SpatialRepsMd-5}
$\| \rh (e_{j, k}) \| \leq 1$ for $j, k = 1, 2, \ldots, d,$
and there exists a measurable partition $X = \coprod_{j = 1}^d X_j$
such that for $j = 1, 2, \ldots, d,$ the matrix unit
$e_{j, j}$ acts (following Notation~\ref{N:MultOps}) as
$\rh (e_{j, j}) = m (\ch_{X_j}).$
\item\label{T:SpatialRepsMd-6}
There exists a measurable partition $X = \coprod_{j = 1}^d X_j$
such that for $j = 1, 2, \ldots, d$ the operator
$\rh (e_{j, 1})$ is a spatial partial isometry
with domain support~$X_1$
and range support~$X_j.$
\item\label{T:SpatialRepsMd-7}
There exists a measurable partition $X = \coprod_{j = 1}^d X_j$
such that for $j, k = 1, 2, \ldots, d$ the operator
$\rh (e_{j, k})$ is zero on $L^p (X \SM X_k, \, \mu)$
and restricts to an isometric isomorphism
from $L^p (X_k, \mu)$ to $L^p (X_j, \mu).$
\item\label{T:SpatialRepsMd-8}
With $\gm$ being counting measure on $N_d = \{ 1, 2, \ldots, d \},$
there exists a \sfm\  $\YCN$ and a bijective isometry
\[
u \colon L^p ( N_d \times Y, \, \gm \times \nu) \to L^p (X, \mu)
\]
such that,
following the notation of Theorem~\ref{T-LpTP}(\ref{T-LpTP-3a}),
for all $a \in M_d$ we have $\rh (a) = u (a \otimes 1) u^{-1}.$
\end{enumerate}
\end{thm}

In the notation of Lemma~\ref{L-TPRep},
Theorem~\ref{T:SpatialRepsMd}(\ref{T:SpatialRepsMd-8})
says that if $\rh_0 \colon M_d \to L \big( l_d^p \big)$
is the standard \rpn,
then $\rh$ is similar,
via an isometry,
to $\rh_0 \otimes_p 1.$

We specifically use complex scalars in the proof
that (\ref{T:SpatialRepsMd-3}) and~(\ref{T:SpatialRepsMd-4})
imply the other conditions.
We don't know whether complex scalars are necessary.

When $p = 2,$
conditions (\ref{T:SpatialRepsMd-1}) and~(\ref{T:SpatialRepsMd-3})
are certainly not equivalent.
We have not investigated what happens when $p = \I,$
but Lamperti's Theorem is not available in this case.

It is essential that the representation be unital.
Many of the conditions of Theorem~\ref{T:SpatialRepsMd}
are never satisfied for nonunital \rpn s,
but (\ref{T:SpatialRepsMd-3}) and~(\ref{T:SpatialRepsMd-4})
do occur.
We show by example that they do not imply that the \rpn\  is spatial.

\begin{exa}\label{E-NonUMd}
We adopt the notation of Example~\ref{E:NormRankOne},
with $p \in [1, \I) \SM \{ 2 \}.$
Set
\[
\xi = 2^{- 1 / p} (1, 1)
\andeqn
\et = 2^{- 1 / q} (1, 1).
\]
(If $p = 1,$ take $\et = (1, 1).$)
Let $e \in \MP{2}{p}$ be the rank one operator called~$a$
in Example~\ref{E:NormRankOne}.
Then
\[
e = \frac{1}{2} \left( \begin{matrix}
  1     &  1        \\
  1     &  1
\end{matrix} \right),\]
which is an idempotent.
Since $\| \xi \|_p = \| \et \|_q = 1,$
we get $\| e \| = 1.$

Now take $(X_0, \cB_0, \mu_0)$ to be any \sfm\  with a spatial
\rpn\  $\rh_0 \colon M_d \to L \big( L^p (X_0, \mu_0) \big).$
Set $X = X_0 \amalg X_0,$
and equip it with the obvious \sga\  $\cB$
and with the measure $\mu$ whose restriction to each copy of~$X_0$
is~$\mu_0.$
Identify $L^p (X, \mu)$ with $l_2^p \otimes_p L^p (X_0, \mu_0)$
as in Theorem~\ref{T-LpTP}.
Define a nonunital \rpn\  $\rh \colon M_d \to \LLp$
by $\rh (a) = e \otimes \rh_0 (a)$
for $a \in M_d.$
Then Theorem~\ref{T-LpTP}(\ref{T-LpTP-3b})
implies that, regarded as a map $\MP{d}{p} \to \LLp,$
the \rpn\  $\rh$ is isometric.
But it is not spatial,
not even in a sense suitable for nonunital \rpn s.
\end{exa}

\begin{proof}[Proof of Theorem~\ref{T:SpatialRepsMd}]
We first prove the equivalence of
(\ref{T:SpatialRepsMd-1}), (\ref{T:SpatialRepsMd-5}),
(\ref{T:SpatialRepsMd-6}), and~(\ref{T:SpatialRepsMd-7}),
beginning with
(\ref{T:SpatialRepsMd-1}) implies~(\ref{T:SpatialRepsMd-5}).

So assume~(\ref{T:SpatialRepsMd-1}).
We have $\| \rh (e_{j, k}) \| \leq 1$ because this is true
for all spatial partial isometries.
For each $j$ and $k$ there is a spatial system
for $\rh (e_{j, k}),$
say $(E_{j, k}, F_{j, k}, S_{j, k}, g_{j, k}).$
Apply Lemma~\ref{L-SPIdem}
to $\rh (e_{j, j}).$
We obtain sets $X_j \subset X$ such that
\[
(E_{j, j}, F_{j, j}, S_{j, j}, g_{j, j})
 = (X_j, X_j, \id_{\cB |_{X_j}}, \ch_{X_j})
\andeqn
\rh (e_{j, j}) = m ( \ch_{X_j} )
\]
for $j = 1, 2, \ldots, d.$
Since $\sum_{j = 1}^d \rh (e_{j, j}) = 1,$
the sets $X_j$ are essentially disjoint
and, up to a set of measure zero, $\bigcup_{j = 1}^d X_j = X.$
Modification by sets of measure zero now gives the rest
of~(\ref{T:SpatialRepsMd-5}).

We next prove
(\ref{T:SpatialRepsMd-5}) implies~(\ref{T:SpatialRepsMd-7}).
We take the partition $X = \coprod_{j = 1}^d X_j$
to be as in~(\ref{T:SpatialRepsMd-5}).
Let $j, k \in \{ 1, 2, \ldots, d \}.$
The equation $\rh (e_{j, k}) \big( 1 - \rh (e_{k, k}) \big) = 0$
translates to $\rh (e_{j, k}) \xi = 0$ for all
$\xi \in L^p (X \SM X_k, \, \mu),$
which is the first part of~(\ref{T:SpatialRepsMd-7}).
For $\xi \in L^p (X_k, \mu),$
use
\[
\| \rh (e_{j, k}) \| \leq 1,
\,\,\,\,\,\,
\| \rh (e_{k, j}) \| \leq 1,
\andeqn
\rh (e_{k, j}) \rh (e_{j, k}) \xi = \rh (e_{k, k}) \xi = \xi
\]
to get $\| \rh (e_{j, k}) \xi \| = \| \xi \|.$
So $\rh (e_{j, k})$ is isometric on $L^p (X_k, \mu).$
Since $\rh (e_{j, k}) \rh (e_{k, j}) \xi = \xi$
for $\xi \in L^p (X_j, \, \mu),$
we have
${\mathrm{ran}} (\rh (e_{j, k})) \supset L^p (X_j, \, \mu),$
while the equation
\[
m (\ch_{X_j}) \rh (e_{j, k}) \xi
  = \rh (e_{j, j}) \rh (e_{j, k}) \xi
  = \rh (e_{j, k}) \xi
\]
for $\xi \in L^p (X_k, \, \mu)$
implies
${\mathrm{ran}} (\rh (e_{j, k})) \subset L^p (X_j, \, \mu).$
This completes the proof of~(\ref{T:SpatialRepsMd-7}).

That (\ref{T:SpatialRepsMd-7}) implies~(\ref{T:SpatialRepsMd-6})
is immediate from Lemma~\ref{L-PiIsSp}.

We prove that
(\ref{T:SpatialRepsMd-6}) implies~(\ref{T:SpatialRepsMd-1}).
From~(\ref{T:SpatialRepsMd-6}),
we see that
${\mathrm{ran}} (\rh (e_{j, 1})) = L^p (X_j, \mu)$
for $j = 1, 2, \ldots, d.$
Moreover, $\rh (e_{1, 1}) = m ( \ch_{X_1})$
by Lemma~\ref{L-SPIdem}.
Now fix~$j,$
and consider $\rh (e_{1, j}).$
For $k \neq j$ we have
$\rh (e_{1, j}) |_{L^p (X_k, \mu)} = 0$ since $e_{1, j} e_{k, 1} = 0.$
Also $\rh (e_{1, j}) \rh (e_{j, 1}) = \rh (e_{1, 1}) = m ( \ch_{X_1}).$
So Lemma~\ref{L-BasicSPI}(\ref{L-BasicSPI-5})
implies that $\rh (e_{1, j})$ is the reverse of $\rh (e_{j, 1}).$
In particular,
$\rh (e_{1, j})$ is a spatial partial isometry.

For $j, k = 1, 2, \ldots, d,$
it now follows from
Lemma~\ref{L-CompSPI}
that $\rh (e_{j, k}) = \rh (e_{j, 1}) \rh (e_{1, k})$
is a spatial partial isometry,
which is~(\ref{T:SpatialRepsMd-1}).

We next prove that
(\ref{T:SpatialRepsMd-2}) and~(\ref{T:SpatialRepsMd-8})
are equivalent to the conditions we have already considered.
We start with
(\ref{T:SpatialRepsMd-7}) implies~(\ref{T:SpatialRepsMd-8}).
Let the partition $X = \coprod_{j = 1}^d X_j$
be as in~(\ref{T:SpatialRepsMd-7}).
Define $Y = X_1$ and $\nu = \mu |_{X_1}.$
Identify $L^p ( N_d \times Y, \, \gm \times \nu)$
with the space of sequences
$( \et_1, \et_2, \ldots, \et_d ) \in L^p (Y, \nu)^d,$
with the norm
\[
\| ( \et_1, \et_2, \ldots, \et_d ) \|
  = \big( \| \et_1 \|_p^p + \| \et_2 \|_p^p
       + \cdots + \| \et_d \|_p^p \big)^{1 / p}.
\]
Define
$u \colon L^p ( N_d \times Y, \, \gm \times \nu) \to L^p (X, \mu)$
by
\[
u ( \et_1, \et_2, \ldots, \et_d )
  = \et_1 + \rh (e_{2, 1}) \et_2 + \rh (e_{3, 1}) \et_3
     + \cdots + \rh (e_{d, 1}) \et_d.
\]
(Note that $\et_1 = \rh (e_{1, 1}) \et_1.$)
Then $u$ is isometric because the summands
$\rh (e_{j, 1}) \et_j$ are supported in disjoint subsets
of~$X.$
It is easy to check that $u$ is bijective and that
$\rh (e_{j, k}) = u (e_{j, k} \otimes 1) u^{-1}$
for $j, k = 1, 2, \ldots, d.$
This proves~(\ref{T:SpatialRepsMd-8}).

Now assume~(\ref{T:SpatialRepsMd-8});
we prove~(\ref{T:SpatialRepsMd-2}).
It is trivial that the standard representation of
$M_d$ on $L^p ( N_d )$ satisfies~(\ref{T:SpatialRepsMd-2}).
It follows from Lemma~\ref{L-TensorSPI}
that the \rpn\  $\rh_0 (a) = a \otimes 1$
also satisfies~(\ref{T:SpatialRepsMd-2}).
The operator $u$ is a spatial isometry by Lemma~\ref{L-PiIsSp},
and Lemma~\ref{L-BasicSPI}(\ref{L-BasicSPI-5}) implies that its
reverse is~$u^{-1}.$
Lemma~\ref{L-CompSPI} now implies that $\rh$
satisfies~(\ref{T:SpatialRepsMd-2}).

That (\ref{T:SpatialRepsMd-2}) implies~(\ref{T:SpatialRepsMd-1})
is trivial.

We finish by proving that
(\ref{T:SpatialRepsMd-3}) and~(\ref{T:SpatialRepsMd-4})
are equivalent to the conditions we have already considered.
That (\ref{T:SpatialRepsMd-8}) implies~(\ref{T:SpatialRepsMd-3})
follows from the norm relation
$\| a \otimes 1 \| = \| a \|.$
(See Theorem~\ref{T-LpTP}(\ref{T-LpTP-3b}).)
That (\ref{T:SpatialRepsMd-3}) implies~(\ref{T:SpatialRepsMd-4})
is trivial.

Assume~(\ref{T:SpatialRepsMd-4});
we prove~(\ref{T:SpatialRepsMd-5}).
For $j = 1, 2, \ldots, d$ and $\zt \in S_1,$
set
\[
t_{j, \zt} = 1 - e_{j, j} + \zt e_{j, j} \in \MP{d}{p}.
\]
One checks immediately that $\| t_{j, \zt} \| = 1,$
and that $t_{j, \zt}^{-1} = t_{j, \zt^{-1}}.$
Since $\rh$ is contractive,
it follows that
$\rh (t_{j, \zt})$ is a bijective isometry for all $j$ and~$\zt.$
So $\rh (t_{j, \zt})$ is spatial by Lemma~\ref{L-PiIsSp}.
Since $\rh (t_{j, \zt})$ is bijective,
its spatial system has the form
$(X, X, S_{j, \zt}, g_{j, \zt}).$
Clearly $S_{j, 1} = \id_{\cB}.$
So Lemma~\ref{L:LampertiHomotopy}
implies that $S_{j, \, - 1} = \id_{\cB}.$
Therefore $\rh (t_{j, \, - 1}) = m (g_{j, \, - 1}).$
It follows that there is a unital algebra \hm\  %
\[
\ph \colon l^{\I} \big( \{ 1, 2, \ldots, d \} \big) \to L^{\I} (X, \mu)
\]
such that $\ph ( \ch_{\{ j \} } ) = \tfrac{1}{2} (1 - g_{j, \, - 1})$
for $j = 1, 2, \ldots, d.$
So there is a partition $X = \coprod_{j = 1}^d X_j$
such that for $j = 1, 2, \ldots, d$ we have $g_{j, \, - 1} = \ch_{X_j}$
almost everywhere~$[\mu].$
Condition~(\ref{T:SpatialRepsMd-5}) now follows.
\end{proof}

We now turn to representations of $L_d$ and~$C_d.$
In Definition~\ref{D:KindsOfReps1},
we defined what it means for a \rpn\  to be \cog,
\fis\   (on the $s_j$),
and \sfi\  (in addition, the linear combinations of the $s_j$ are
sent to scalar multiples of isometries).
We now introduce several further properties of a \rpn.

\begin{dfn}\label{D:SpatialRep}
Let $A$ be any of $L_d$ (Definition~\ref{D:Leavitt}),
$C_d$ (Definition~\ref{D:Cohn}),
or~$L_{\I}$ (Definition~\ref{D:LInfty}).
Let $p \in [1, \I],$
let $\XBM$ be a \sfm,
and let $\rh \colon A \to \LLp$ be a \rpn.
\begin{enumerate}
\item\label{D:SpatialRep-Disj}
We say that $\rh$ is {\emph{disjoint}} if
there exist disjoint sets $X_1, X_2, \ldots, X_d \subset X$
(if $A = L_{\I},$ disjoint sets $X_1, X_2, \ldots \subset X$)
such that ${\mathrm{ran}} (\rh (s_j)) \subset L^p (X_j, \mu)$
for all~$j.$
\item\label{D:SpatialRep-Sp}
We say that $\rh$ is {\emph{spatial}} if
for each~$j,$
the operators $\rh (s_j)$ and $\rh (t_j)$ are
spatial partial isometries,
with $\rh (t_j)$ being the reverse of $\rh (s_j)$
in the sense of Definition~\ref{D-Reverse}.
\end{enumerate}
\end{dfn}

The definition of a spatial \rpn\  is quite strong.
For $p \neq 2, \I,$
we will see that \rpn s satisfying some much weaker conditions
are necessarily spatial.

The \rpn s of $L_d$ in Examples \ref{E:LdRepOnlp} and~\ref{E:01RepLd}
are spatial and disjoint.
The \rpn s in Examples \ref{E-Mult} and~\ref{E-SumMult}
are disjoint, but not spatial;
in fact, they are neither \cog\  nor \fis.
The \rpn s in Examples \ref{E-LdSkew}, \ref{E-LdDblSkew},
\ref{E-L3PartSkew}, and~\ref{E-LdGpSkew}
are \cog\  but not disjoint and not spatial.
The \rpn s of $L_{\I}$ in Examples
\ref{E:LIRepOnlp}, \ref{E:01RepI2}, and~\ref{E:01RepI1}
are spatial and disjoint.
The one in Example~\ref{E:01RepID} is disjoint,
but it is not spatial.
(For $j \in \N,$ the operator $\rh (s_j)$
is a spatial partial isometry, but $\rh (t_j)$ is not.)
The \rpn\  of $L_{\I}$ of
Example~\ref{E-LIOpenSkew}
is disjoint but not spatial,
since the images of neither the $s_j$ nor the $t_j$ are spatial.

\begin{lem}\label{R-InjSpRpn}
Let $A$ be any of $L_d,$
$C_d,$
or~$L_{\I}.$
Let $p \in [1, \I].$
Then there is an injective spatial \rpn\  of~$A$
on $l^p (\N).$
\end{lem}

\begin{proof}
For $L_d$ (including $d = \I$),
we can use any spatial \rpn,
say that of Example~\ref{E:LdRepOnlp} for $d < \I$
and that of Example~\ref{E:LIRepOnlp} for $d = \I,$
because $L_d$ is simple
(Theorem~2 of~\cite{Lv2} for $d < \I$
and Example 3.1(ii) of~\cite{AA2} for $d = \I$).
For $C_d,$
we use the \rpn~$\pi$ of Example~\ref{E:CdRepOnlp}.
We check injectivity.
With $\rh$ as in Example~\ref{E:LdRepOnlp}
(for $d + 1$ in place of~$d$)
and $\io_{d, \, d + 1}$ as in Lemma~\ref{L-CdLdPlus1},
we have $\pi = \rh \circ \io_{d, \, d + 1}.$
Moreover, $\io_{d, \, d + 1}$ is injective by Lemma~\ref{L-CdLdPlus1}
and we saw above that $\rh$ is injective.
\end{proof}

We give two further conditions,
also motivated by Equation~(\ref{Eq:2Std})
(before Definition~\ref{D:KindsOfReps1})
for the \ca s
and by the analogous equation for the adjoints.

\begin{dfn}\label{D:StdRep}
Let $A$ be any of $L_d$ (Definition~\ref{D:Leavitt}),
$C_d$ (Definition~\ref{D:Cohn})
or~$L_{\I}$ (Definition~\ref{D:LInfty}).
Let $E$ be a nonzero Banach space,
and let $\rh \colon A \to L (E)$ be a  \rpn.
Let $p \in [1, \I].$
\begin{enumerate}
\item\label{D:StdRep-s}
We say that $\rh$
is {\emph{$p$-standard}} on $\spn (s_1, s_2, \ldots, s_d)$
if, following Definition~\ref{N-LComb},
the map $\ld \mapsto \rh (s_{\ld})$
from $\C^d$ to $L (E)$
is isometric from $l^p \big( \{ 1, 2, \ldots, d \} \big)$
to $L (E).$
(In the case $A = L_{\I}$ and $p \neq \I,$
we say that $\rh$
is {\emph{$p$-standard}} on $\spn (s_1, s_2, \ldots)$
if the map $\ld \mapsto \rh (s_{\ld})$
from $\C^{\I}$ to $L (E)$
extends to an isometric map from $l^p (\N)$ to $L (E).$
For $p = \I,$ we use $C_0 (\N)$ in place of $l^{\I}.$)
\item\label{D:StdRep-t}
Let $q \in [1, \I]$ be the conjugate exponent,
that is, $\frac{1}{p} + \frac{1}{q} = 1.$
We say that $\rh$
is {\emph{$p$-standard}} on $\spn (t_1, t_2, \ldots, t_d)$
if, following Definition~\ref{N-LComb},
the map $\ld \mapsto \rh (t_{\ld})$
from $\C^d$ to $L (E)$
is isometric from $l^q \big( \{ 1, 2, \ldots, d \} \big)$
to $L (E).$
(In the case $A = L_{\I}$ and $p \neq 1,$
we say that $\rh$
is {\emph{$p$-standard}} on $\spn (t_1, t_2, \ldots)$
if the map $\ld \mapsto \rh (t_{\ld})$
from $\C^{\I}$ to $L (E)$
extends to an isometric map from $l^q (\N)$ to $L (E).$
For $p = 1,$ we use $C_0 (\N)$ in place of $l^{\I}.$)
\end{enumerate}
\end{dfn}

We will see in Theorem~\ref{T-SpatialRepsL} below that a spatial
representation of $L_d$ on a space of the form $L^p (X, \mu)$
is necessarily $p$-standard on both
$\spn (s_1, s_2, \ldots, s_d)$
and $\spn (t_1, t_2, \ldots, t_d).$
Example~\ref{E-SumMult} shows that
a \rpn\  which is $p$-standard on
$\spn (s_1, s_2, \ldots, s_d)$
need not be spatial, or even \cog.
If in that example one instead defines
\[
\pi (s_j) = \rh (s_j) \oplus_p 2 \rh (s_j)
\andeqn
\pi (t_j) = \rh (t_j) \oplus_p \tfrac{1}{2} \rh (t_j),
\]
the resulting \rpn\  is $p$-standard on $\spn (t_1, t_2, \ldots, t_d)$
but not spatial.
One could fix the difficulty with
Example~\ref{E-SumMult} by incorporating \sfi\  in
the definition of $p$-standard on
$\spn (s_1, s_2, \ldots, s_d),$
but we don't know what the analogous fix
for $p$-standard on $\spn (t_1, t_2, \ldots, t_d)$ should be.

\begin{thm}\label{T-SpatialRepsL}
Let $d \in \N,$
let $L_d$ be as in Definition~\ref{D:Leavitt},
let $p \in (1, \I) \SM \{ 2 \},$
let $\XBM$ be a \sfm,
and let $\rh \colon L_d \to \LLp$ be a \rpn.
Then \tfae:
\begin{enumerate}
\item\label{T-SpatialRepsL-N1}
$\rh$ is spatial.
\item\label{T-SpatialRepsL-N2}
For $j = 1, 2, \ldots, d,$
the operator $\rh (s_j)$ is a spatial partial isometry.
\item\label{T-SpatialRepsL-N3}
For $j = 1, 2, \ldots, d,$
the operator $\rh (t_j)$ is a spatial partial isometry.
\item\label{T-SpatialRepsL-N4}
$\rh$ is \fis\  and the restriction of $\rh$ to
$\spn \big( ( s_j t_k )_{j, k = 1}^{d} \big) \cong M_d$
(see Lemma~\ref{L:mSum})
is a spatial \rpn\  of~$M_d$
in the sense of Definition~\ref{D:M_dSpatialRep}.
\item\label{T-SpatialRepsL-N5}
$\rh$ is \cog\  and the restriction of $\rh$ to
$\spn \big( ( s_j t_k )_{j, k = 1}^{d} \big)$
is a spatial \rpn\  of $M_d.$
\item\label{T-SpatialRepsL-N6}
$\rh$ is \fis\  and disjoint.
\item\label{T-SpatialRepsL-N7}
$\rh$ is \cog\  and disjoint.
\item\label{T-SpatialRepsL-N8}
$\rh$ is \sfi\  and disjoint.
\item\label{T-SpatialRepsL-N9}
$\rh$ is $p$-standard on $\spn (s_1, s_2, \ldots, s_d)$
and is \sfi.
\item\label{T-SpatialRepsL-N10}
$\rh$ is $p$-standard on $\spn (t_1, t_2, \ldots, t_d)$
and the \rpn\  $\rh'$ of Lemma~\ref{L:TransposeRep} is \sfi.
\item\label{T-SpatialRepsL-Row}
$\big( \rh (s_1) \,\,\, \rh (s_2) \,\,\, \cdots \,\,\, \rh (s_d) \big)$
is a row isometry,
in the sense that,
using the notation of Remark~\ref{R-pSumRmk},
it defines an isometric linear map
\[
L^p (X, \mu) \oplus_p L^p (X, \mu) \oplus_p
   \cdots \oplus_p L^p (X, \mu)
\to L^p (X, \mu).
\]
\item\label{T-SpatialRepsL-N11}
$\rh$ is \fis\  and for $j = 1, 2, \ldots, d$ there is $X_j \subset X$
such that ${\mathrm{ran}} (\rh (s_j)) = L^p (X_j, \mu).$
\item\label{T-SpatialRepsL-N12}
$\rh$ is \cog\  and for $j = 1, 2, \ldots, d$ there is $X_j \subset X$
such that ${\mathrm{ran}} (\rh (s_j)) = L^p (X_j, \mu).$
\item\label{T-SpatialRepsL-N13}
The \rpn\  $\rh'$ of Lemma~\ref{L:TransposeRep} is spatial.
\end{enumerate}
For $p = 1,$
all the conditions except (\ref{T-SpatialRepsL-N3}),
(\ref{T-SpatialRepsL-N10}), and~(\ref{T-SpatialRepsL-N13})
are equivalent,
and the other conditions imply these three.
\end{thm}

Various remarks are in order.
First, when $p = 1,$
we do not know whether the conditions (\ref{T-SpatialRepsL-N3}),
(\ref{T-SpatialRepsL-N10}), and~(\ref{T-SpatialRepsL-N13})
imply the \rpn\  is spatial.
(The last two of these are the ones which for $p = 1$
involve \rpn s
on $L^{\I} (X, \mu).$)

Second, some of the equivalences fail for \rpn s of $L_{\I}$
and~$C_d.$
Assume $p \neq 1.$
Then the \rpn\  of $L_{\I}$
of Example~\ref{E:01RepID}
satisfies (\ref{T-SpatialRepsL-N2}), (\ref{T-SpatialRepsL-N6}),
(\ref{T-SpatialRepsL-N9}),
(\ref{T-SpatialRepsL-Row}),
and~(\ref{T-SpatialRepsL-N11}),
but not~(\ref{T-SpatialRepsL-N3})
and hence not~(\ref{T-SpatialRepsL-N1}).
The same is true for the restriction to~$C_d,$
using the map $\io_{d, \I}$ of Lemma~\ref{L-CdLdPlus1}.
For $p \in (1, \I) \SM \{ 2 \},$
the \rpn\  of $L_{\I}$ of
Example~\ref{E-LIOpenSkew}, and its dual,
both satisfy (\ref{T-SpatialRepsL-N6}),
(\ref{T-SpatialRepsL-N7}),
(\ref{T-SpatialRepsL-N8}),
(\ref{T-SpatialRepsL-N9}),
(\ref{T-SpatialRepsL-N10}),
and~(\ref{T-SpatialRepsL-Row}),
but none of (\ref{T-SpatialRepsL-N2}),
(\ref{T-SpatialRepsL-N3}),
(\ref{T-SpatialRepsL-N11}),
(\ref{T-SpatialRepsL-N12}).
or~(\ref{T-SpatialRepsL-N13}).
The same is true for the \rpn\  of $C_{d_0}$
of Example~\ref{E-CdOpenSkew}.

Third, various other implications one might hope for fail.
Example~\ref{E-LdSkew} shows that
contractive on generators does not imply that the
restriction to $M_d$ is spatial.
Example~\ref{E-Mult} shows that
the restriction to $M_d$ being spatial does not imply
that the whole \rpn\  is spatial.
If $\rh$ is $p$-standard on $\spn (s_1, s_2, \ldots, s_d),$
it does not follow that
$\rh$ is $p$-standard on $\spn (t_1, t_2, \ldots, t_d),$
or that $\rh'$ is $q$-standard on $\spn (s_1, s_2, \ldots, s_d),$
by Example~\ref{E-SumMult}.

A spatial \rpn\  of $L_d$ must be \sfi,
since that is part of
Theorem~\ref{T-SpatialRepsL}(\ref{T-SpatialRepsL-N8}).
The converse is not true;
Example~\ref{E-LdSkew} is a counterexample.

Next,
we recall from Theorem~\ref{T:SpatialRepsMd} that
conditions (\ref{T-SpatialRepsL-N4}) and~(\ref{T-SpatialRepsL-N5})
actually have many equivalent formulations.
Similarly,
condition~(\ref{T-SpatialRepsL-N13})
is equivalent to analogs on $L^q (X, \mu)$
of all the other conditions of Theorem~\ref{T-SpatialRepsL}.

If $p > 2$ then
$p$-standard on $\spn (s_1, s_2, \ldots, s_d)$
can be weakened to $\| \rh (s_{\ld}) \| \leq \| \ld \|_p$
in~(\ref{T-SpatialRepsL-N9}).
(See Lemma~\ref{L-IfpStd}(\ref{L-IfpStd>2}).)
If $p < 2,$
then
$p$-standard on $\spn (t_1, t_2, \ldots, t_d)$
can be weakened to $\| \rh (t_{\ld}) \| \leq \| \ld \|_q$
in~(\ref{T-SpatialRepsL-N10}).
(See Corollary~\ref{C-IfqStd}(\ref{C-IfqStd<2}).)

Finally, it is again essential that the representation be unital.
Many of the conditions in Theorem~\ref{T-SpatialRepsL}
actually do make sense for nonunital \rpn s,
but the following example shows that
they do not imply that the \rpn\  is spatial.

\begin{exa}\label{E-NonUnL}
Let the notation be as in Example~\ref{E-NonUMd},
except take $\rh_0$ to be a spatial \rpn\  of $L_d$
on $L^p (X_0, \mu_0).$
Set $\rh (a) = e \otimes \rh_0 (a)$
for $a \in L_d.$
Then $\rh$ is not spatial,
not even in a sense suitable for nonunital \rpn s.
However, it is disjoint, \sfi, \cog,
$p$-standard on $\spn (s_1, s_2, \ldots, s_d),$
and isometric (although not spatial)
on $\spn \big( ( s_j t_k )_{j, k = 1}^{d} \big).$
\end{exa}

It is convenient to break the proof of Theorem~\ref{T-SpatialRepsL}
into several lemmas.
Some of them hold in greater generality than needed.
The first, in effect,
shows that (\ref{T-SpatialRepsL-N1}) implies all the other
conditions.

\begin{lem}\label{L-CnsqOfSp}
Let $A$ be any of $L_d$ (Definition~\ref{D:Leavitt}),
$C_d$ (Definition~\ref{D:Cohn}),
or~$L_{\I}$ (Definition~\ref{D:LInfty}).
Let $p \in [1, \I],$
let $\XBM$ be a \sfm,
and let $\rh \colon L_d \to \LLp$ be a spatial \rpn.
Then:
\begin{enumerate}
\item\label{L-CnsqOfSp-1}
$\rh$ is \cog.
\item\label{L-CnsqOfSp-2}
$\rh$ is \sfi.
\item\label{L-CnsqOfSp-2b}
$\rh$ is disjoint.
\item\label{L-CnsqOfSp-3}
$\rh$ is $p$-standard on $\spn (s_1, s_2, \ldots, s_d)$
(Definition~\ref{D:StdRep}(\ref{D:StdRep-s})).
\item\label{L-CnsqOfSp-4}
$\rh$ is $p$-standard on $\spn (t_1, t_2, \ldots, t_d)$
(Definition~\ref{D:StdRep}(\ref{D:StdRep-t})).
\item\label{L-CnsqOfSp-Row}
$\big( \rh (s_1) \,\,\, \rh (s_2) \,\,\, \cdots \,\,\, \rh (s_d) \big)$
is a row isometry,
in the sense described in
Theorem~\ref{T-SpatialRepsL}(\ref{T-SpatialRepsL-Row}).
\item\label{L-CnsqOfSp-5}
For each $j$ there is $X_j \subset X$
such that ${\mathrm{ran}} (\rh (s_j)) = L^p (X_j, \mu).$
\item\label{L-CnsqOfSp-6}
If $A = L_d$ with $d \neq \I,$
the restriction of $\rh$ to
$\spn \big( ( s_j t_k )_{j, k = 1}^{d} \big) \cong M_d$
(see Lemma~\ref{L:mSum})
is a spatial \rpn\  of $M_d$
in the sense of Definition~\ref{D:M_dSpatialRep}.
\item\label{L-CnsqOfSp-7}
If $p \neq \I,$
the \rpn\  $\rh'$ of Lemma~\ref{L:TransposeRep} is spatial.
\end{enumerate}
\end{lem}

\begin{proof}
Part~(\ref{L-CnsqOfSp-1})
follows from Lemma~\ref{L-BasicWSPI}(\ref{L-BasicWSPI-2}),
part~(\ref{L-CnsqOfSp-5})
follows from Lemma~\ref{L-CndForSp},
and part~(\ref{L-CnsqOfSp-7})
follows from Lemma~\ref{L-DualSPI}.
Part~(\ref{L-CnsqOfSp-6})
follows from the fact (Lemma~\ref{L-CompSPI})
that the product of spatial partial isometries is again
a spatial partial isometry.

To prove part~(\ref{L-CnsqOfSp-2b}),
let $X_j$ be as in part~(\ref{L-CnsqOfSp-5}).
Suppose $\mu (X_j \cap X_k) \neq 0.$
Then $\rh (t_j) \rh (s_k)$ is a nonzero
spatial partial isometry by Lemma~\ref{L-CompSPI},
contradicting the relation
(\ref{Eq:Leavitt2}) or~(\ref{Eq:ILeavitt2}),
as appropriate, in the definition of~$A.$
Part~(\ref{L-CnsqOfSp-2b}) follows.

We prove (\ref{L-CnsqOfSp-2}), (\ref{L-CnsqOfSp-3}),
and~(\ref{L-CnsqOfSp-Row}) together.
Let
\begin{align*}
s & = \big( \rh (s_1) \,\,\, \rh (s_2) \,\,\,
                \cdots \,\,\, \rh (s_d) \big)
             \\
  & \in L \big( L^p (X, \mu) \oplus_p L^p (X, \mu) \oplus_p
          \cdots \oplus_p L^p (X, \mu), \,
       L^p (X, \mu) \big).
\end{align*}
Thus
\[
s (\xi_1, \xi_2, \ldots, \xi_d) = \sum_{j = 1}^d \rh (s_j) \xi_j
\]
for $\xi_1, \xi_2, \ldots, \xi_d \in L^p (X, \mu).$
We already proved that $\rh$ is disjoint,
so
\[
\rh (s_1) \xi_1, \, \rh (s_2) \xi_2, \, \ldots, \, \rh (s_d) \xi_d
\]
have disjoint supports.
Since $\| \rh (s_j) \xi_j \| = \| \xi_j \|$ by
Lemma~\ref{L-BasicWSPI}(\ref{L-BasicWSPI-5}),
it follows from Remark~\ref{R:LpSuppNorm} that
\begin{equation}\label{Eq:CnsqRow}
\| s (\xi_1, \xi_2, \ldots, \xi_d) \|_p^p
  = \sum_j \| \rh (s_j) \xi_j \|_p^p
  = \sum_j \| \xi \|_p^p
  = \| (\xi_1, \xi_2, \ldots, \xi_d) \|_p^p.
\end{equation}
This proves~(\ref{L-CnsqOfSp-Row}).
Now let $\ld \in \C^d$
and let $\xi \in L^p (X, \mu).$
In~(\ref{Eq:CnsqRow}), take $\xi_j = \ld_j \xi$
for $j = 1, 2, \ldots, d,$
getting
\[
\| \rh ( s_{\ld} ) \xi \|_p^p
  = \sum_j | \ld_j |^p \| \xi \|_p^p
  = \| \ld \|_p^p \| \xi \|_p^p.
\]
This equation implies
both (\ref{L-CnsqOfSp-2}) and~(\ref{L-CnsqOfSp-3}).

It remains to prove~(\ref{L-CnsqOfSp-4}).
Let $q \in [1, \I]$ satisfy $\frac{1}{p} + \frac{1}{q} = 1.$
For $\gm \in \C^{d}$ (possibly $d = \I,$
in which case we follow the convention of Definition~\ref{N-LComb}),
define $\om_{\gm} \colon \C^d \to \C$ by
$\om_{\gm} (\ld) = \sum_{j = 1}^d \gm_j \ld_j.$

We show that for $\gm \in \C^d,$
we have $\| \rh (t_{\gm}) \| \geq \| \gm \|_q.$
Let $\ep > 0,$
and choose (using the usual pairing between $l^p$ and $l^q$)
an element $\ld \in \C^d$ such that $\| \ld \|_p = 1$
and $| \om_{\gm} (\ld) | > \| \gm \|_q - \ep.$
By part~(\ref{L-CnsqOfSp-3}) (already proved),
we have $\| \rh (s_{\ld}) \| = 1.$
Using Lemma~\ref{L-LCombProd} at the third step,
we then get
\[
\| \rh (t_{\gm}) \|
 = \| \rh (t_{\gm}) \| \cdot \| \rh (s_{\ld}) \|
 \geq \| \rh (t_{\gm}) \rh (s_{\ld}) \|
 = \| \om_{\gm} (\ld) \cdot 1 \|
 > \| \gm \|_q - \ep.
\]
Since $\ep > 0$ is arbitrary,
the desired conclusion follows.

We now show the reverse.
Let $\gm \in \C^d$ and let $\xi \in L^p (X, \mu).$
Let the disjoint sets $X_j$ be as in~(\ref{L-CnsqOfSp-2b})
(already proved).
Set $\xi_j = \xi |_{X_j}.$
Since $\rh (t_j)$ is a spatial partial isometry
with domain support~$X_j,$
Lemma~\ref{L-BasicWSPI}(\ref{L-BasicWSPI-2}) gives
$\| \rh (t_j) \xi \|_p = \| \xi_j \|_p.$
Now,
using H\"{o}lder's inequality at the third step and
Remark~\ref{R:LpSuppNorm} at the last step,
we get
\begin{align*}
\| \rh (t_{\gm}) \xi \|_p
& \leq \sum_{j = 1}^d | \gm_j | \cdot \| \rh (t_j) \xi \|_p
     \\
& = \sum_{j = 1}^d | \gm_j | \cdot \| \xi_j \|_p
  \leq \| \gm \|_q \left( \sssum{j = 1}{d} \| \xi_j \|_p^p \right)^{1/p}
  = \| \gm \|_q \| \xi \|_p.
\end{align*}
So $\| \rh (t_{\gm}) \| \leq \| \gm \|_q.$
\end{proof}

\begin{lem}\label{L-IfpStd}
Let $A$ be any of $L_d,$
$C_d,$
or~$L_{\I}.$
Let $p \in [1, \I) \SM \{ 2 \},$
let $\XBM$ be a \sfm,
and let $\rh \colon L_d \to \LLp$ be a \rpn.
Adopt the notational conventions of Definition~\ref{D:StdRep}
in case $A = L_{\I}.$
\begin{enumerate}
\item\label{L-IfpStd>2}
Suppose $p > 2.$
If $\rh$ is \fis\  and for all $\ld \in \C^d$
we have $\| \rh (s_{\ld}) \| \leq \| \ld \|_p,$
then $\rh$ is disjoint.
\item\label{L-IfpStd<2}
Suppose $p < 2.$
If $\rh$ is \sfi\  %
and is $p$-standard on $\spn (s_1, s_2, \ldots, s_d),$
then $\rh$ is disjoint.
\end{enumerate}
\end{lem}

\begin{proof}
For each~$j,$
Theorem~\ref{T:Lamperti} implies that $\rh (s_j)$ is semispatial.
Let $S_j$ be its semispatial realization and let $X_j$ be its
range support.

Let $\xi \in L^p (X, \mu).$
We now claim that, under either set of hypotheses,
$\rh (s_j) \xi$ and $\rh (s_k) \xi$ have
essentially disjoint supports
for $j \neq k.$

Assume the hypotheses of~(\ref{L-IfpStd>2}).
Let $\dt_1, \dt_2, \ldots, \dt_d$ be the standard basis vectors
in~$\C^d.$
Use the hypotheses with the choices $\ld = \dt_j + \dt_k$
and $\ld = \dt_j - \dt_k$ at the first step,
and the fact that $\rh (s_j)$ and $\rh (s_k)$ are isometric
at the second step, to get
\begin{align}\label{Eq:IfpStd}
\| \rh (s_j) \xi + \rh (s_k) \xi \|_p^p
   + \| \rh (s_j) \xi - \rh (s_k) \xi \|_p^p
& \leq 2 \| \xi \|_p^p + 2 \| \xi \|_p^p
   \\
& = 2 \| \rh (s_j) \xi \|_p^p + 2 \| \rh (s_k) \xi \|_p^p.
\notag
\end{align}
By Corollary~2.1 of~\cite{Lp},
we must have in fact
\[
\| \rh (s_j) \xi + \rh (s_k) \xi \|_p^p
   + \| \rh (s_j) \xi - \rh (s_k) \xi \|_p^p
 = 2 \| \rh (s_j) \xi \|_p^p + 2 \| \rh (s_k) \xi \|_p^p,
\]
and it then follows,
by the condition for equality in Corollary~2.1 of~\cite{Lp},
that $\rh (s_j) \xi$ and $\rh (s_k) \xi$ have
essentially disjoint supports.

Now assume instead the hypotheses of~(\ref{L-IfpStd<2}).
In this case,
\[
\| \rh (s_j) \xi + \rh (s_k) \xi \|_p^p
   + \| \rh (s_j) \xi - \rh (s_k) \xi \|_p^p
 = \| \dt_j + \dt_k \|_p^p \| \xi \|_p^p
   + \| \dt_j - \dt_k \|_p^p \| \xi \|_p^p,
\]
so we have equality at the first step in~(\ref{Eq:IfpStd}).
The condition for equality in Corollary~2.1 of~\cite{Lp}
therefore implies that $\rh (s_j) \xi$ and $\rh (s_k) \xi$ have
essentially disjoint supports.
The claim is proved.

Since $\mu$ is \sft,
there is $\xi \in L^p (X, \mu)$
such that $\xi (x) > 0$ for all $x \in X.$
It follows from the definition of a semispatial partial isometry,
Proposition~\ref{P:SStar}(\ref{P:SStar-5}),
and Lemma~\ref{L-PrelimComp}(\ref{L-PrelimComp-0})
that $\rh (s_j) \xi$ is nonzero almost everywhere on~$X_j,$
and similarly
that $\rh (s_k) \xi$ is nonzero almost everywhere on~$X_k.$
Therefore $X_j$ and $X_k$ are essentially disjoint.
The conclusion of the lemma follows.
\end{proof}

\begin{cor}\label{C-IfqStd}
Let $A$ be any of $L_d,$
$C_d,$
or~$L_{\I}.$
Let $p \in (1, \I) \SM \{ 2 \},$
let $\XBM$ be a \sfm,
and let $\rh \colon L_d \to \LLp$ be a \rpn.
Adopt the notational conventions of Definition~\ref{D:StdRep}
in case $A = L_{\I}.$
Let $q \in (1, \I) \SM \{ 2 \}$
satisfy $\frac{1}{p} + \frac{1}{q} = 1,$
and let $\rh' \colon A \to L (L^q (X, \mu))$
be as in Lemma~\ref{L:TransposeRep}.
\begin{enumerate}
\item\label{C-IfqStd>2}
Suppose $p > 2.$
If $\rh$ is $p$-standard on $\spn (t_1, t_2, \ldots, t_d)$
and $\rh'$ is \sfi,
then $\rh'$ is disjoint.
\item\label{C-IfqStd<2}
Suppose $p < 2.$
If for all $\ld \in \C^d$
we have $\| \rh (t_{\ld}) \| \leq \| \ld \|_q,$
and $\rh'$ is \fis,
then $\rh'$ is disjoint.
\end{enumerate}
\end{cor}

\begin{proof}
Under the hypotheses of~(\ref{C-IfqStd>2}),
we have $q < 2.$
Also, by duality,
for $\ld \in \C^d$ we have
\[
\| \rh' (s_{\ld}) \| = \| \rh (t_{\ld}) \| = \| \ld \|_q.
\]
Therefore $\rh'$ satisfies the hypotheses of
Lemma~\ref{L-IfpStd}(\ref{L-IfpStd<2}), with $q$ in place of~$p,$
so is disjoint.

A similar argument shows that if $\rh$
satisfies the hypotheses of~(\ref{C-IfqStd<2}),
then $\rh'$ is disjoint by
Lemma~\ref{L-IfpStd}(\ref{L-IfpStd>2}).
\end{proof}

\begin{lem}\label{L-FISAndDisjImpSp}
Let $d \in \N,$
let $p \in [1, \I) \SM \{ 2 \},$
let $\XBM$ be a \sfm,
and let $\rh \colon L_d \to \LLp$ be a \rpn\  %
which is disjoint
(Definition~\ref{D:SpatialRep}(\ref{D:SpatialRep-Disj}))
and \fis\  %
(Definition~\ref{D:KindsOfReps1}(\ref{D:Repn-FI})).
Then $\rh$ is spatial
(Definition~\ref{D:SpatialRep}(\ref{D:SpatialRep-Sp})).
\end{lem}

\begin{proof}
Let $X_1, X_2, \ldots, X_d \subset X$ be the sets of
Definition~\ref{D:SpatialRep}(\ref{D:SpatialRep-Disj}),
so that
${\mathrm{ran}} (\rh (s_j)) \subset L^p (X_j, \mu)$
for all~$j.$
Since $\sum_{j = 1}^d \rh (s_j) \rh (t_j) = 1,$
the closed linear span of the ranges of the $\rh (s_j)$
is all of $L^p (X, \mu),$
so (up to a set of measure zero, which we may ignore)
$\bigcup_{j = 1}^d X_j = X$
and ${\mathrm{ran}} (\rh (s_j)) = L^p (X_j, \mu).$
It follows from Lemma~\ref{L-PiIsSp}
that $\rh (s_j)$ is a spatial isometry for $j = 1, 2, \ldots, d.$
Using Lemma~\ref{L-BasicSPI}(\ref{L-BasicSPI-4}),
Definition~\ref{D-Reverse},
and $\coprod_{j = 1}^d X_j = X,$
we see that there is a \rpn\  $\sm \colon L_d \to \LLp$
such that, for $j = 1, 2, \ldots, d,$
we have $\sm (s_j) = \rh (s_j)$ and $\sm (t_j)$ is
the reverse of $\rh (s_j).$
Clearly $\sm$ is spatial.
Lemma~\ref{L:LdUniq} implies that $\rh = \sm.$
\end{proof}

\begin{proof}[Proof of Theorem~\ref{T-SpatialRepsL}]
We begin by observing that (\ref{T-SpatialRepsL-N1}) implies
all the other conditions.
For (\ref{T-SpatialRepsL-N2}) and~(\ref{T-SpatialRepsL-N3}),
this is trivial.
For all the others except~(\ref{T-SpatialRepsL-N10}),
this follows from the various conclusions
of Lemma~\ref{L-CnsqOfSp}.
To get~(\ref{T-SpatialRepsL-N10}),
we must also use Lemma~\ref{L-CnsqOfSp}(\ref{L-CnsqOfSp-7})
to see that $\rh'$ is spatial,
and then apply Lemma~\ref{L-CnsqOfSp}(\ref{L-CnsqOfSp-2}) to~$\rh'.$

We now prove that all the other conditions
imply~(\ref{T-SpatialRepsL-N1})
(omitting (\ref{T-SpatialRepsL-N3}), (\ref{T-SpatialRepsL-N10}),
and~(\ref{T-SpatialRepsL-N13}) when $p = 1$),
mostly via one of (\ref{T-SpatialRepsL-N2}),
(\ref{T-SpatialRepsL-N6}), or (\ref{T-SpatialRepsL-N13}).
That (\ref{T-SpatialRepsL-N6}) implies~(\ref{T-SpatialRepsL-N1})
is Lemma~\ref{L-FISAndDisjImpSp}.
To see that (\ref{T-SpatialRepsL-N13})
implies~(\ref{T-SpatialRepsL-N1})
for $p \neq 1,$
we apply Lemma~\ref{L-CnsqOfSp}(\ref{L-CnsqOfSp-7}) to~$\rh'$
and use $(\rh')' = \rh.$
We prove that (\ref{T-SpatialRepsL-N2})
implies~(\ref{T-SpatialRepsL-N6})
(and hence also implies~(\ref{T-SpatialRepsL-N1})).
For $j = 1, 2, \ldots, d,$
let $X_j \subset X$ be the range support of $\rh (s_j).$
Then ${\mathrm{ran}} (\rh (s_j)) = L^p (X_j, \mu)$
because $\rh (s_j)$ is spatial.
Therefore ${\mathrm{ran}} (\rh (s_j t_j)) = L^p (X_j, \mu).$
If now $j \neq k,$
then $\rh (s_j t_j)$ and $\rh (s_k t_k)$
are idempotents whose product is zero,
so $L^p (X_j, \mu) \cap L^p (X_k, \mu) = \{ 0 \}.$
Disjointness of $\rh$ follows,
and the other past of~(\ref{T-SpatialRepsL-N6}) is immediate.

For $p \neq 1,$ we prove that
(\ref{T-SpatialRepsL-N3}) implies~(\ref{T-SpatialRepsL-N13}).
It follow from Lemma~\ref{L-DualSPI}
that $\rh'$ satisfies~(\ref{T-SpatialRepsL-N2}).
We have already proved that
(\ref{T-SpatialRepsL-N2})
implies~(\ref{T-SpatialRepsL-N1}),
so we conclude that $\rh'$ is spatial,
which is~(\ref{T-SpatialRepsL-N13}).

To prove that (\ref{T-SpatialRepsL-N4})
and~(\ref{T-SpatialRepsL-N5})
imply~(\ref{T-SpatialRepsL-N6}),
we use the implication from
(\ref{T:SpatialRepsMd-1}) to (\ref{T:SpatialRepsMd-5})
in Theorem~\ref{T:SpatialRepsMd}
to conclude that $\rh$ is disjoint.
If we start with~(\ref{T-SpatialRepsL-N5}),
we also use Remark~\ref{R-cogImpfis}.

That (\ref{T-SpatialRepsL-N7})
implies~(\ref{T-SpatialRepsL-N6})
is Remark~\ref{R-cogImpfis},
and that (\ref{T-SpatialRepsL-N8})
implies~(\ref{T-SpatialRepsL-N6}) is clear.

That (\ref{T-SpatialRepsL-N9})
implies~(\ref{T-SpatialRepsL-N6})
follows from Lemma~\ref{L-IfpStd}.
We prove that if $p \neq 1$ then (\ref{T-SpatialRepsL-N10})
implies~(\ref{T-SpatialRepsL-N13}).
It follows from Corollary~\ref{C-IfqStd}
that $\rh'$ is disjoint.
We have already proved that
(\ref{T-SpatialRepsL-N6})
implies~(\ref{T-SpatialRepsL-N1}),
so we conclude that $\rh'$ is spatial,
which is~(\ref{T-SpatialRepsL-N13}).

We now prove that (\ref{T-SpatialRepsL-Row})
implies~(\ref{T-SpatialRepsL-N9}).
Since we already proved that (\ref{T-SpatialRepsL-N9})
implies~(\ref{T-SpatialRepsL-N6}),
this will show that (\ref{T-SpatialRepsL-Row})
implies~(\ref{T-SpatialRepsL-N6}).
Let $\ld \in \C^d$ and let $\xi \in L^p (X, \mu).$
Using the assumption at the second step, we get
\begin{align*}
\| \rh (s_{\ld} ) \xi \|_p
& = \left\| \big( \rh (s_1) \,\,\, \rh (s_2) \,\,\,
                \cdots \,\,\, \rh (s_d) \big)
            \big( \ld_1 \xi, \, \ld_2 \xi, \, \ldots, \, \ld_d \xi \big)
                        \right\|_p
             \\
& = \left\| \big( \ld_1 \xi, \, \ld_2 \xi, \, \ldots, \, \ld_d \xi \big)
                        \right\|_p
  = \| \ld \|_p \| \xi \|_p.
\end{align*}
This equation implies both parts of condition~(\ref{T-SpatialRepsL-N9}).

That (\ref{T-SpatialRepsL-N11})
implies~(\ref{T-SpatialRepsL-N2})
is immediate from Lemma~\ref{L-PiIsSp}.
To see that (\ref{T-SpatialRepsL-N12})
implies~(\ref{T-SpatialRepsL-N2}),
use in addition Remark~\ref{R-cogImpfis}.
\end{proof}

\section{Spatial representations give isometric
 algebras}\label{Sec:SpatialIsSame}

\indent
In this section, we prove that for $d < \I,$
any two spatial \rpn s of $L_d$
(in the sense of Definition~\ref{D:SpatialRep}(\ref{D:SpatialRep-Sp}))
on $L^p (X, \mu),$
for the same value of~$p,$
give isometrically isomorphic Banach algebras.
The main technical tools are the notion of a free \rpn\  %
(Definition~\ref{D:FreeRep})
and the spatial realizations of spatial isometries.

\begin{dfn}\label{D:FreeRep}
Let $A$ be any of $L_d$ (Definition~\ref{D:Leavitt}),
$C_d$ (Definition~\ref{D:Cohn}),
or~$L_{\I}$ (Definition~\ref{D:LInfty}).
Let $\XBM$ be a \sfm,
let $p \in [1, \I],$
and let $\rh \colon A \to \LLp$
be a \rpn.
\begin{enumerate}
\item\label{D:FreeRep-Free}
We say that $\rh$ is {\emph{free}} if there is a partition
$X = \coprod_{m \in \Z} E_m$
such that for all $m \in \Z$ and all~$j,$
we have
\[
\rh (s_j) ( L^p (E_m, \mu)) \subset L^p (E_{m + 1}, \mu)
\andeqn
\rh (t_j) ( L^p (E_m, \mu)) \subset L^p (E_{m - 1}, \mu).
\]
\item\label{D:FreeRep-Approx}
We say that $\rh$ is {\emph{approximately free}} if for every $N \in \N$
there is $n \geq N$
and a partition
$X = \coprod_{m = 0}^{n - 1} E_m$
such that for $m = 0, 1, \ldots, n - 1$ and all~$j,$
taking $E_n = E_0$ and $E_{-1} = E_{n - 1},$
we have
\[
\rh (s_j) ( L^p (E_m, \mu)) \subset L^p (E_{m + 1}, \mu)
\andeqn
\rh (t_j) ( L^p (E_m, \mu)) \subset L^p (E_{m - 1}, \mu).
\]
\end{enumerate}
When dealing with approximately free \rpn s,
we always take the index $m$ in $E_m$ mod~$n.$
\end{dfn}

To produce free \rpn s,
we follow Lemmas \ref{L:MultLd}, \ref{L:MultCd}, and~\ref{L:MultLI},
producing \rpn s of the form $(\rh ( \cdot ) \otimes 1)^{1 \otimes u}$
for suitable~$u.$
To simplify the notation and avoid conflict,
we will abbreviate this \rpn\  to~$\rh_u.$

\begin{lem}\label{L:TensU}
Let $A$ be any of $L_d$ (Definition~\ref{D:Leavitt}),
$C_d$ (Definition~\ref{D:Cohn}),
or~$L_{\I}$ (Definition~\ref{D:LInfty}).
Let $p \in [1, \I).$
Let $\XBM$ and $\YCN$ be \sfm s.
Let $\rh \colon A \to \LLp$ be a \rpn,
and let $u \in L (L^p (Y, \nu))$ be invertible.
Then there exists a unique \rpn\  %
$\rh_u \colon A \to L \big( L^p (X \times Y, \, \mu \times \nu) \big)$
such that for all $j$ we have
(using the notation of Theorem~\ref{T-LpTP}(\ref{T-LpTP-3a}))
\[
\rh_u (s_j) = \rh (s_j) \otimes u
\andeqn
\rh_u (t_j) = \rh (t_j) \otimes u^{-1}.
\]
This construction has the following properties:
\begin{enumerate}
\item\label{L:TensU-0}
If $a \in A$ is homogeneous of degree~$k$
(with respect to the $\Z$-grading of Proposition~\ref{P:ZGrading}),
then $\rh_u (a) = \rh (a) \otimes u^k.$
\item\label{L:TensU-1}
If $u$ is isometric, $p \neq 2,$ and $\rh$ is spatial
(Definition~\ref{D:SpatialRep}(\ref{D:SpatialRep-Sp})),
then $\rh_u$ is spatial.
\item\label{L:TensU-2}
If there is a partition $Y = \coprod_{m \in \Z} F_m$
such that $u ( L^p (F_m, \nu)) = L^p (F_{m + 1}, \nu)$
for all $m \in \Z,$
then $\rh_u$ is free
in the sense of Definition~\ref{D:FreeRep}(\ref{D:FreeRep-Free}).
\end{enumerate}
\end{lem}

\begin{proof}
Existence and uniqueness of $\rh_u$ follow from
Lemma~\ref{L-TPRep},
combined with the appropriate one of
Lemmas \ref{L:MultLd}, \ref{L:MultCd}, and~\ref{L:MultLI},
taking $v = u^{-1}$ if $A = C_d$ or~$L_{\I}.$

It suffices to prove part~(\ref{L:TensU-0})
when $a$ is a product of generators $s_j$ and $t_j.$
By Lemma~\ref{L:PropOfWords}(\ref{L:PropOfWords-2}),
we may assume that there are words $\af, \bt \in W_{\I}^d$
such that $a = s_{\af} t_{\bt}.$
Then $\deg (a) = l (\af) - l (\bt)$
(by Lemma~\ref{L:PropOfWords}(\ref{L:PropOfWords-1})),
and one checks directly that
\[
\rh_u (s_{\af} t_{\bt})
  = \rh (s_{\af} t_{\bt}) \otimes u^{l (\af) - l (\bt)}.
\]

We prove~(\ref{L:TensU-1}).
Lemma~\ref{L-PiIsSp}
implies that $u$ is spatial,
and Lemma~\ref{L-BasicSPI}(\ref{L-BasicSPI-5})
implies that $u^{-1}$ is the reverse of~$u.$
Now apply Lemma~\ref{L-TensorSPI}
to conclude that $\rh (s_j \otimes u)$ is spatial
with reverse $\rh (t_j) \otimes u^{-1}.$

To prove~(\ref{L:TensU-2}),
we use the partition
\[
X \times Y = \coprod_{m \in \Z} X \times Y_m.
\]
The sets $X \times Y_m$ clearly have the required properties.
\end{proof}

The construction of Lemma~\ref{L:TensU}
preserves various other properties of \rpn s,
but we will not need this information.

\begin{prp}\label{P:FreeBigger}
Let $A$ be any of $L_d$ (Definition~\ref{D:Leavitt}),
$C_d$ (Definition~\ref{D:Cohn}),
or~$L_{\I}$ (Definition~\ref{D:LInfty}).
Let $p \in [1, \I),$
let $\XBM$ be a \sfm,
and let $\rh \colon A \to \LLp$ be a \rpn.
Let $u \in L (l^p (\Z))$ be the bilateral shift,
$(u \et) (m) = \et (m - 1)$ for $\et \in l^p (\Z).$
Let $\rh_u$ be as in Lemma~\ref{L:TensU}.
Then for every $a \in A$ we have $\| \rh_u (a) \| \geq \| \rh (a) \|.$
\end{prp}

\begin{proof}
Let $a \in A.$
Let $\ep > 0$;
we show that
$\| \rh_u (a) \| \geq \| \rh (a) \| - \ep.$

Recall the $\Z$-grading of Proposition~\ref{P:ZGrading}.
Choose $N_0 \in \N$ such that
there are homogeneous elements
\[
a_{- N_0}, \, a_{- N_0 + 1}, \, \ldots, a_{N_0 - 1}, a_{N_0} \in A
\]
such that $\deg (a_k) = k$ for all~$k$ and
$a = \sum_{k = - N_0}^{N_0} a_k.$
Choose $\zt \in L^p (X, \mu)$
such that
\[
\| \zt \|_p = 1
\andeqn
\| \rh (a) \zt \| > \| \rh (a) \| - \tfrac{1}{2} \ep.
\]
Set $r = \| \rh (a) \zt \|.$
If $r \leq \tfrac{1}{2} \ep,$ then $\| \rh (a) \| < \ep,$
and we are done.
Otherwise, choose $N \in \N$ such that
\[
N > \frac{N_0 r^p}{r^p - \left( r - \tfrac{1}{2} \ep \right)^p}.
\]

Let $\nu$ be counting measure on~$\Z.$
We identify $L^p (X \times \Z, \, \mu \times \nu)$
with the space of all sequences
$\xi = (\xi_m)_{m \in \Z}$ with $\xi_m  \in L^p (X, \mu)$ for
all~$m$ and such that
$\sum_{m \in \Z} \| \xi_m \|_p^p < \I,$
with
\[
\| \xi \|_p = \left( \sum_{m \in \Z} \| \xi_m \|_p^p \right)^{1/p}.
\]

Now define $\xi \in L^p (X \times \Z, \, \mu \times \nu)$ by
\[
\xi_m = \begin{cases}
  0 & | n | > N
        \\
  (2 N + 1)^{-1/p} \zt & | n | \leq N.
\end{cases}
\]
Then $\| \xi \|_p^p = 1.$

Set $\et = \rh_u (a) \xi.$
We have $\big[ (1 \otimes u) \xi \big]_m = \xi_{m - 1}$
for all $m \in \Z.$
Using Lemma~\ref{L:TensU}(\ref{L:TensU-0}),
for all $m \in \Z$ with $- N + N_0 \leq m \leq N - N_0$ we get
\begin{align*}
\et_m
 & = \sum_{k = - N_0}^{N_0} (2 N + 1)^{-1/p} \rh (a_k) \xi_{m - k}
      \\
 & = (2 N + 1)^{-1/p} \sum_{k = - N_0}^{N_0} \rh (a_k) \zt
   = (2 N + 1)^{-1/p} \rh (a) \zt.
\end{align*}
There are $2 N - 2 N_0 + 1$ such values of~$m.$
It follows that
\begin{align*}
\| \et \|_p^p
 & \geq (2 N - 2 N_0 + 1) (2 N + 1)^{-1} \| \rh (a_k) \zt \|_p^p
   = \left( \frac{2 N - 2 N_0 + 1}{2 N + 1} \right) r^p
 \\
 & > \left( 1 - \frac{N_0}{N} \right) r^p
   > \left( 1 - \frac{r^p - \left( r
                    - \tfrac{1}{2} \ep \right)^p}{r^p} \right) r^p
 = \left( r - \tfrac{1}{2} \ep \right)^p.
\end{align*}
So
\[
\| \et \| > r - \tfrac{1}{2} \ep
          = \| \rh (a) \zt \| - \tfrac{1}{2} \ep
          > \| \rh (a) \| - \ep.
\]
This shows that $\| \rh_u (a) \| > \| \rh (a) \| - \ep.$
\end{proof}

\begin{cor}\label{C:FreeRepsExist}
Let $A$ be any of $L_d,$ $C_d,$ or~$L_{\I}.$
Let $p \in [1, \I) \SM \{ 2 \}.$
Then there exists
an injective free spatial \rpn\  of~$A$ on $l^p (\N).$
\end{cor}

\begin{proof}
By Lemma~\ref{R-InjSpRpn},
there is an injective spatial \rpn~$\rh$ of~$A$ on $l^p (\N).$
Let $\rh_u$ be
the \rpn\  of~$A$
on $l^p (\N) \otimes_p l^p (\Z) \cong l^p (\N)$
of Proposition~\ref{P:FreeBigger}.
The inequality $\| \rh_u (a) \| \geq \| \rh (a) \|$
for all $a \in A$ implies that $\rh_u$ is also injective.
Moreover,
$\rh_u$ is spatial by Lemma~\ref{L:TensU}(\ref{L:TensU-1})
and free by Lemma~\ref{L:TensU}(\ref{L:TensU-2}).
\end{proof}

We want to prove an inequality in the opposite direction from
that of Proposition~\ref{P:FreeBigger}.
We need a lemma.

\begin{lem}\label{L:Combinatorial}
Let $\XBM$ be a \sfm, with $\mu \neq 0.$
Let $d \in \{ 2, 3, 4, \ldots, \I \},$
let $X_1, X_2, \ldots, X_d \subset X$
($X_1, X_2, \ldots \subset X$ if $d = \I$)
be disjoint \mb\  sets,
and for each~$j$ let $S_j$ be an injective \mst\  %
(Definition~\ref{D:SetTrans})
from $\XBM$ to $\big( X_j, \cB |_{X_j}, \mu |_{X_j} \big).$
Then for every $n \in \N$ there exists $E \in \cB$
with $\mu (E) \neq 0$ such that the elements
\[
S_{\af (1)} \circ S_{\af (2)} \circ \cdots \circ S_{\af (m)} ([E])
  \in \cB / {\mathcal{N}} (\mu),
\]
for all $m \in \{ 0, 1, 2, \ldots, d \}$
and all words $\af = (\af (1), \af (2), \ldots \af (m)) \in W_m^d$
(see Notation~\ref{N:Words})
are disjoint in the sense of~Definition~\ref{D:BooleanOrder}.
\end{lem}

\begin{proof}
In this proof, we will write expressions
like $S_j (E)$ for \mb\  subsets $E \subset X,$
meaning that $S_j (E)$ is taken to be some \mb\  subset of~$X$
whose image in $\cB / {\mathcal{N}} (\mu)$ is $S_j ([E])$
in the sense of Definition~\ref{D:Quot}.
We remember that such a set is only defined up to sets of measure
zero.
Also, disjointness and containment of subsets of~$X$
will only be up to sets of measure zero.
No problem will arise,
because we only deal with countably many subsets of~$X,$
and we can therefore make the conclusion exact at the end
by deleting a set of measure zero.

By analogy with Notation~\ref{N:WordsInGens},
for a word $\af = (\af (1), \af (2), \ldots, \af (m)) \in W_m^d,$
we define
\[
S_{\af} = S_{\af (1)} \circ S_{\af (2)} \circ \cdots \circ S_{\af (m)},
\]
which is a \mst\  from $\XBM$ to a suitable
subset of~$X.$
We take $S_{\E} = \id_{\cB / {\mathcal{N}} (\mu)}.$
As with products of generators of $L_d,$
if $\af, \bt \in W_{\I}^d$ and $\af \bt$ is their concatenation,
then $S_{\af} \circ S_{\bt} = S_{\af \bt}.$

We first claim that for all $m \in \N$ and all
$\af, \bt \in W_m^d,$
we have $S_{\af} (X) \cap S_{\bt} (X) = \E.$
To see this,
let $k$ be the least integer such that $\af (k) \neq \bt (k).$
Set
\[
\af_0 = (\af (k), \, \af (k + 1), \, \ldots, \af (m)),
\,\,\,\,\,\,
\bt_0 = (\bt (k), \, \bt (k + 1), \, \ldots, \bt (m)),
\]
and
\[
\gm = (\af (1), \af (2), \ldots, \af (k - 1)).
\]
Then $\gm \af_0 = \af$ and $\gm \bt_0 = \bt.$
The sets $S_{\af_0} (X)$ and $S_{\bt_0} (X)$ are disjoint
because they are contained
in the disjoint sets $X_{\af (k)}$ and $X_{\bt (k)}.$
It now follows from
Lemma~\ref{L:ShmProp}(\ref{L:ShmProp-2})
that $(S_{\gm} \circ S_{\af_0}) (X)$
and $(S_{\gm} \circ S_{\bt_0}) (X)$
are disjoint,
which implies the claim.

Our second claim is that if $D \subset X$ and $n \in \N$
satisfy $\mu (D) > 0$ and $D \cap S_1^n (D) = \E,$
then there exists a subset $F \subset D$
such that $\mu (F) \neq 0$ and such that for all $m \in \N$
such that $m$ divides~$n,$
we have $F \cap S_1^m (F) = \E.$
To prove the claim,
first observe that if for some fixed $m$ we have
$F \cap S_1^m (F) = \E,$
then for every subset $G \subset F$
we also have $G \cap S_1^m (G) = \E.$
Thus, if we prove the claim for just one divisor $m$ of~$n,$
an induction argument will yield the claim as stated.

Define $F = D \SM (D \cap S_1^m (D)).$
Clearly $F \cap S_1^m (F) = \E.$
We need only show that $\mu (F) > 0.$
Suppose not.
Then (as usual, up to a set of measure zero) we have
$D \subset S_1^m (D).$
By induction, we get $D \subset S_1^{k m} (D)$ for all~$k \in \N.$
In particular,
$D \subset S_1^{n} (D),$
which contradicts $D \cap S_1^n (D) = \E$ and $\mu (D) > 0.$
The claim is proved.

Our third claim is that for all $n \in \N,$
we have $\mu \big( S^{n} (X) \SM S^{2 n} (X) \big) \neq 0.$
Indeed,
$X \SM S_1^n (X)$ contains $X \SM S_1 (X),$
which contains $S_2 (X),$
and $\mu (S_2 (X)) > 0$ because $S_2$
is an injective \mst.
Since $S_1^n$ is an injective \mst,
it follows that $\mu \big( S_1^n (X \SM S_1^n (X) \big) \neq 0.$
This proves the claim.

Set $N = n!.$
Define $E_0 = S_1^{N} (X) \SM S_1^{2 N} (X).$
Then $\mu (E_0) > 0$ by the third claim,
and also $E_0 \cap S_1^N (E_0) = \E.$
The second claim therefore provides a subset $E \subset E_0$
such that $\mu (E) \neq 0$ and such that
$E \cap S_1^m (E) = \E$ for $m = 1, 2, \ldots, n.$

We show that $E$ satisfies the conclusion of the lemma.
So let $\af$ and $\bt$ be distinct words with length at most~$n.$
We have to prove that $S_{\af} (E) \cap S_{\bt} (E) = \E.$
\Wolog\  $l (\af) \geq l (\bt).$

Suppose $l (\af) = l (\bt).$
Then, using the first claim at the second step,
\[
S_{\af} (E) \cap S_{\bt} (E)
  \subset S_{\af} (X) \cap S_{\bt} (X)
  = \E.
\]

Suppose now $l (\af) > l (\bt).$
Set $m = l (\af)$ and $r = l (\bt).$
Define a new word $\gm$ with $l (\gm) = m$ by
\[
\gm = (\bt (1), \bt (2), \ldots, \bt (r), 1, 1, \ldots, 1).
\]
Thus $S_{\gm} = S_{\bt} \circ S_1^{m - r}.$

We consider two cases, the first of which is $\gm \neq \af.$
Then, since $m - r \leq n \leq N,$
\[
S_{\af} (E) \subset S_{\af} (X)
\andeqn
S_{\bt} (E) \subset S_{\bt} (S_1^N (X)) \subset S_{\gm} (X).
\]
So $S_{\af} (E) \cap S_{\bt} (E) = \E$ by the first claim.

It remains to consider the case $\gm = \af.$
Then
\[
S_{\af} (E) \cap S_{\bt} (E)
  = S_{\bt} (S_1^{m - r} (E)) \cap S_{\bt} (E).
\]
Since $S_1^{m - r} (E) \cap E = \E,$
it follows from
Lemma~\ref{L:ShmProp}(\ref{L:ShmProp-2})
that $S_{\af} (E) \cap S_{\bt} (E) = \E.$
\end{proof}

\begin{prp}\label{P:FreeIsSmaller}
Let $d \in \{ 2, 3, 4, \ldots \},$
and let $\XBM$ and $\YCN$ be \sfm s.
Let $p \in [1, \I) \SM \{ 2 \}.$
Let $\rh \colon L_d \to \LLp$ be an approximately free spatial \rpn,
and let $\ph \colon L_d \to L (L^p (Y, \nu))$ be a spatial \rpn.
Then for all $a \in L_d,$
we have $\| \rh (a) \| \leq \| \ph (a) \|.$
\end{prp}

\begin{proof}
We adopt the same conventions with regard to
\mst s as described at the beginning of the
proof of Lemma~\ref{L:Combinatorial}.
In particular, set operations and relations are only up to
sets of measure zero.
Also, as there, for the \mst s $S_j$ and $R_j$ defined below,
we write
\[
S_{\af} = S_{\af (1)} \circ S_{\af (2)} \circ \cdots \circ S_{\af (m)},
\]
and define $R_{\af}$ similarly.

By definition,
the operators $\ph (s_j)$ are spatial isometries.
Therefore they have spatial systems
$(Y, Y_j, S_j, g_j),$
in which $S_j$ is a bijective \mst.
In particular, $\ph (s_j) (L^p (E, \mu)) = L^p (S_j (E), \mu).$
By Lemma~\ref{L-CnsqOfSp}(\ref{L-CnsqOfSp-2b}),
the sets $Y_j$ are disjoint.
Since $d < \I,$ we get $Y = \coprod_{j = 1}^d Y_j.$

Similarly,
the operators $\rh (s_j)$ have range supports,
say, $X_j,$ and $X = \coprod_{j = 1}^d X_j.$
Moreover, the spatial realization $R_j$ of $\rh (s_j)$
is a bijective \mst\  from $X$ to~$X_j.$

Let $a \in L_d.$
We want to prove that $\| \rh (a) \| \leq \| \ph (a) \|.$
By scaling,
\wolog\  $\| \rh (a) \| = 1.$
Apply Lemma~\ref{L:SameLength},
obtaining~$N_0 \in \Nz,$
a finite set $F_0 \subset W_{\I}^d,$
and numbers $\ld^{(0)}_{\af, \bt} \in \C$
for $\af \in F_0$ and $\bt \in W_{N_0}^d,$
such that
\[
a = \sum_{\af \in F_0} \sum_{\bt \in W_{N_0}^d}
        \ld^{(0)}_{\af, \bt} s_{\af} t_{\bt}.
\]
Set $N_1 = \max \big( \{ l (\af) \colon \af \in F_0 \} \big).$
Let $\ta$ be any fixed word of length~$N_0.$
Set $b = s_{\ta} a.$
Set $F = \{ \ta \af_0 \colon \af_0 \in F_0 \}.$
Thus, for all $\af \in F,$
we have $N_0 \leq l (\af) \leq N_0 + N_1.$
For $\af \in F$ and $\bt \in W_{N_0}^d,$
write $\af = \ta \af_0$ with $\af_0 \in F_0,$
and set $\ld_{\af, \bt} = \ld^{(0)}_{\af_0, \bt}.$
Then
\[
b = \sum_{\af \in F} \sum_{\bt \in W_{N_0}^d}
                 \ld_{\af, \bt} s_{\af} t_{\bt}.
\]

Since $\rh$ and $\ph$ are both \cog\  %
(by Lemma~\ref{L-CnsqOfSp}(\ref{L-CnsqOfSp-1})),
we have
\begin{align*}
\| \rh ( a) \|
& = \| \rh (t_{\ta}) \rh (s_{\ta} a) \|
  \leq \| \rh (t_{\ta}) \| \cdot \| \rh (s_{\ta} a) \|
     \\
& \leq \| \rh (b) \|
  \leq \| \rh (s_{\ta}) \| \cdot \| \rh ( a) \|
  \leq \| \rh (a) \|,
\end{align*}
so $\| \rh (b) \| = \| \rh (a) \| = 1,$
and similarly $\| \ph (b) \| = \| \ph (a) \|.$
It therefore suffices to prove that
$\| \rh (b) \| \leq \| \ph (b) \|.$

Let $\ep > 0.$
We prove that $\| \ph (b) \| > 1 - \ep.$
If $N_0 = N_1 = 0,$
then $b$ is a scalar,
and this inequality is immediate.
Otherwise, $N_0 + N_1 > 0.$
Choose $r \in \N$ such that
\[
r > (N_0 + N_1) \left( \frac{2}{\ep} \right)^p.
\]
Choose $\xi^{(0)} \in L^p (X, \mu)$ such that
\[
\big\| \xi^{(0)} \big\|_p = 1
\andeqn
\big\| \rh (b) \xi^{(0)} \big\|_p > 1 - \tfrac{1}{2} \ep.
\]
Since $\rh$ is approximately free,
there is $N \geq (N_0 + N_1) r$
and a partition
$X = \coprod_{m = 0}^{N - 1} D_m$
such that for $m = 0, 1, \ldots, N - 1$ and all~$j,$
taking $D_N = D_0$ and $D_{-1} = D_{N - 1},$
we have
\[
\rh (s_j) ( L^p ( D_m, \mu ) ) \subset L^p ( D_{m + 1}, \, \mu )
\andeqn
\rh (t_j) ( L^p ( D_m, \mu ) ) \subset L^p ( D_{m - 1}, \, \mu ).
\]
Write
\[
\xi^{(0)} = \sum_{m = 0}^{N - 1} \xi^{(0)}_m
\]
with $\xi^{(0)}_m \in L^p ( D_m, \mu )$
for $m = 0, 1, \ldots, N - 1.$

We claim that there is a set $T$
of $N_0 + N_1$ consecutive values of~$m$
such that
\[
\left\| \ssum{m \in T} \xi^{(0)}_m \right\|_p < \frac{\ep}{2}.
\]
Since the sets $D_m$ are disjoint,
Remark~\ref{R:LpSuppNorm} gives
\[
\left\| \ssum{m \in T} \xi^{(0)}_m \right\|_p^p
  = \sum_{m \in T} \big\| \xi^{(0)}_m \big\|_p^p.
\]
It is therefore enough to prove that there is a set $T$
of $N_0 + N_1$ consecutive values of~$m$
such that for all $n \in T$ we have
\[
\big\| \xi^{(0)}_m \big\|_p
 < \frac{\ep}{2} \left( \frac{1}{N_0 + N_1} \right)^{1/p}.
\]

Suppose that there is no such set~$T.$
Then, in particular,
for $k = 0, 1, \ldots, r - 1$ there is
\[
n_k \in \big[ k (N_0 + N_1),
              \,  (k + 1) (N_0 + N_1)  \big) \cap \Z
\]
such that
\[
\big\| \xi^{(0)}_{n_k} \big\|_p \geq
   \frac{\ep}{2} \left( \frac{1}{N_0 + N_1} \right)^{1/p}.
\]
Then
\[
\big\| \xi^{(0)} \big\|_p^p
  \geq \sum_{k = 0}^{r - 1} \big\| \xi^{(0)}_{n_k} \big\|_p^p
  \geq r \left( \frac{\ep}{2} \right)^p
        \left( \frac{1}{N_0 + N_1} \right)
  > 1,
\]
contradicting $\big\| \xi^{(0)}_{n_k} \big\|_p = 1.$
This proves the claim.

By cyclically permuting the indices of the sets~$D_m,$
we may assume that
\[
\left\| \sssum{m = 0}{N_0 - 1} \xi^{(0)}_m
  + \sssum{m = N - N_1}{N - 1} \xi^{(0)}_m \right\| < \frac{\ep}{2}.
\]
Define
\[
\xi_m = \begin{cases}
  \xi^{(0)}_m & {\mbox{$N_0 \leq m \leq N - N_1 - 1$}}
        \\
  0 & {\mbox{$0 \leq m \leq N_0 - 1$ and
                    $N - N_1 \leq m \leq N - 1$}}
\end{cases}
\]
and
\[
\xi = \sum_{m = 0}^{N - 1} \xi_m
    = \xi^{(0)} - \sum_{m = 0}^{N_0 - 1} \xi^{(0)}_m
      - \sum_{m = N - N_1}^{N - 1} \xi^{(0)}_m.
\]
Then $\big\| \xi - \xi^{(0)} \big\| < \tfrac{1}{2} \ep,$
so $\| \rh (b) \xi \| > 1 - \ep.$
Also clearly $\| \xi \| \leq \big\| \xi^{(0)} \big\| = 1.$

Following Lemma~\ref{L-TPRep},
except with the factors in the other order,
define a \rpn\  %
$\ps \colon L_d \to L \big( L^p (X \times Y, \, \mu \times \nu) \big)$
by $\ps (c) = 1 \otimes \ph (c)$ for all $c \in L_d.$
It follows from Lemma~\ref{L-TPRep} that
$\| \ps (b) \| = \| \ph (b) \|.$
We are now going to define an isometry (not necessarily surjective)
\[
u \colon L^p (X, \mu) \to L^p (X \times Y, \, \mu \times \nu)
\]
which will partially intertwine $\rh$ and~$\ps.$

The bijective \mst s $R_j$ at the
beginning of the proof preserve
disjointness and finite intersections and unions.
Since $X = \coprod_{j = 1}^d X_j,$
and identifying, as usual, sets with their images
in $\cB / {\mathcal{N}} (\mu),$
we get
\[
D_{m + 1} = \coprod_{j = 1}^{d} R_j (D_m)
\]
for $m = 0, 1, \ldots, N - 2.$
Define
\[
W = \bigcup_{m = 0}^{N - 1} W_m^d.
\]
Set $D_{\gm} = R_{\gm} (D_0)$ for $\gm \in W.$
Then
\[
X = \coprod_{\gm \in W} D_{\gm}.
\]
For $\gm \in W,$ define
$e_{\gm} = m (\ch_{D_{\gm}}) \in \LLp$
(following Notation~\ref{N:MultOps}).
Then the $e_{\gm}$ are idempotents of norm~$1$
and $\sum_{\gm \in W} e_{\gm} = 1.$

Apply Lemma~\ref{L:Combinatorial}
to the injective \mst s $S_j$
at the beginning of the proof,
obtaining a set $E \subset Y$ with $\nu (E) \neq 0$
such that the sets
$E_{\gm} = S_{\gm} (E),$ for $\gm \in W,$
are disjoint.
Then $\ph (s_{\gm}) (L^p (E, \nu)) = L^p ( E_{\gm}, \nu)$
for all~$\gm.$
Choose $\et_0 \in L^p (E, \nu)$
such that $\| \et_0 \|_p = 1.$

As in Theorem~\ref{T-LpTP},
identify $L^p (X \times Y, \, \mu \times \nu)$
with $L^p (X, \mu) \otimes_p L^p (Y, \nu).$
For any $\xi \in L^p (X, \mu)$
(not just the specific element $\xi$ considered above),
define
\[
u \xi = \sum_{\gm \in W}
     \rh (t_{\gm}) e_{\gm} \xi \otimes \ph (s_{\gm}) \et_0.
\]
Then
$u \in
 L \big( L^p (X, \mu), \, L^p (X \times Y, \, \mu \times \nu) \big).$

We claim that $u$ is isometric.
Let $\xi \in L^p (X, \mu)$ be arbitrary.
Since the sets $D_{\gm}$ are disjoint,
Remark~\ref{R:LpSuppNorm} gives
\[
\| \xi \|_p^p = \sum_{\gm \in W} \| e_{\gm} \xi \|_p^p.
\]
On the other hand,
for $\gm \in W,$
we have $L^p (D_{\gm}, \mu) \S {\mathrm{ran}} (\rh (s_{\gm})),$
so $\| \rh (t_{\gm}) e_{\gm} \xi \|_p = \| e_{\gm} \xi \|_p.$
Also,
the elements
$\rh (t_{\gm}) e_{\gm} \xi \otimes \ph (s_{\gm}) \et_0$
are supported in the disjoint sets
$X \times E_{\gm}$
(in fact, in $D_{\varnothing} \times E_{\gm}$),
so (using $\| \ph (s_{\gm}) \et_0 \|_p = \| \et_0 \|_p = 1$
at the third step)
\begin{align*}
\| u \xi \|_p^p
& = \left\| \ssum{\gm \in W}
     \rh (t_{\gm}) e_{\gm} \xi \otimes \ph (s_{\gm}) \et_0 \right\|_p^p
   \\
& = \sum_{\gm \in W}
        \| \rh (t_{\gm}) e_{\gm} \xi \|_p^p
            \cdot \| \ph (s_{\gm}) \et_0 \|_p^p
  = \sum_{\gm \in W} \| e_{\gm} \xi \|_p^p
  = \| \xi \|_p^p.
\end{align*}
This completes the proof that $u$ is isometric.

Now set
\[
G_0 = \bigcup_{m = 0}^{N_1} W_m^d
\andeqn
G = \bigcup_{m = N_0}^{N_0 + N_1} W_m^d
  = \big\{ \af_0 \af_1 \colon
        {\mbox{$\af_0 \in G_0$ and $\af_1 \in W_{N_0}^d$}} \big\}.
\]
Thus, $G$ is the set of all words with lengths from $N_0$
through $N_0 + N_1,$
and $F \S G.$
We can therefore write
\begin{equation}\label{Eq:bFormula}
b = \sum_{\af \in G} \sum_{\bt \in W_{N_0}^d}
   \ld_{\af, \bt} s_{\af} t_{\bt},
\end{equation}
by taking $\ld_{\af, \bt} = 0$ for $\af \in G \SM F.$

We claim that for any $\xi \in L^p (X, \mu)$
which is supported in
\[
\bigcup_{m = N_0}^{N - N_1 - 1} \bigcup_{\gm \in W_m^d} D_{\gm},
\]
and any $b \in L_d$ of the form~(\ref{Eq:bFormula})
(not just the particular $b$ used above),
we have
\begin{equation}\label{Eq:bInterTw}
u \rh (b) \xi = \ps (b) u \xi.
\end{equation}

By linearity, it suffices to prove
that for each $\gm \in W$
with $N_0 \leq l (\gm) \leq N - N_1 - 1,$
the claim holds for all $\xi$ which are supported in $D_{\gm}.$
Write $\gm = \gm_0 \gm_1$ with
\[
\gm_0 \in W_{N_0}^d
\andeqn
0 \leq l (\gm_1) \leq N - N_0 - N_1 - 1.
\]
Then
$\xi \in {\mathrm{ran}} (\rh (s_{\gm}))
    \S {\mathrm{ran}} (\rh (s_{\gm_0})),$
so 
$\rh (s_{\gm_0} ) \rh (t_{\gm_0}) \xi = \xi.$
For $\bt \in W_{N_0}^d,$
we have $t_{\bt} s_{\gm_0} = \dt_{\bt, \gm_0} \cdot 1,$
so
\[
\rh (b) \xi
 = \sum_{\af \in G} \sum_{\bt \in W_{N_0}^d}
   \ld_{\af, \bt}
    \rh ( s_{\af} ) \rh ( t_{\bt} ) \rh (s_{\gm_0} ) \rh (t_{\gm_0}) \xi
 = \sum_{\af \in G}
   \ld_{\af, \gm_0}
    \rh ( s_{\af} ) \rh (t_{\gm_0}) \xi.
\]
Since $\gm = \gm_0 \gm_1,$
the element $\rh (t_{\gm_0}) \xi$ is supported in $D_{\gm_1}.$
Let $\af \in G.$
Since
$l (\af) + l (\gm_1) \leq  N - 1,$
it follows that $\rh ( s_{\af} ) \rh (t_{\gm_0}) \xi$
is supported in $D_{\af \gm_1}.$
Therefore
(recalling that
Notation~\ref{N:WordsInGens} gives $t_{\af \gm_1} = t_{\gm_1} t_{\af}$),
\begin{align*}
u \rh ( s_{\af} ) \rh (t_{\gm_0}) \xi
& = \rh ( t_{\af \gm_1} ) \rh ( s_{\af} ) \rh (t_{\gm_0}) \xi
           \otimes \ph ( s_{\af \gm_1} ) \et_0
        \\
& = \rh ( t_{\gm_1} t_{\af} s_{\af} t_{\gm_0}) \xi
           \otimes \ph ( s_{\af \gm_1} ) \et_0
  = \rh ( t_{\gm}) \xi
           \otimes \ph ( s_{\af \gm_1} ) \et_0.
\end{align*}
Thus
\[
u \rh (b) \xi
 = \sum_{\af \in G}
   \ld_{\af, \gm_0}
    \rh ( t_{\gm}) \xi
           \otimes \ph ( s_{\af \gm_1} ) \et_0.
\]

On the other hand,
using $t_{\bt} s_{\gm_0} = \dt_{\bt, \gm_0}$ for $\bt \in W_{N_0}^d$
at the third step,
\begin{align*}
\ps (b) u \xi
& = (1 \otimes \ph (b))
     \big( \rh (t_{\gm}) \xi \otimes \ph (s_{\gm}) \et_0 \big)
   \\
& = \sum_{\af \in G} \sum_{\bt \in W_{N_0}^d}
   \ld_{\af, \bt}
     \rh (t_{\gm}) \xi
        \otimes \ph ( s_{\af} ) \ph ( t_{\bt} )
                \ph (s_{\gm_0} ) \ph (s_{\gm_1}) \et_0
   \\
& = \sum_{\af \in G}
   \ld_{\af, \gm_0}
     \rh (t_{\gm}) \xi
        \otimes \ph ( s_{\af} ) \ph (s_{\gm_1}) \et_0
  = u \rh (b) \xi.
\end{align*}
This completes the proof of the claim.

We now return to our specific choices of $\xi$ and~$b.$
They satisfy the hypotheses in the claim,
so we have, using~(\ref{Eq:bInterTw}) in the second calculation,
\[
\| u \xi \|_p = \| \xi \|_p \leq 1
\andeqn
\| \ps (b) u \xi \|_p
  = \| u \rh (b) \xi \|_p
  = \| \rh (b) \xi \|_p
  > 1 - \ep.
\]
Therefore
$\| \ph (b) \| = \| \ps (b) \| > 1 - \ep.$
This completes the proof.
\end{proof}

\begin{thm}\label{T:SpatialIsSame}
Let $d \in \{ 2, 3, 4, \ldots \},$
let $\XBM$ and $\YCN$ be \sfm s,
and let $\rh \colon L_d \to \LLp$
and $\ph \colon L_d \to L (L^p (Y, \nu))$ be spatial \rpn s
(Definition~\ref{D:SpatialRep}(\ref{D:SpatialRep-Sp})).
Then the map $\rh (s_j) \mapsto \ph (s_j)$
and $\rh (t_j) \mapsto \ph (t_j),$
for $j = 1, 2, \ldots, d,$
extends to an isometric isomorphism
${\ov{\rh (L_d)}} \to {\ov{\ph (L_d)}}.$
\end{thm}

\begin{proof}
The statement is symmetric in $\rh$ and~$\ph,$
so it suffices to prove that for all $a \in L_d,$
we have $\| \ph (a) \| \leq \| \rh (a) \|.$

Let $u \in L (l^p (\Z))$ be the bilateral shift,
and let
$\ph_u \colon L_d \to L \big( L^p (Y, \nu) \otimes l^p (\Z) \big)$
be as in Lemma~\ref{L:TensU}.
Proposition~\ref{P:FreeBigger} implies that
$\| \ph (a) \| \leq \| \ph_u (a) \|.$
Since $\ph_u$ is free
(by Lemma~\ref{L:TensU}(\ref{L:TensU-2})),
it is clearly essentially free.
Since $\ph_u$ is spatial (by Lemma~\ref{L:TensU}(\ref{L:TensU-1})),
Proposition~\ref{P:FreeIsSmaller} therefore implies that
$\| \ph_u (a) \| \leq \| \rh (a) \|.$
\end{proof}

Theorem~\ref{T:SpatialIsSame}
justifies the following definition.

\begin{dfn}\label{D:LPCuntzAlg}
Let $d \in \{ 2, 3, 4, \ldots \}$ and let $p \in [1, \I).$
We define $\OP{d}{p}$
to be the completion of $L_d$ in the norm
$a \mapsto \| \rh (a) \|$ for any spatial \rpn\  $\rh$
of $L_d$ on a space of the form $L^p (X, \mu)$
for a \sfm\  $\XBM.$
We write $s_j$ and $t_j$ for the elements in $\OP{d}{p}$
obtained as the images of the elements with the same names
in~$L_d,$
as in Definition~\ref{D:Leavitt}.
\end{dfn}

When $p = 2,$
we get the usual Cuntz algebra~${\mathcal{O}}_d.$
Indeed, if $\rh$ is a spatial \rpn\  %
on $L^2 (X, \mu),$
then, by Remark~\ref{R-RevAdj},
we have $\rh (t_j) = \rh (s_j)^*$ for $j = 1, 2, \ldots, d.$
Thus,
$\rh$ is a *-\rpn.

\begin{prp}\label{P-ContrImpIsom}
Let $d \in \{ 2, 3, 4, \ldots \},$ let $p \in [1, \I) \SM \{ 2 \},$
let $\XBM$ be a \sfm,
and let $\rh \colon \OP{d}{p} \to \LLp$
be a unital contractive \hm.
Then $\rh$ is isometric.
\end{prp}

\begin{proof}
Let $\rh_0$ be the composition of $\rh$ with the obvious map
$L_d \to \OP{d}{p}.$
By Theorem~\ref{T:SpatialIsSame},
it suffices to prove that $\rh_0$ is spatial.
We prove this by verifying
condition~(\ref{T-SpatialRepsL-N5}) of Theorem~\ref{T-SpatialRepsL}.
(By the last part of Theorem~\ref{T-SpatialRepsL},
this suffices when $p = 1$
as well as when $p \in (1, \I) \SM \{ 2 \}.$)
That $\rh_0$ is \cog\  is immediate.
Also, the obvious map sending
the standard matrix unit $e_{j, k} \in M_d$
to $s_j t_k \in \OP{d}{p}$
is isometric from $\MP{d}{p}$ to~$\OP{d}{p},$
by Lemma~\ref{L-CnsqOfSp}(\ref{L-CnsqOfSp-6})
and the equivalence of conditions
(\ref{T:SpatialRepsMd-1}) and~(\ref{T:SpatialRepsMd-3})
in Theorem~\ref{T:SpatialRepsMd}.
Therefore the restriction of $\rh_0$ to
$\spn \big( ( s_j t_k )_{j, k = 1}^{d} \big)$
is contractive as a map from $\MP{d}{p}$ to $\LLp.$
The equivalence of conditions
(\ref{T:SpatialRepsMd-1}) and~(\ref{T:SpatialRepsMd-4})
in Theorem~\ref{T:SpatialRepsMd}
therefore shows that this restriction is spatial.
This completes the verification of
condition~(\ref{T-SpatialRepsL-N5}) of Theorem~\ref{T-SpatialRepsL}.
\end{proof}

\begin{cor}\label{C-NonUContrImpIsom}
Let $d \in \{ 2, 3, 4, \ldots \},$ let $p \in [1, \I) \SM \{ 2 \},$
let $\XBM$ be a \sfm\  such that $L^p (X, \mu)$ is separable,
and let $\rh \colon \OP{d}{p} \to \LLp$
be a not necessarily unital contractive \hm.
Then $\rh$ is isometric.
\end{cor}

\begin{proof}
It is clear that $e = \rh (1)$ is an idempotent in $\LLp.$
Set $E = {\mathrm{ran}} (e).$
Then $\rh$ defines a contractive unital \hm\  %
from $\OP{d}{p}$ to $L (E).$

The hypotheses imply that $\| \rh (1) \| = 1.$
It follows from Theorem~3 in Section~17 of~\cite{Lc}
that there is a \msp\  $\YCN$
such that $E$ is isometrically isomorphic to $L^p (Y, \nu).$
Since $L^p (X, \mu)$ is separable, so is~$E,$
and therefore we may take $\nu$ to be \sft.
Now apply Proposition~\ref{P-ContrImpIsom}.
\end{proof}

Unfortunately,
unlike the case of \ca s ($p = 2$),
we know of no result which allows us to deduce
simplicity of $\OP{d}{p}$
from Proposition~\ref{P-ContrImpIsom}
or Corollary~\ref{C-NonUContrImpIsom},
or the other way around.
We will prove simplicity of $\OP{d}{p}$ in~\cite{Ph6},
using methods much closer to those used in the \ca\  case
in~\cite{Cu1}.

\section{Nonisomorphism of algebras generated by
  spatial representations for different~$p$}\label{Sec:NonIso}

\indent
To what extent do the various algebras $\OP{d}{p}$
differ from each other?
Taking $p = 2,$
the K-theory computation in~\cite{Cu2}
shows that, for $d_1 \neq d_2,$
we have $\OP{d_1}{2} \not\cong \OP{d_2}{2}.$
We will show in~\cite{Ph7}
that the K-theory is the same for $p \neq 2$ as for $p = 2,$
giving the analogous nonisomorphism result.
Here, we address what happens when one instead varies~$p.$
Here, K-theory is of no help.
Instead,
we show by more direct methods that for $p_1 \neq p_2$
and for $d_1$ and $d_2$ arbitrary,
there is no nonzero \ct\  \hm\  %
from $\OP{d_1}{p_1}$ to~$\OP{d_2}{p_2}.$
In fact,
there is no nonzero \ct\  \hm\  %
from $\OP{d_1}{p_1}$ to $L (l^{p_2} (\N)).$

This result gives a different proof of the fact
(Corollary~6.15 of~\cite{BP})
that for distinct $p_1, p_2 \in (1, \I),$
there is no nonzero \ct\  \hm\  %
from $L (l^{p_1} (\N))$ to $L (l^{p_2} (\N)).$
(We are grateful to Volker Runde for providing this reference.
In this connection, we note that
Corollary~2.18 of~\cite{BP}
implies that if $E$ and $F$ are Banach spaces
such that $L (E)$ is isomorphic to $L (F)$ as Banach algebras,
then $E$ is isomorphic to~$F$ as Banach spaces.)

\begin{lem}\label{L:MainNonIso}
Let $p \in [1, \I).$
Let $(X, {\mathcal{B}}, \mu)$
be a \sfm.
Let $\rh \colon L_{\I} \to L (l^{p} (X, \mu))$
be a spatial \rpn.
Let $E$ be a Banach space.
Suppose there is a nonzero \ct\  \hm\  %
$\ph \colon {\ov{\rh ( L_{\I} )}} \to L (E).$
Then $l^{p} (\N)$ is isomorphic as a Banach space
(recall the conventions in Definition~\ref{D:BSpTerm})
to a subspace of~$E.$
\end{lem}

\begin{proof}
The element $\ph (1) \in L (E)$ is an idempotent.
Replacing $E$ by $\ph (1) E,$
we may assume that $\ph$ is unital
(and $E$ is nonzero).

Let $q \in (1, \I]$ satisfy
$\frac{1}{p} + \frac{1}{q} = 1.$

Let $\ld \mapsto s_{\ld}$ and $\ld \mapsto t_{\ld}$
be as in Definition~\ref{N-LComb} for $L_{\I}.$
For $p \in (1, \I),$
it follows from Lemma~\ref{L-CnsqOfSp}(\ref{L-CnsqOfSp-3})
and Lemma~\ref{L-CnsqOfSp}(\ref{L-CnsqOfSp-4}) that the maps
$\ld \mapsto \rh (s_{\ld})$ and $\ld \mapsto \rh (t_{\ld})$
extend to isometric maps
$s^{\rh} \colon l^{p} (\N) \to L (L^{p} (X, \mu))$
and $t^{\rh} \colon l^{q} (\N) \to L (L^{p} (X, \mu)).$
For $p = 1,$
we similarly get $s^{\rh}$ as above and
$t^{\rh} \colon C_0 (\N) \to L (L^{p} (X, \mu)).$

Fix $\et_0 \in E$ with $\| \et_0 \| = 1.$
Define $v \colon l^{p} (\N) \to E$
as follows.
For $\ld = (\ld_1, \ld_2, \ldots) \in l^{p} (\N),$
set $v (\ld) = \ph ( s^{\rh} (\ld)) \et_0.$
Then $v$ is bounded,
because
\[
\| v (\ld) \|
 \leq \| \ph \| \cdot \| s^{\rh} (\ld) \| \cdot \| \et_0 \|
 = \| \ph \| \cdot \| \ld \|_{p}.
\]

We claim that for all $\ld \in l^{p} (\N),$
we have $\| v (\ld) \| \geq \| \ph \|^{-1} \| \ld \|_{p}.$
The claim will imply that $v$ is an isomorphism of $l^{p} (\N)$
with some closed subspace of~$E,$
completing the proof of the lemma.

We prove the claim.
Let $\ld \in l^{p} (\N).$
\Wolog\  $\ld \neq 0.$
First suppose $p \neq 1.$
It is well known
that there exists $\gm \in l^{q} (\N) \SM \{ 0 \}$
such that
\begin{equation}\label{Eq:EvTo1}
\sum_{j = 1}^{\I} \gm_j \ld_j = 1
\end{equation}
and
\begin{equation}\label{Eq:NormGm}
\| \gm \|_{q} = \| \ld \|_{p}^{-1}.
\end{equation}
(The method of proof can be found,
for example, at the beginning of Section~6.5 of~\cite{Ry}.)

By Lemma~\ref{L-LCombProd}, for every $n \in \N$ we have
\[
\left( \sssum{j = 1}{n} \gm_j \rh (t_j) \right)
    \left( \sssum{j = 1}{n} \ld_j \rh (s_j) \right)
  = \left( \sssum{j = 1}{n} \gm_j \ld_j \right) \cdot 1.
\]
Letting $n \to \I$ and applying~(\ref{Eq:EvTo1}),
we get $t^{\rh} (\gm) s^{\rh} (\ld) = 1.$
Therefore
\[
\et_0 = \ph (t^{\rh} (\gm)) \ph (s^{\rh} (\ld)) \et_0
  = \ph (t^{\rh} (\gm)) v (\ld).
\]
So, using~(\ref{Eq:NormGm}) at the third step,
\[
1 = \| \et_0 \|
  \leq \| \ph \| \cdot \| t^{\rh} (\gm) \| \cdot \| v (\ld) \|
  = \| \ph \| \cdot \| \ld \|_{p}^{-1} \cdot \| v (\ld) \|.
\]
It follows that
$\| v (\ld) \| \geq \| \ph \|^{-1} \| \ld \|_{p},$
as desired.

Now suppose $p = 1.$
Let $\ep > 0.$
Choose $\gm \in c_0 (\N)$
with finite support and such that
$\gm_j \ld_j \geq 0$ for $j \in \N,$
\[
\sum_{j = 1}^{\I} \gm_j \ld_j > 1 - \ep,
\andeqn
\| \gm \|_{q} = \| \ld \|_{p}^{-1}.
\]
Then $t^{\rh} (\gm) s^{\rh} (\ld) = \af \cdot 1$
with $\af > 1 - \ep.$
We get
\[
\af \et_0 = \ph (t^{\rh} (\gm)) v (\ld)
\andeqn
1 - \ep \leq \| \ph \| \cdot \| \ld \|_{p}^{-1} \cdot \| v (\ld) \|.
\]
Since $\ep > 0$ is arbitrary,
we again get
$\| v (\ld) \| \geq \| \ph \|^{-1} \| \ld \|_{p},$
as desired.
\end{proof}

\begin{thm}\label{T-NoNonZH}
Let $p_1, p_2 \in [1, \I)$ be distinct.
Let $A$ be any of $L_d$ (Definition~\ref{D:Leavitt}),
$C_d$ (Definition~\ref{D:Cohn}),
or~$L_{\I}$ (Definition~\ref{D:LInfty}).
Let $\rh \colon A \to L (l^{p_1} (\N))$ be a spatial \rpn.
Then there is no nonzero \ct\  \hm\  %
from ${\overline{\rh (A)}}$ to $L (l^{p_2} (\N)).$
\end{thm}

\begin{proof}
Suppose that $\ph \colon {\ov{\rh (A_1)}} \to L (l^{p_2} (\N))$
is a \ct\  \hm.

Regardless of what $A_1$ is,
there is a unital \hm\  $\ps \colon L_{\I} \to A_1$
such that (in the notation of Definition~\ref{D:LInfty})
\[
\ps \big( s_j^{(\infty)} \big) = s_2^j s_1
\andeqn
\ps \big( t_j^{(\infty)} \big) = t_1 t_2^j
\]
for all $j \in \N.$
Since $\rh$ is spatial,
one easily checks that $\rh \circ \ps$ is a spatial \rpn\  %
of $L_{\I}$ on $l^{p_1} (\N).$
Set $B = {\ov{(\rh \circ \ps) (L_{\I}) }}.$
Then Lemma~\ref{L:MainNonIso},
applied to $\ph |_{ B },$
provides an isomorphism of $l^{p_1} (\N)$
with a subspace of $l^{p_2} (\N).$
The remark after Proposition 2.a.2
of~\cite{LT1} (on page~54)
therefore implies that $p_1 = p_2.$
\end{proof}

\begin{cor}\label{T:NonIsoP}
Let $p_1, p_2 \in [1, \I)$ be distinct.
Let $A_1$ and $A_2$ be any two of $L_d$ (Definition~\ref{D:Leavitt}),
$C_d$ (Definition~\ref{D:Cohn}),
or~$L_{\I}$ (Definition~\ref{D:LInfty}).
Let $\rh_1 \colon A_1 \to L (l^{p_1} (\N))$
be a spatial \rpn\  %
(Definition~\ref{D:SpatialRep}(\ref{D:SpatialRep-Sp})),
and let $\rh_2 \colon A_2 \to L (l^{p_2} (\N))$
be an arbitrary \rpn.
Then there is no nonzero \ct\  \hm\  from ${\ov{\rh (A_1)}}$
to ${\ov{\rh (A_2)}}.$
\end{cor}

\begin{proof}
Combine Theorem~\ref{T-NoNonZH}
and Lemma~\ref{R-InjSpRpn}.
\end{proof}

In particular, there is no nonzero \ct\  \hm\  from $\OP{d_1}{p_1}$
to $\OP{d_2}{p_2}.$

We recover part of Corollary~6.15 of~\cite{BP}.

\begin{cor}\label{C-NoNZHLlp}
Let $p_1, p_2 \in [1, \I)$ be distinct.
Then there is no nonzero \ct\  \hm\  from $L (l^{p_1} (\N))$
to $L (l^{p_2} (\N)).$
\end{cor}

\begin{proof}
Suppose $\ph \colon L (l^{p_1} (\N)) \to L (l^{p_2} (\N))$
is a nonzero \ct\  \hm.
Use Lemma~\ref{R-InjSpRpn} to choose an injective
spatial \rpn\  $\rh \colon L_{\I} \to L (l^{p_1} (\N)).$
Then $1 \in {\ov{ \rh (L_{\I}) }},$
so $\ph |_{ {\ov{ \rh (L_{\I}) }} }$ is nonzero,
contradicting Theorem~\ref{T-NoNonZH}.
\end{proof}

Corollary~\ref{T:NonIsoP}
does not rule out isomorphism as Banach spaces.
In fact, we have the following result.

\begin{prp}\label{P:AntiIso}
Let $p \in [1, \I),$
and suppose $\frac{1}{p} + \frac{1}{q} = 1.$
Let $A$ be any of $L_d$ (Definition~\ref{D:Leavitt}),
$C_d$ (Definition~\ref{D:Cohn}),
or~$L_{\I}$ (Definition~\ref{D:LInfty}),
let $\XBM$ be a \sfm,
and let $\rh \colon A \to \LLp$
be any \rpn.
Then ${\ov{\rh (A)}}$ is isometrically antiisomorphic
to a subalgebra of $L (L^q (X, \mu)),$
namely ${\ov{\rh' (A)}}$ with $\rh'$ as in Lemma~\ref{L:TransposeRep}.
\end{prp}

\begin{proof}
Use Lemma~\ref{L:TransposeRep}
and the fact that (following Notation~\ref{N:Dual})
one always has $\| a' \| = \| a \|.$
\end{proof}

In particular,
when $\rh$ is spatial,
${\ov{\rh (A)}}$ and ${\ov{\rh' (A)}}$
are isometrically isomorphic as Banach spaces,
even though (for $p \neq 2$)
they are not even isomorphic as Banach algebras.

\end{document}